\documentclass[letterpaper,11pt]{article}
\usepackage[margin=1in]{geometry}
\usepackage[bookmarks, colorlinks=true, plainpages = false, colorlinks=true,
            citecolor=red,
            linkcolor=blue,
            anchorcolor=red,
            urlcolor=blue]{hyperref}
\usepackage{url}
\usepackage{amsmath,amsfonts,amsthm,amssymb,bm, verbatim,dsfont,mathtools}
\usepackage{color,graphicx,appendix}
\usepackage{etoolbox}

\usepackage{array}
\usepackage{multirow}
\usepackage{float}
\usepackage{tikz}
\usetikzlibrary{matrix,arrows,calc,shapes,backgrounds}
\usetikzlibrary{shapes.callouts,decorations.text} 

\usetikzlibrary{shapes.misc}

\tikzset{cross/.style={cross out, draw=black, minimum size=2*(#1-\pgflinewidth), inner sep=0pt, outer sep=0pt},
cross/.default={1pt}}

\tikzstyle{int}=[draw, fill=blue!20, minimum size=2em]
\tikzstyle{dot}=[circle, draw, fill=blue!20, minimum size=2em]
\tikzstyle{dotred}=[circle, draw, fill=red!20, minimum size=2em]
\tikzstyle{init} = [pin edge={to-,thin,black}]
\tikzstyle{initred} = [pin edge={to-,thin,red}]
\tikzstyle{plan}=[draw, fill=blue!20, minimum size=2em, text width=5em, rounded corners,align=center]
\tikzstyle{planwide}=[draw, fill=blue!20, minimum size=2em, text width=8em, rounded corners,align=center]
\tikzstyle{nodedot}=[circle, draw, fill=white, minimum size=0.3cm,inner sep=0pt]
\tikzstyle{nodedot}=[circle, draw, fill=white, minimum size=3,inner sep=0pt]
\tikzstyle{Medge}=[green!60!black, thick]
\tikzstyle{Bedge}=[red, thick]
\tikzstyle{Cedge}=[blue, thick]
\tikzstyle{Sedge}=[black, thick]
\tikzstyle{Mgiantedge}=[green!60!black, line width=3.0pt]
\tikzstyle{Bgiantedge}=[red,line width=3.0pt]
\tikzstyle{Cgiantedge}=[blue,line width=3.0pt]
\tikzstyle{Sgiantedge}=[black,line width=3.0pt]
\tikzstyle{shadedgiantnode}=[circle, draw, fill=black!10, minimum size=1cm, inner sep=0pt]
\tikzstyle{unshadedgiantnode}=[circle, draw, fill=white, minimum size=1cm, inner sep=0pt]
\makeatletter
\tikzset{my loop/.style =  {to path={
  \pgfextra{}
  [looseness=5,min distance=10mm]
  \tikz@to@curve@path},font=\sffamily\small
  }}  
\makeatletter 
\newcolumntype{C}[1]{>{\centering\arraybackslash}p{#1}}

\tikzstyle{vertexdot}=[circle, draw, fill=white, minimum size=3,inner sep=0pt]
\tikzstyle{vertexdotsolid}=[circle, draw, fill=black, minimum size=3,inner sep=0pt]

\newcommand{\PathThree}{\tikz[scale=0.5]{\draw (0,0) node (zero) [nodedot] {} -- (0.5,0) node (zero) [nodedot] {} -- (1.0,0) node (zero) [nodedot] {};}}

\newcommand{\PathFive}{\tikz[scale=0.75]{\draw (0,0) node (zero) [nodedot] {} -- (0.5,0) node (zero) [nodedot] {} -- (1.0,0) node (zero) [nodedot] {} -- (1.5,0) node (zero) [nodedot] {} -- (2,0) node (zero) [nodedot] {};}}

\newcommand{\Chair}{\tikz[scale=0.75,baseline = -0.5]{\draw (0,-0.5) node (zero) [vertexdot] {} -- (0,0) node[vertexdot] {} -- (0.5,0) node[vertexdot] {} -- (1,0) node[vertexdot] {}; \draw (0,0) node[vertexdot] {} -- (0,0.5) node[vertexdot] {} }}

\newcommand{\Star}{\tikz[scale=0.75,baseline= -0.5]{\draw (0,-0.5) node (zero) [vertexdot] {} -- (0,0) node[vertexdot] {} -- (0.5,0) node[vertexdot] {}; \draw (0.5,0) node[vertexdot] {} -- (-0.5,0) node[vertexdot] {} ;  
\draw (0,0) node[vertexdot] {} -- (0,0.5) node[vertexdot] {} 
}}

\renewcommand{\Diamond}{\tikz[scale=0.5,baseline=(zero.base)]{\draw (0.5,0)--(0,0.5); \draw (0,0) node (zero) [vertexdot] {} -- (0.5,0) node[vertexdot] {} -- (0.5,0.5) node[vertexdot] {} -- (0,0.5) node[vertexdot] {} -- cycle; }}

\usepackage{xr,xspace}
\usepackage{todonotes}
\usepackage{paralist}
\usepackage{enumitem}

\usepackage{caption,subcaption,soul}
\usepackage{algorithm}
\usepackage{algpseudocode}
\makeatletter
\usepackage{romanbar}

\theoremstyle{plain}
\newtheorem{theorem}{Theorem}
\newtheorem{lemma}{Lemma}
\newtheorem{proposition}{Proposition}

\theoremstyle{definition}

\newtheorem{problem}{Problem}

\newtheorem{remark}{Remark}

\newtheorem*{remark*}{Remark}

\newcommand{\floor}[1]{\left\lfloor #1 \right\rfloor}
\newcommand{\ceil}[1]{\left\lceil #1 \right\rceil}

\newcommand{\argmax}{\mathop{\arg\max}}

\newcommand{\Uniform}{\mathrm{Uniform}}

\usepackage{xspace,prettyref}

\newcommand\numberthis{\addtocounter{equation}{1}\tag{\theequation}}

\newcommand{\termI}{\text{(I)}}
\newcommand{\termII}{\text{(II)}}

\newcommand{\diverge}{\to\infty}

\newcommand{\iiddistr}{{\stackrel{\text{\iid}}{\sim}}}

\newcommand{\reals}{{\mathbb{R}}}

\newcommand{\naturals}{{\mathbb{N}}}

\newcommand{\red}{\color{red}}
\newcommand{\blue}{\color{blue}}
\newcommand{\green}{\color{green!60!black}}

\newcommand{\nb}[1]{{\sf\blue[#1]}}
\newcommand{\nbr}[1]{{\sf\red[#1]}}
\newcommand{\ssy}[1]{{\sf\green[#1]}}
\newcommand{\Expect}{\mathbb{E}}
\newcommand{\expect}[1]{\mathbb{E}\left[ #1 \right]} 

\newcommand{\expects}[2]{\mathbb{E}_{#2}\left[ #1 \right]}

\newcommand{\Prob}{\mathbb{P}}

\newcommand{\prob}[1]{ \mathbb{P}\left\{ #1 \right\} }

\newcommand{\Cov}{\mathrm{Cov}}

\def\Var{\mathrm{Var}}

\newcommand{\Bern}{{\rm Bern}}

\newcommand{\ie}{i.e.\xspace}
\newcommand{\iid}{i.i.d.\xspace}
\newrefformat{eq}{(\ref{#1})}
\newrefformat{chap}{Chapter~\ref{#1}}
\newrefformat{sec}{Section~\ref{#1}}
\newrefformat{alg}{Algorithm~\ref{#1}}
\newrefformat{fig}{Figure~\ref{#1}}
\newrefformat{tab}{Table~\ref{#1}}
\newrefformat{rmk}{Remark~\ref{#1}}
\newrefformat{clm}{Claim~\ref{#1}}
\newrefformat{def}{Definition~\ref{#1}}
\newrefformat{cor}{Corollary~\ref{#1}}
\newrefformat{lmm}{Lemma~\ref{#1}}
\newrefformat{prop}{Proposition~\ref{#1}}
\newrefformat{prob}{Problem~\ref{#1}}
\newrefformat{app}{Appendix~\ref{#1}}
\newrefformat{hyp}{Hypothesis~\ref{#1}}
\newrefformat{thm}{Theorem~\ref{#1}}

\newcommand{\pth}[1]{\left( #1 \right)}
\newcommand{\qth}[1]{\left[ #1 \right]}
\newcommand{\sth}[1]{\left\{ #1 \right\}}

\newcommand{\iprod}[2]{\left \langle #1, #2 \right\rangle}

\newcommand{\indc}[1]{{\mathbf{1}_{\left\{{#1}\right\}}}}

\newcommand{\sfC}{{\mathsf{C}}}

\newcommand{\sfS}{{\mathsf{S}}}
\newcommand{\sfT}{{\mathsf{T}}}

\newcommand{\calG}{{\mathcal{G}}}
\newcommand{\calH}{{\mathcal{H}}}

\newcommand{\calN}{{\mathcal{N}}}

\newcommand{\calP}{{\mathcal{P}}}
\newcommand{\calQ}{{\mathcal{Q}}}

\newcommand{\calS}{{\mathcal{S}}}
\newcommand{\calT}{{\mathcal{T}}}

\newcommand{\ER}{Erd\H{o}s--R\'enyi\xspace}

\renewcommand{\tilde}{\widetilde}
\renewcommand{\bar}{\overline}
\newcommand{\aut}{\mathsf{aut}}
\newcommand{\sub}{\mathsf{sub}}
\newcommand{\col}{\mathrm{col}}
\newcommand{\NumVer}{M}

\begin{document}




\title{Testing network correlation efficiently via counting trees
}

\author{Cheng Mao, Yihong Wu,  Jiaming Xu, and Sophie H.\ Yu\thanks{
C.\ Mao is with the School of Mathematics, Georgia Institute of Technology, Atlanta, Georgia, USA
\texttt{cheng.mao@math.gatech.edu}.
Y.\ Wu is  with the Department of Statistics and Data Science, Yale University, New Haven CT, USA, 
\texttt{yihong.wu@yale.edu}.
J.\ Xu and S.\ H.\ Yu are with The Fuqua School of Business, Duke University, Durham NC, USA, \texttt{\{jx77,haoyang.yu\}@duke.edu}.
C.~Mao is supported in part by the NSF Grant DMS-2053333. Y.~Wu is supported in part by the NSF Grant CCF-1527105, an NSF CAREER award CCF-1651588, and an Alfred Sloan fellowship. J. Xu is supported in part by the NSF Grant CCF-1856424
and an NSF CAREER award CCF-2144593. S. H.~Yu is supported by the NSF Grant CCF-1856424.
}}

\date{\today}
\maketitle

\begin{abstract}

We propose a new procedure for testing whether two networks are edge-correlated through some latent vertex correspondence. The test statistic is based on counting the co-occurrences of signed trees for a family of non-isomorphic trees.
When the two networks are \ER random graphs $\calG(n,q)$
 that are either independent or correlated with correlation coefficient $\rho$, 
our test runs in $n^{2+o(1)}$ time and succeeds with high probability as $n\to\infty$, provided that $n\min\{q,1-q\} \ge n^{-o(1)}$ and 
$\rho^2>\alpha \approx 0.338$, where 
$\alpha$ is Otter's constant so that the number of unlabeled trees with $K$ edges grows as $(1/\alpha)^K$. 
 This significantly  improves  the  prior work in  terms  of  statistical accuracy, running time, and graph  sparsity.


\end{abstract}

\tableofcontents

\section{Introduction}

In recent years, there is a surge of interest in studying the problem of graph matching or network alignment, which aims to find the latent vertex correspondence between the two graphs solely based on their network topologies. This paradigm arises in a suite of diverse applications, such as 
social network analysis \cite{narayanan2008robust,narayanan2009anonymizing}, computer vision \cite{cour2007balanced,berg2005shape}, computational biology \cite{cour2007balanced,berg2005shape}, and natural language processing \cite{haghighi2005robust, bayati2013message}.

Finding the optimal vertex correspondence that best aligns the two graphs amounts to solving the 
the NP-hard quadratic assignment problem (QAP). Aiming to circumvent its worst-case intractability,
a popular statistical model for graph matching is the \emph{correlated \ER graph model}, denoted by $\calG(n,q, \rho)$, in which the observed graphs are two instances of the \ER graph $\calG(n,q)$ whose edges are correlated through a hidden vertex correspondence.
Specifically, let $\pi$ be a latent uniform random permutation on $[n]\triangleq \{1,\ldots,n\}$.
Denote the observed graphs by $G_1$ and $G_2$ and their adjacency matrices by $A=(A_{ij})$ and $B=(B_{ij})$ respectively. 
Conditioned on the permutation $\pi$, the pairs of edges $\{(A_{ij}, B_{\pi(i)\pi(j)}): 1 \leq i<j\leq n\}$ are i.i.d.\ pairs of Bernoulli random variables with
parameter $q \in (0,1)$ and correlation coefficient $\rho$.\footnote{One can verify that the correlation coefficient $\rho$ between two Bernoulli random variables with parameter $q$ takes values in $[-\min\{ \frac{q}{1-q} , \frac{1-q}{q} \}, 1]$. While most of previous work focuses on the positively correlated case, in this paper we allow negative correlation.}
In the special case of $\rho>0$, $A$ and $B^\pi=(B_{\pi(i)\pi(j)})$
can be viewed as adjacency matrices of 
two children graphs that are 
independently edge-subsampled from a common parent \ER graph $\calG(n, p)$ with subsampling probability $s$, where $p=q/s$
and $s=\rho(1-q)+q$ \cite{pedarsani2011privacy}.
The goal is to recover the true vertex mapping $\pi$ based on $G_1$ and $G_2$. Under the correlated \ER graph model, the information-theoretic thresholds for both exact and partial recovery have been characterized \cite{cullina2016improved,cullina2017exact,Hall2020partial,wu2021settling} and various efficient matching algorithms with provable 
 performance guarantees have been devised \cite{feizi2016spectral,lyzinski2016graph,ding2018efficient,barak2019nearly,FMWX19a,FMWX19b,ganassali2020tree,ganassali2021correlation,mao2021random,mao2021exact}. 

\subsection{Detecting network correlation}

Despite the significant amount of research activities and remarkable progress in the graph matching problem, relatively less attention has been paid to the even more basic problem of detecting the presence of correlation in network topology between two otherwise independently generated graphs. This problem is practically important in many aforementioned application domains such as detecting similar $3$-D objects in computer vision or similar biological networks across different species. From a theoretical point of view, network correlation detection can be viewed as a natural extension of the classical problem of correlation detection for vector data
(testing the correlation between two random vectors under an unknown orthogonal transformation \cite{stephens1979vector})
to network data. In addition, the detection problem offers insights into the computational limits of graph matching; see~\prettyref{sec:hardness} for an in-depth discussion.

Following~\cite{barak2019nearly,wu2020testing}, we formulate the problem of detecting network correlation as a hypothesis testing problem, where
\begin{itemize}
\item 
Under the null hypothesis 
$\calH_0$, $G_1$ and $G_2$
are independently generated from the \ER graph model $\calG(n,q)$; 

\item
Under the alternative hypothesis,
$\calH_1$, $G_1$ and $G_2$ are generated from the correlated \ER graph model $\calG(n,q,\rho)$.
\end{itemize}
Note that under both $\calH_0$ and $\calH_1$, the graphs $G_1$ and $G_2$ are marginally distributed 
as $\calG(n,q)$. 
The goal is to distinguish $\calH_0$ from $\calH_1$ based on the observation of 
$G_1$ and $G_2$. 
We say a test statistic $f ( G_1, G_2)$ with threshold $\tau \in \reals$ achieves consistent detection 
if the sum of type I and type II errors converges to 0 as $n\diverge$, that is, 
\begin{align}
    \lim_{n\diverge} \left[\calQ(f(G_1,G_2)\geq \tau)
    +\calP(f(G_1,G_2)<\tau)\right] = 0, \label{eq:detection}
\end{align}
where $\calQ$ and $\calP$ denote the joint distribution of $G_1$ and $G_2$ under $\calH_0$
and $\calH_1$, respectively. 

\subsection{Subgraph counts}\label{sec:subgraph_count}

Note that due to the latent random vertex mapping $\pi$, the problem is equivalent to testing the correlation between two unlabeled graphs.
Thus, any test must rely on on \emph{graph invariants} ---
graph properties that are invariant under graph isomorphisms, 
such as subgraph counts or graph eigenvalues. 
This paper adopts a strategy based on subgraph counts and improves upon prior works \cite{barak2019nearly,wu2020testing}, which we now discuss.



In order to determine the information-theoretic limit, \cite{wu2020testing} considered the QAP test statistic, namely, the maximum number of common edges between $G_1$ and $G_2$ over all possible vertex correspondences, which by definition is invariant under isomorphisms of both graphs (Equivalently, this test can be viewed as finding the maximum common subgraph of $G_1$ and $G_2$).
The QAP test is shown to achieve the optimal detection threshold with sharp constant in the dense regime and within constant factors in the sparse regime. 
The drawback of this test is the computational intractability of solving the QAP.
In addition, a simple test based on comparing  the number of edges in $G_1$ and $G_2$ is  also analyzed in \cite{wu2020testing};
however, its statistical power is weak, requiring the correlation parameter $\rho$ to approach $1$ to achieve consistent detection. 

To obtain a more powerful yet computationally efficient statistic, a natural idea is to count 
 more complex subgraphs than edges. 
Specifically, let $\sub(H,G)$ denote the subgraph count, i.e., the number of copies, of $H$ in $G$. 
Crucially, for any given graph $H$, $\sub(H,G_1)$ and $\sub(H,G_2)$ are 
independent under $\calH_0$ but correlated under $\calH_1$. 
Thus, one can distinguish $\calH_0$ and $\calH_1$
by thresholding on the covariance of $\sub(H,G_1)$ and $\sub(H,G_2)$, that is,
\[
P_H \triangleq \left(\sub(H,G_1)-\expect{\sub(H,G_1)}\right) \left(\sub(H,G_2)-\expect{\sub(H,G_2)}\right) \, .
\]
However, counting a single graph $H$ may not suffice, especially when the graphs are sparse and the correlation is weak. To obtain a better test statistic, \cite{barak2019nearly} considers 
a large family $\calH$ of non-isomorphic subgraphs and further sum $P_H$ over $H \in \calH$. A key step in the analysis is to ensure that  $P_H$'s are approximately independent across different  $H$ so that the $\sum_{H \in \calH} P_H$ has a relatively small variance; this requires a careful choice of the collection $\calH$.

To this end, \cite{barak2019nearly} proposes to count a family of the so-called \emph{(strictly) balanced} graphs. A graph $H$
is
called (strictly) balanced if every proper subgraph of $H$
is (strictly) less dense than the graph $H$ itself, where the density is defined as the number of edges divided by the number of nodes. 
It is well known in random graph theory that strict balancedness ensures that the subgraph counts are well concentrated~\cite[Chapter~3]{janson2011random}. 
Specifically, 
\cite{barak2019nearly} gives a probabilistic construction of a family $\calH$ of strictly balanced graphs of a constant size $K$ and
a constant edge density $\eta>1$,
assuming 
$nq^\eta=\Theta(1)$.
Under this choice, it is shown that the test statistic $\sum_{H \in \calH} P_H$ distinguishes between the two hypotheses with probability at least $0.9$,
when $\rho = \Omega (1)$ and the average degree 
$nq$ belongs to $[n^{\epsilon},n^{1/153}] \cup [n^{2/3},n^{1-\epsilon}] $ for any small constant $\epsilon>0$.

While the test in~\cite{barak2019nearly} succeeds under small correlation $\rho$,
it still has several key limitations. First, the existing performance guarantee in~\cite{barak2019nearly} requires the average degree $nq$ to fall into a very specific range and in particular does not accommodate sparse graphs such as $nq = O(1)$; this restriction originates from the construction of strictly balanced graphs based on random regular graphs. Second, as an exhaustive search of subgraphs of size $K$ takes $n^{O(K)}$
time, in order to achieve a polynomial run time, \cite{barak2019nearly} only counts subgraphs of constant size $K=\Theta(1)$, which precludes the possibility of achieving a vanishing error probability. 
Third, to tolerate small correlation $\rho$, $K$ needs to be as large as $(1/\rho)^C$ for a large constant $C$, leading to a time complexity of $n^{(1/\rho)^C}$, which can be prohibitive for large networks. 

In this paper, we propose a new  test that  runs in time $n^{2+o(1)}$ and succeeds with high probability as long as $\rho^2$ exceeds $0.338$ and $n\min\{q,1-q\} \ge n^{-o(1)}$, i.e., the graphs are not overly sparse or dense. This significantly improves on
the previous results in~\cite{barak2019nearly,wu2020testing} in terms of  graph sparsity, testing error, and computational complexity. 

Our main strategy is to count \emph{trees} of $K$ edges in the observed graphs, where, crucially, $K$ grows with $n$. 
While by definition a tree with  $K$ edges is also strictly balanced with density $K/(K+1)$, our test statistic is fundamentally different from that in~\cite{barak2019nearly}. In particular, if we adopt their test statistic with $\calH$ being the collection of trees with $K$ edges, their existing result requires $nq^{K/(K+1)}=\Theta(1)$ and thus cannot deal with the most interesting case of $nq \ge 1$.
In fact, when $nq \ge 1$, for any two non-isomorphic trees $H \not \cong H'$, 
the covariance between $\sub(H,G_1)$ and $\sub(H',G_1)$ is on the same order as the variance of $\sub(H,G_1)$. 
As a consequence, the covariance between $P_{H}$ and $P_{H'}$ is relatively large compared to  $(\expect{P_H})^2$ under $\calH_1$. Hence, the standard deviation of $\sum_{H \in \calH} P_H$ is on par with its mean under $\calH_1$, leading to the likely failure of the test statistic in~\cite{barak2019nearly}.  
Therefore, a key challenge is how to better utilize the tree counts so that
 the correlation between the counts of non-isomorphic trees does not overwhelm the \emph{signal} --- the correlation of counts of isomorphic trees.

Furthermore, direct enumeration of trees of size $K$ takes $n^{O(K)}$ time which is super-polynomial if $K$ grows with $n$. 
To resolve this issue, we leverage the idea of \emph{color coding}~\cite{alon1995color,arvind2002approximation,hopkins2017bayesian} to develop an $n^2 e^{O(K)}$-time algorithm that approximates our test statistic and succeeds under the same condition.
When applied to the correlated \ER model, $K$ can be chosen to grow with $n$ arbitrarily slowly so that the test runs in $n^{2+o(1)}$ time.

\subsection{Notation and paper organization} \label{sec:graph_def_obs}

For any graph $H$, let $V(H)$ denote the vertex set of $H$ and $E(H)$ denote the edge set of $H$.  
		Two graphs $H$ and $H'$ are isomorphic, denoted by $H \cong H'$, if there exists a bijection $\pi:V(H)\to V(H')$ such that $(\pi(u),\pi(v)) \in E(H')$ if and only if $(u,v) \in E(H)$.
		Denote by $[H]$ the isomorphism class of $H$; it is customary to refer to these isomorphic classes as \emph{unlabeled graphs}.
    Let $\aut(H)$ be the number of automorphisms of $H$ (graph isomorphisms to itself). 
    We say $H$ is a subgraph of $G$, denoted by $H \subset G$,  if $V(H)\subset V(G)$ and $E(H) \subset E(G)$.
		We define $\sub(H,G)$ as 
    the number of subgraphs in $G$ that are isomorphic to $H$, i.e.,
   $   \sub(H,G)\triangleq  \sum_{H' \subset G} \indc{H' \cong H}$.
	For example, $\sub(\PathThree,\; \Diamond) = 8$.
     Denoting by $K_n$ the complete graph with vertex set $[n]$ and edge set $\binom{[n]}{2} \triangleq \{\{u,v\}: u,v\in[n], \, u \ne v\}$,  
  we abbreviate 
  \begin{equation}
  \sub_n(H) \triangleq \sub(H,K_n) = \binom{n}{|V(H)|}  \frac{|V(H)|!}{\aut(H)},
\label{eq:subH}
\end{equation}
where the equality follows by enumerating all possible vertex relabeling of $H$: $V(H) \to [n]$ 
modulo the 
automorphism of $H$ (see, e.g.~\cite[Lemma 5.1]{frieze2016introduction}). 
	For each subset $S \subset \binom{[n]}{2}$, we identify it with an (edge-induced) subgraph of $K_n$.
    Let $\calS_n$ denote the set of permutations $\pi : [n] \to [n]$.
    For a given permutation $\pi \in \calS_n$, let  
    $\pi(S) = \{(\pi(u),\pi(v)): (u,v) \in S\}$, which is identified with a relabeled version of $S$.
    

    

    For two real numbers $a$ and $b$, we let $a \vee b \triangleq \max\{a, b\}$
and $a\wedge b \triangleq \min\{a, b\}$. 
We use standard asymptotic notation: For two sequences $\{a_n\}$ and $\{b_n\}$ of positive numbers, we write $a_n = O(b_n)$, if $a_n \le C b_n$ for an absolute constant $C$ and for all $n$; $a_n = \Omega(b_n)$, if $b_n = O(a_n)$; $a_n = \Theta(b_n)$, if $a_n = O(b_n)$ and $a_n = \Omega(b_n)$; 
$a_n = o(b_n)$ or $b_n = \omega(a_n)$, if $a_n / b_n \to 0$ as $n\diverge$.

 The rest of the paper is organized as follows. In \prettyref{sec:main_result}, we 
first present the theoretical guarantees of our signed tree-counting statistic; then discuss the connection to the low-degree approximation of likelihood ratio and the computational hardness based on the low-degree approximation; and finally conclude this section with further related works. \prettyref{sec:stats} provides the statistical analysis of our signed tree-counting statistic, with detailed proofs deferred to  \prettyref{app:stats}. 
In \prettyref{sec:alg_tree_counting}, we present an efficient algorithm to approximately compute our signed tree-counting statistic based on color coding. Detailed proofs of the statistical and computational guarantees of the algorithm are postponed to \prettyref{app:algorithm}. 
In \prettyref{sec:numerical}, we conduct numerical experiments to corroborate our theoretical findings. 
In addition, several preliminary facts on graphs are collected in \prettyref{app:prelim}. 

\section{Main results and discussions} \label{sec:main_result}

\subsection{Test statistics and theoretical guarantees}

Deviating from the previous approach of counting subgraphs in $G_1$ and $G_2$ directly~\cite{barak2019nearly},
we propose to count trees in the centered version of $G_1$ and $G_2$, which we call \emph{signed} trees following \cite{bubeck2016testing}.
Specifically, denote by $A$ and $B$ the adjacency matrices of $G_1$ and $G_2$ respectively, and by $\bar{A}=A-\expect{A}$ and $\bar{B}=B-\expect{B}$ their centered version. 
Let $\calT$ denote the set of unlabeled trees with $K$ edges. 
For example, for $K=4$, $\calT$ consists of three trees shown in pictograms below (see \cite[App.~I]{harary1959number} for bigger examples)
\[
\calT = \sth{\PathFive \, , \quad  \Chair \, , \quad \Star }
\]
and our test statistic is determined by the number of their copies in the observed graphs.

Define
\begin{align}
f_{\calT} (A,B) \triangleq \sum_{ [H] \in \calT} f_H(A, B), \quad \text{ where }
f_H(A, B) \triangleq  \beta \, \aut(H)  W_{H} (\Bar{A})  W_{H} (\Bar{B}) \, .
\label{eq:f}
\end{align}
Here 
\begin{align}
    \beta= \left(\frac{\rho}{q(1-q)}\right)^{K} \frac{(n-K-1)!}{n!} \label{eq:beta} 
\end{align} 
is a scaling factor introduced for ease of analysis, 
each $[H]\in\calT$ is an unlabeled tree,
and for any weighted adjacency matrix $M$ on vertex set $[n]$, we define
\begin{align}
    W_{H}(M) \triangleq  \sum_{ S \cong H } \prod_{(i,j)\in S } M_{ij} \, , \label{eq:W_H}
\end{align}
where the sum is over subgraphs of $K_n$ that are isomorphic to $H$.
When $M$ is the unweighted adjacency matrix of a graph $G$, $W_H(M)$ reduces to the subgraph count $\sub(H,G)$.  Thus $W_H(M)$ can be viewed as a natural generalization of the subgraph count to weighted graphs. Crucially, after centering
$\expect{W_H(\bar{A}) W_{H'}(\bar{A})} \neq 0$ if and only if 
$H \cong H'$; the same identity holds for $\bar{B}$. 
This orthogonality property immediately implies that $f_H(A,B)$ and $f_{H'}(A,B)$ are uncorrelated under the null hypothesis $\calH_0$ for $H \ncong H'$ and further enables us to control the correlation between $f_H(A,B)$ and $f_{H'}(A,B)$ under $\calH_1$. Indeed, we can readily  show that $\Expect_{\calQ}[f_{\calT} (A,B)]=0$ and 
$$
\frac{(\Expect_{\calP}[f_{\calT}(A,B)])^2}{\Var_\calQ[f_{\calT}(A,B)]} =\rho^{2K} |\calT| \, .
$$
A celebrated result of Otter \cite{otter1948number} is that the number of unlabeled trees grows exponentially with
\begin{equation}
\lim_{K\to \infty} |\calT|^{1/K}= 1/\alpha,
    \label{eq:otter}
\end{equation}
 where $\alpha \approx 0.33833$ is Otter's constant. Therefore whenever the correlation satisfies $\rho^2 >\alpha$, we have $\Var_\calQ[f_{\calT}] =o((\Expect_{\calP}[f_{\calT}])^2)$. 
Furthermore, with additional assumptions,
we can show that 
$\Var_\calP[f_{\calT}]=o((\Expect_{\calP}[f_{\calT}])^2)$.
This requires a delicate analysis of the covariance between
$W_H(\bar{A}) W_{H}(\bar{B})$
and $W_{H'}(\bar{A}) W_{H'}(\bar{B})$ and  further leveraging the tree property (see~\prettyref{rmk:prop2}). Combining these variance bounds with Chebyshev's inequality, we arrive at the following sufficient condition for the statistic $f_{\calT}(A,B)$ to achieve consistent detection. 

\begin{theorem} \label{thm:low_degree}
Suppose   
\begin{align}
    n \min \{q, 1-q \} \ge n^{-o(1)}, \quad \rho^2 > \alpha,  \quad \omega(1) \le K \le \frac{\log n}{16 \log \log n \vee 2 \log \left(\frac{1}{n \min\{q, 1-q\}}\right)} \, , \label{eq:K_condition}
\end{align}
%
where $\alpha \approx 0.33833$ is  Otter's constant. 
Then the testing error satisfies 
\begin{equation}
    \label{eq:fsuccess}
\calQ (f_{\calT}(A,B)\ge \tau) + \calP (f_{\calT}(A,B)\le \tau )=o(1),
\end{equation}
where the threshold is chosen as
\[
\tau = C \expects{f_{\calT}(A,B)}{\calP}
=C\rho^{2K}|\calT|
\]
for any fixed constant $0<C<1$.
\end{theorem}  

Note that the condition \prettyref{eq:K_condition} remains unchanged with $q$ replaced by $1-q$, as we can equivalently test the correlation between the complement graphs of the observed graphs, which follow the correlated \ER model with parameters $(n,1-q,\rho)$. 
The condition $nq \ge n^{-o(1)}$ in fact applies to the very sparse regime of vanishing average degrees, as long as they are slower than any polynomial in $n$. This sparsity condition turns out to be necessary for the proposed test to succeed. To see this, observe that the threshold for the emergence of trees
with $K$ edges in $\calG(n,q)$ is at $nq=\Theta(n^{-1/K})$~\cite[Corollary 2.7]{frieze2016introduction}. Thus to ensure the existence of trees
with $K=\omega(1)$ edges, we need $nq\ge n^{-o(1)}$.

From a computational perspective,  evaluating each $W_H(\bar{A})$ in \prettyref{eq:f} by exhaustive search  takes $n^{O(K)}$ time which is super-polynomial when $K=\omega(1)$.
To resolve this computational issue, 
 in \prettyref{sec:computation_color_coding}
we design 
an $n^{2+o(1)}$-time
algorithm
(see  \prettyref{alg:algorithm1}) to 
compute an approximation $\tilde f_\calT(A,B)$ (see \prettyref{eq:Y_calT})
for $f_{\calT}(A,B)$ using the strategy of \emph{color coding}~\cite{alon1995color,arvind2002approximation,hopkins2017bayesian}.
    The following result shows that 
    the statistic $\tilde f_\calT$ 
    achieves consistent detection under the same condition as in \prettyref{thm:low_degree}. 
    
  \begin{theorem}\label{thm:compution}
   Suppose \prettyref{eq:K_condition} holds.
Then
\prettyref{eq:fsuccess}
holds with $\tilde f_\calT$ in place of $f_\calT$, namely
   \begin{align}
   \label{eq:tildefsuccess}
      \calQ (\tilde f_{\calT}(A,B)\ge \tau )+  \calP (\tilde f_{\calT}(A,B)\le \tau )=o(1).
   \end{align} 
  Moreover, $\tilde f_{\calT}(A,B)$ can be computed 
  in  $n^{2+o(1)}$ time.
  \end{theorem}

  \begin{remark}[Comparison to  statistical  limit]
  It is instructive to compare the performance guarantee of 
  our polynomial-time algorithm with the detection threshold derived in~\cite{wu2020testing}.
  Recall the equivalence between the correlated \ER graph model $(n,q,\rho)$ and the subsampling model $(n,p,s)$, when $q=ps$
  and $\rho=\frac{s(1-p)}{1-ps}>0$. As shown by~\cite{wu2020testing}, 
  in the dense
  regime with $n^{-o(1)} \le p \le 1-\Omega(1)$, 
  the information-theoretic threshold for detection is $s^2= \frac{2 \log n}{np \left(\log (1/p)-1+p \right)}$ with a sharp constant $2$,
  whereas in the sparse regime with $p=n^{-\Omega(1)}$, 
  the detection threshold is $s^2 \asymp \frac{1}{np} \wedge 0.01$. 
Therefore, when $n^{-o(1)} \le np \le O(1)$, our test statistic 
  $\tilde f(A,B)$ achieves the information-theoretic detection threshold up to a 
  constant factor in polynomial time. Interestingly, when $n^{-o(1)} \le np<2/\alpha$, the requirement $\rho^2 > \alpha$ of our test statistic becomes even less stringent than the existing performance guarantee for the maximum correlation test, which computes the maximum number of common edges over all possible vertex correspondences and demands $s^2 >\frac{2 \log n}{np\log (1/p)}$~\cite{wu2020testing}. However, when $np=\omega(1)$, our performance guarantee is far away from the information-theoretic detection threshold. Whether or not the detection threshold can be achieved up to a constant factor in polynomial time when $np =\omega(1)$ remains an open problem. 
  
  \end{remark}

\subsection{Low-degree polynomial approximation of likelihood ratio}
\label{sec:f_test}
    
    Our tree-counting statistic can also be interpreted as a low-degree polynomial approximation of the (optimal) likelihood ratio test. 
    Specifically,  the likelihood ratio between $\calP$ and $\calQ$ is  given by
    $$
    L(A,B)=\frac{\calP(A,B)}{\calQ(A,B)} = \Expect_\pi\left[\frac{\calP(A,B|\pi)}{\calQ(A,B)}\right],
    $$ 
    where the latent $\pi$ is uniform over $\calS_n$, the set of all permutations on $[n]$. As such, directly computing this average over $n!$ permutations is computationally intractable. To obtain a computationally efficient test, as discussed in \cite{hopkins2017bayesian,hopkins2018statistical,kunisky2019notes}, one approach is via
    projecting $L(A,B)$ onto the space of low-degree polynomials. To this end, 
    define an inner product 
    $$
    \langle f,g\rangle \triangleq \Expect_{\calQ} \left[f(A,B)g(A,B)\right]
    $$
    for any  functions $f,g:\{0,1\}^{2\binom{n}{2}} \to \reals$. 
     Next, we introduce an orthonormal 
    polynomial basis under $\calQ$, indexed by subsets of $\binom{[n]}{2}$ or equivalently edge-induced subgraphs of $K_n$.
		Throughout the paper we write as shorthand
		\[
		\sfS \triangleq (S_1 , S_2), \quad S_1, S_2 \subset \binom{[n]}{2}.
		\]
		Define
    \begin{align}
    \phi_\sfS(A,B) \triangleq \sigma^{-|S_1|-|S_2|}\prod_{(i,j)\in S_1}\bar{A}_{ij} \prod_{(k,\ell)\in S_2}\bar{B}_{k\ell}  \, , \label{eq:phi_S}
    \end{align}
    where $\sigma^2 \triangleq q(1-q)$ is the variance of $A_{ij}$ (resp.\ $B_{ij}$) for any $(i,j)\in \binom{[n]}{2}$.
    Then, $\{\phi_{\sfS}\}_{S_1,S_2 \subset \binom{[n]}{2}}$
		is a Fourier basis for functions on the hypercube $\{0,1\}^{2\binom{n}{2}}$ \cite{o2014analysis}. In particular, we have 
    $
    \Expect_{\calQ}\left[\phi_\sfS\right]=0,
    $
    and 
 $   \left \langle \phi_\sfS, \phi_{\sfT} \right \rangle =\indc{\sfS=\sfT}, $
    for any $S_1,S_2,T_1,T_2 \subset\binom{[n]}{2}$. 
    Moreover, note that $\phi_\sfS$ is a degree-$|\sfS|$ polynomial of the entries of $A$ and $B$, where $|\sfS| \triangleq |S_1| + |S_2|$. 

It turns out that our tree-counting statistic $f_{\calT}(A,B)$ 
corresponds to the projection of $L(A,B)$ to the space spanned by the bases $\{\phi_\sfS\}$, where $\sfS$ consists of $S_1$ and $S_2$ that are both trees with $K$ edges. More formally, for each  collection $\calH$ of unlabeled graphs containing no isolated vertex,\footnote{Throughout the paper, unless otherwise stated, all subgraphs are  edge-induced subgraphs and  contain no isolated vertices.} we define 
\begin{align}
f_{\calH}(A,B) 
         \triangleq \sum_{[H]\in \calH}  \sum_{S_1:S_1 \cong H} \sum_{S_2: S_2 \cong H} \left\langle L,\phi_\sfS\right \rangle \phi_\sfS  \label{eq:f_calH}  .
\end{align}
We will verify in \prettyref{sec:low_degree_analysis} that
the definition of $f_{\calT}$ given in~\prettyref{eq:f} is equivalent to that in~\prettyref{eq:f_calH}. 
    
    

 \subsection{Computational hardness conjecture based on low-degree approximation} \label{sec:hardness}
 Equipped with the view of the low-degree approximation, our result  also sheds light on the computational limits of polynomial-time algorithms. 
 Going beyond our tree-counting statistic, it is natural to consider the ``optimal'' degree-$2K$ polynomial of $(A,B)$, that is,
 $$
    f^* 
    = \argmax_{f: \text{ degree$(f)$}\le 2 K}  \frac{\Expect_{\calP}\left[f\right]}{\sqrt{\Expect_{\calQ}\left[f^2\right]}}
    = \argmax_{f: \text{ degree$(f)$}\le 2 K}  \frac{\langle L, f \rangle}{\sqrt{\langle f, f \rangle}},
    $$
    where $\frac{\Expect_{\calP}\left[f\right]}{\sqrt{\Expect_{\calQ}\left[f^2\right]}}$ can be viewed as the signal-to-noise ratio for the test statistic  $f$. By the Cauchy-Schwarz inequality,
    we readily get that $f^*=f_{\calH^*}(A,B)$, where 
    \begin{align}
     \calH^* \triangleq \left\{\left[H\right]: |E(H)| \le K, \text{ $H$ contains no isolated vertex}\right\}.
    \label{eq:def_H_star}
    \end{align}
    Thus $f^*$ corresponds to the statistic by counting all subgraphs with $K$
    edges and no isolated vertex. 

 It is postulated in \cite{hopkins2018statistical,kunisky2019notes} that, if the signal-to-noise ratio $\frac{\Expect_{\calP}\left[f^*\right]}{\sqrt{\Expect_{\calQ}\left[(f^*)^2\right]}}$ stays bounded for $K = \mathrm{polylog}(n)$ as $n \to \infty$, then no polynomial-time algorithm can distinguish between $\calP$ and $\calQ$ with vanishing error. 
Our result in \prettyref{prop:mean_p_q_var_q} shows that 
$$
\frac{\Expect_{\calP}\left[f^*\right]}{\sqrt{\Expect_{\calQ}\left[(f^*)^2\right]}}
= \pth{ \sum_{[H]\in  \calH^*}  \rho^{2 |E(H)|} }^{1/2} . 
$$
Since an unlabeled graph $[H] \in \calH^*$ with $k$ edges has at most $2k$ vertices, 
the number of such graphs is at most $\binom{(2k)^2}{k} \le (4ek)^k$ and hence
$$
\sum_{[H]\in  \calH^*}  \rho^{2 |E(H)|}
\le \sum_{k=1}^K (4ek)^k \rho^{2 k} 
= O(1) 
$$
for $K = \mathrm{polylog}(n)$, provided that $\rho^2 \le \frac{1}{\mathrm{polylog}(n)}$. 
Therefore, if the squared correlation $\rho^2$  is smaller than $\frac{1}{\mathrm{polylog}(n)}$, then the  signal-to-noise ratio for any degree-$\mathrm{polylog}(n)$ polynomial test is bounded, in which case the testing problem is conjectured to be computationally hard. 
In view of the close connection between hypothesis testing and estimation, we further conjecture the graph matching problem (namely, recovering the latent permutation $\pi$ under the correlated \ER model $\calG(n,q,\rho)$)
is computationally hard when $\rho^2 \le \frac{1}{\mathrm{polylog}(n)}$. Note that these conjectures are consistent with the state-of-the-art results for which no polynomial-time test or matching algorithm is known when $\rho^2 \le \frac{1}{\mathrm{polylog}(n)}$ (cf.~\cite{ding2018efficient,
FMWX19a,FMWX19b,
ganassali2020tree,ganassali2021correlation,mao2021random,mao2021exact} and the present paper). 
An intriguing numerical coincidence was recently reported in \cite[Fig.~9]{piccioli2021aligning},
leading to a speculation that $\rho^2>\alpha$ is the computational limit of both detection and recovery of graph matching when $nq \diverge$.

\subsection{Further related work}
Cycle counting has been widely used for hypothesis testing in networks with latent structures. 
In the context of community detection, 
counting 
cycles of logarithmic length
has been shown to achieve the optimal detection threshold for distinguishing the stochastic block model (SBM) with two symmetric communities from the \ER graph model in the sparse regime \cite{mossel2015reconstruction}. 
In the dense regime, 
counting \emph{signed} cycles turns out to achieve the optimal asymptotic power \cite{banerjee2018contiguity,banerjee2017optimal}. These results have been extended to the degree-corrected SBM in \cite{gao2017testing1,gao2017testing,jin2019optimal}, which focus on the asymptotic normality of cycle counts of constant length.

The importance of counting signed subgraphs was recognized in \cite{bubeck2016testing}, which showed that counting signed triangles 
is nearly optimal for testing high-dimensional random geometric graphs versus \ER random graphs in the dense regime.
Indeed, the variance of signed triangles can be dramatically smaller than the variance of triangles, due to the cancellations introduced by the centering of adjacency matrices.  

While we also focus on counting signed subgraphs in this paper, importantly, our approach deviates from the existing literature on cycles in the following two aspects. First, counting trees 
turns out to be much more powerful than counting cycles for detecting network correlation. 
A simple explanation is that the class of cycles is not rich enough: There are exponentially many unlabeled $K$-trees but only one $K$-cycle.
Second, in the aforementioned literature on the SBM, 
when the graphs are sufficiently sparse, one can simply count the usual subgraphs as opposed to their signed version; however, for our problem, even when the average degree  $np$ is as small as a constant $c>1$, 
it is still more advantageous to center the adjacency matrices and count
 signed trees, which helps mitigate the  correlations across counts of non-isomorphic trees. 

Let us compare the algorithmic approach in the present paper with the existing literature. Computationally, to count (signed) subgraphs of size $K$ in graphs with $n$ vertices, the na\"ive exhaustive search takes $O(n^K)$ time which is not polynomial in $n$ if $K\to\infty$. This difficulty was overcome in the previous work on the SBM as follows:
In the sparse regime, 
\cite{mossel2015reconstruction} showed that local neighborhoods of \ER and SBM graphs are tangle-free with high probability so that  
cycles of logarithmic sizes can be counted efficiently by
counting the number of non-backtracking walks. In the dense regime, \cite{banerjee2018contiguity,banerjee2017optimal} showed that with high probability counts of signed cycles of growing lengths
can be approximated by certain linear spectral statistics of the standardized adjacency matrix, which can be efficiently computed using the eigendecomposition.
In our paper, we take a different route  following~\cite{hopkins2017bayesian} by showing that with high probability our test statistic involving counting signed trees can be efficiently approximated in both the dense and sparse regimes, using the randomized algorithm of color coding.

In passing, 
we remark that the recent work \cite{ganassali2021correlation} studies a related problem on detecting whether two Galton-Watson trees are either independently or  correlatedly generated.
A recursive message passing algorithm is proposed to compute the likelihood ratio and
a set of necessary and sufficient conditions 
for achieving the so-called one-sided detection (type-I error is $o(1)$ and type-II error is $1-\Omega(1)$) 
is obtained. 
Furthermore, these results are used to construct an efficient algorithm for partially aligning two correlated \ER graphs
by testing the correlation of local neighborhoods, which can be approximately viewed as Galton-Watson trees in the sparse regime.

\section{Statistical analysis of tree counting}\label{sec:stats}

In this section, we establish the statistical guarantee of our tree-counting statistic $f_{\calT}(A,B)$, as stated in~\prettyref{thm:low_degree}. To this end, in~\prettyref{sec:low_degree_analysis} we first show that $f_{\calT}(A,B)$ can be equivalently rewritten as a low-degree test defined in~\prettyref{eq:f_calH}. Then in~\prettyref{sec:mean_variance}, we bound the mean and variance of $f_{\calH}$ for a general family $\calH$ of subgraphs. 
By specializing these bounds to trees, we finally prove the
statistical guarantee for $f_{\calT}(A,B)$ and hence the desired~\prettyref{thm:low_degree}. 




\subsection{Equivalence to low-degree test} \label{sec:low_degree_analysis}

    We now compute the low-degree projection of the likelihood ratio $L=\frac{d\calP}{d\calQ}$ using the orthonormal basis $\{\phi_\sfS\}_{S_1,S_2\subset \binom{[n]}{2}}$. 
    First, the coefficient of $L$ along the basis function $\phi_\sfS$ is 
    \begin{align*}
    \left\langle L, \phi_\sfS\right \rangle  
    = \Expect_{\calP}\left[\phi_\sfS(A,B)\right] 
    = \Expect_{\pi}\Expect_{\calP|\pi}\left[\sigma^{-|S_1|-|S_2|}\prod_{(i,j)\in S_1}\bar{A}_{ij}\prod_{(k,\ell)\in S_2} \bar{B}_{k\ell}\right] ,
    \end{align*}
    where $\sigma^2=q(1-q)$. Recall that when $\calP$ is the correlated \ER graph model $\calG(n,q,\rho)$, $\{(A_{ij}, B_{\pi(i)\pi(j)})\}$ are i.i.d pairs of $\Bern(q)$ random variables with correlated coefficient $\rho$, so that
		\[
		\Expect_\calP[\bar{A}_{ij}\bar{B}_{\pi(i)\pi(j)}]= \rho \sigma^2.
		\]
		Therefore, the inner expectation vanishes if $\pi(S_1) \ne S_2$ where $\pi(S_1) \triangleq \{(\pi(i),\pi(j)): (i,j) \in S\}$,
		 and is equal to $\rho^{|S_1|}$ if $\pi(S_1) = S_2$. Consequently, we have 
    \begin{align*}
    \left\langle L, \phi_\sfS\right \rangle  
    = \rho^{|S_1|}\Prob\left(\pi(S_1)=S_2\right) . 
    \end{align*}
    For any graph $H$, define
    \begin{align}
        a_H \triangleq \rho^{|E(H)|} \frac{1}{\sub_n(H)} \, . \label{eq:a_H}
    \end{align}
     If $S_1\cong S_2\cong H$ for some $H$,  then $\Prob\left(\pi(S_1)=S_2\right)=\frac{1}{\sub_n(H)}$ (see \prettyref{lmm:O}\ref{O:1} in \prettyref{app:prelim}).
		In all, we have 
    \begin{align}
        \left\langle L, \phi_\sfS\right \rangle  = 
				\begin{cases}
				a_H 	&  S_1\cong S_2\cong H\\
				0 & S_1\not\cong S_2
				\end{cases}
				\label{eq:L_phi_s}.
    \end{align}
    By \prettyref{eq:f_calH} and \prettyref{eq:L_phi_s}, we have
    \begin{align}
        f_{\calH}(A,B) 
        & = \sum_{[H]\in \calH} a_H  \sum_{S_1 \cong H} \sum_{S_2 \cong H} \phi_\sfS  \label{eq:f_calH_1}   \\
				& =  \sum_{[H]\in \calH} a_H \sigma^{-2|E(H)|} \underbrace{\sum_{S_1 \cong H}  \prod_{(i,j)\in S_1} \bar{A}_{ij}}_{W_H(\bar A)}  \underbrace{\sum_{S_2 \cong H} \prod_{(k,\ell)\in S_2}\bar{B}_{k\ell}}_{W_H(\bar B)}. \label{eq:f_calH_2}
    \end{align}
Recall that $\calT$ is the set of unlabeled trees with $K$ edges. 
For any $[H] \in \calT$, by \prettyref{eq:subH},
$$
\sub_n(H) = \binom{n}{K+1} \frac{(K+1)!}{\aut(H)},
$$
so 
$ a_H \sigma^{-2|E(H)|} = \rho^K \sigma^{-2K} \frac{(n-K-1)!}{n!} \aut(H) 
= \left(\frac{\rho}{q(1-q)}\right)^K \frac{(n-K-1)!}{n!} \aut(H)$. 
Therefore, by 
\prettyref{eq:f_calH_2}, we obtain the equivalence of the definition of
$f_{\calT}$ given in  \prettyref{eq:f} and
that in 
\prettyref{eq:f_calH}.

\subsection{Mean and variance calculation}
\label{sec:mean_variance}
In this section, we compute the mean and variance of the
test statistic $f_{\calH}$ for a general family of unlabeled graphs $\calH \subset \calH^*$, 
where $\calH^*$ 
is defined in \prettyref{eq:def_H_star}. 

\begin{proposition}\label{prop:mean_p_q_var_q}
For any subfamily $\calH \subset \calH^*$, 
   \begin{align}
        \Expect_{\calQ}\left[f_{\calH}\right]
        & = 0, \label{eq:erqmean} \\
    \Expect_{\calP}\left[f_{\calH}\right]
        & =   \Var_{\calQ}\left[f_{\calH}\right] = \sum_{[H]\in  \calH}  \rho^{2 |E(H)|}. \label{eq:erpmean_qvar}
    \end{align}
It follows that if $\calH$ consists of graphs with $K$ edges and 
    \begin{align}
        \label{eq:prop1} |\calH| = \omega \left( \frac{1}{\rho^{2K}} \right), 
    \end{align}
    then $\Expect_\calP[f_\calH] 
    =\omega(1)$.
\end{proposition}

  \begin{proof}
		First, we calculate the mean of $f_{\calH}$ under $\calQ$ and  $\calP$.
    Since $\Expect_{\calQ}[\phi_\sfS]=0$ for each $\sfS=(S_1,S_2)$, we have by linearity
		$\Expect_{\calQ}\left[f_{\calH}\right]=0$.
    Since $L=\frac{d\calP}{d\calQ}$, by a change of measure and using the expression of $f_{\calH}$ in \prettyref{eq:f_calH_1}, 
    \begin{align*}
        \Expect_{\calP}\left[f_{\calH}\right]
        & = \sum_{[H]\in  \calH} a_H \sum_{S_1\cong H} \sum_{S_2\cong H} \left\langle L,\phi_\sfS\right \rangle \overset{(a)}{=} \sum_{[H]\in  \calH} a_H^2 \sub_n(H)^2  \overset{(b)}{=} \sum_{[H]\in  \calH}  \rho^{2 |E(H)|}, 
    \end{align*}
    where $(a)$ follows from \prettyref{eq:L_phi_s} and the fact
    that 
    $\sum_{S_1\cong H} 1 = \sub_n(H)$
    as $H$ does not contain any isolated vertex;
    $(b)$ holds because $a_H = \rho^{|E(H)|} \frac{1}{\sub_n(H)}$.
    Moreover,
     \begin{align*}
        \Var_{\calQ}\left[f_{\calH}\right]
        & =\Expect_{\calQ}\left[f_{\calH}^2\right] = \left\langle f_{\calH}, f_{\calH} \right \rangle  \\
        &= \sum_{[H]\in \calH} a_H \sum_{[I]\in \calH} a_I \sum_{S_1\cong H} \sum_{S_2\cong H} \sum_{T_1\cong I}\sum_{T_2\cong I} \iprod{\phi_\sfS}{\phi_\sfT} \\
        &\overset{(a)}= \sum_{[H]\in \calH}a_H \sum_{[I]\in \calH} a_I \sum_{S_1\cong H} \sum_{S_2\cong H} \sum_{T_1\cong I}\sum_{T_2\cong I} \indc{\sfS = \sfT} \\
        & \overset{(b)}{=}  \sum_{[H]\in \calH} a_H^2 \sub_n(H)^2  
        = \sum_{[H] \in \calH} \rho^{2 |E(H)|} , 
        \end{align*}
        where $(a)$ holds because $\left\{\phi_{\sfS}\right\}_{S_1,S_2\subset \binom{[n]}{2}}$ are orthonormal;
        $(b)$ holds because $\sfS = \sfT$ implies that $[H]=[I]$ and $a_H=a_I$. 
\end{proof}



 The following result bounds the variance of $f$ under the alternative hypothesis $\calP$. 
\begin{proposition}\label{prop:ER_var_calP}
Assume $q\le \frac{1}{2}$. 
Let $\calH$ denote a subset of connected graphs with $K$ edges and at most $\NumVer$ vertices. 
Define 
\begin{equation}
\Phi_{\calH} \triangleq\min_{[H]\in\calH} \Phi_H
, \quad \text{ where }
\Phi_H \triangleq \min_{J \subset H: |V(J)| \ge 1}  n^{ |V(J)|} q^{|E(J)| },
\label{eq:PhiH}
\end{equation}
where the minimum is taken over all subgraphs $J$ of $H$ containing at least one vertex. 
%
 If in addition to \eqref{eq:prop1}, 
  \begin{equation}
 \label{eq:prop2}
    \Phi_{\calH} 
     = \omega \left( \rho^{-2K} M^{5M} 2^{5M+ 2K}\right), 
 \end{equation}
 then $
        \Var_{\calP}\left[f_{\calH}\right]/
        \left(\Expect_{\calP}\left[f_{\calH}\right]\right)^2 =  o\left(1 \right).
   $
Moreover,  for trees, \ie, $\calH=\calT$,
both~\prettyref{eq:prop1} and~\prettyref{eq:prop2} hold under assumption~\prettyref{eq:K_condition}.
\end{proposition}

\begin{remark}\label{rmk:prop2}
Note that $n^{|V(J)|}q^{|E(J)|}$  is approximately equal to $\expect{\sub(J,G_1)}$ where $G_1 \sim \calG(n,q)$.
Thus $\Phi_H$ determines the expected number of copies of the rarest non-empty subgraph of $H$ in $G_1$. It is well known that the ratio 
$\Var(\sub(H,G_1))/\left(\expect{\sub(H,G_1)}\right)^2$ is approximately given by  
$1/\Phi_H $
 (see e.g.~\cite[Lemma 3.5]{janson2011random}). In a similar vein, it turns out that  for
two non-isomorphic subgraphs $H,H' \in \calH$,
the correlation coefficient between $W_H(\bar{A}) W_{H}(\bar{B})$ and $W_{H'}(\bar{A}) W_{H'}(\bar{B})$ 
under $\calP$  is also approximately upper bounded by $1/\Phi_H$. Therefore, condition~\prettyref{eq:prop2} ensures that $W_H(\bar{A}) W_{H}(\bar{B})$ are weakly correlated across different $H \in \calH$ so that  $ \Var_{\calP}\left[f_{\calH}\right]$
is small compared to $
        \left(\Expect_{\calP}\left[f_{\calH}\right]\right)^2 $. 

To satisfy both condition~\prettyref{eq:prop1} and condition~\prettyref{eq:prop2}, the class $\calH$ should be chosen as rich as possible (at least exponentially large), while at the same time keeping $\Phi_\calH$ large. A natural choice for $\calH$  is $\calT$, the class of trees  
with $K$ edges, whose cardinality grows exponentially in $K$ as $(1/\alpha)^K$ --- see \prettyref{eq:otter}. Moreover, since any subgraph $J$ of a tree has $|V(J)| \ge |E(J)|+1$, it follows that $\Phi_\calT=n \min\{(nq)^K, 1\}$.  Thus it can be verified that both~\prettyref{eq:prop1} and~\prettyref{eq:prop2} hold under assumption~\prettyref{eq:K_condition}. 
In addition to this statistical benefit, the choice of trees also enables efficient computation of $f_{\calT}$, as we will see shortly in the next section. 


Computational considerations aside, another choice of $\calH$ is the class of 
balanced graphs (recall \prettyref{sec:subgraph_count}) with $K$ edges and $K/(1+\epsilon)$ vertices for a small constant $\epsilon>0$.
It is shown in~\cite{barak2019nearly} that $|\calH| \ge K^{cK}$ for some constant $c=c(\epsilon)>0$. Moreover, when $nq \ge n^{2\epsilon}$, by the balanced property, $\Phi_{\calH} \ge n^{\epsilon}$. Thus, it can be verified that both~\prettyref{eq:prop1} and~\prettyref{eq:prop2} hold by choosing $K=c'(\epsilon) \frac{\log n}{\log \log n}$ for some constant $c'(\epsilon)>0$ and $\rho> 2/K^c$.
This result extends the statistical guarantee in~\cite{barak2019nearly} to the full range of average degree as low as $n^{2\epsilon}$; however, it still suffers from the computational complexity issue, as there is no known
polynomial-time algorithm to efficiently compute $f_{\calH}$ over balanced graphs for large $K$.
\end{remark}

\subsection{Proof sketch for \prettyref{prop:ER_var_calP}} 
\label{sec:pf-sk-var}
To prove \prettyref{prop:ER_var_calP}, we need a key lemma that provides estimates for $\left\langle L, \phi_\sfS \phi_\sfT \right \rangle $. Recall that we denote $\sfS = (S_1,S_2)$ and $ \sfT= (T_1,T_2)$ where $S_1,S_2,T_1,T_2 \subset \binom{[n]}{2}$. In view of \prettyref{eq:L_phi_s}, we assume $S_1 \cong S_2$ and $T_1 \cong T_2$ for otherwise $\left\langle L, \phi_\sfS \phi_\sfT \right \rangle=0$. 
\begin{lemma} \label{lmm:L_phi_S_T}
		Assume that $S_1 \cong S_2 \cong H$ and $T_1 \cong T_2 \cong I$, where $H$ and $I$ are connected graphs with $K$ edges and at most $\NumVer$ vertices.
    Then we have the following: 
\begin{enumerate}[label=(\roman*)]
    \item (No overlap)
    If $V(S_1)\cap V(T_1)=\emptyset$ and $V(S_2)\cap V(T_2) = \emptyset$, then 
    \begin{align}
     0 \le \left\langle L, \phi_\sfS \phi_\sfT \right \rangle 
     & \le a_H a_I \gamma \left(1+\indc{H\cong I}\right)\, , \label{eq:L_phi_S_T_disjoint}
    \end{align}
    where $a_H$ is defined in \prettyref{eq:a_H} and
    \[
    \gamma \triangleq \exp\left(\frac{\NumVer^2}{n-2\NumVer+1}\right) \, ;
    \]
    \item (Single overlap)
    If $V(S_1)\cap V(T_1) = \emptyset$  and $V(S_2)\cap V(T_2) \ne \emptyset$, or if $V(S_1)\cap V(T_1) \ne \emptyset$ and $V(S_2)\cap V(T_2) = \emptyset$, then 
    \begin{align}
         \left\langle L, \phi_\sfS \phi_\sfT \right \rangle  
         = 0 \, ; \label{eq:L_phi_S_T_not_disjoint-1}
    \end{align}
    \item (Double overlap)     If
     $V(S_1)\cap V(T_1) \ne \emptyset$ and $V(S_2)\cap V(T_2) \ne \emptyset$, assuming $q \le 1/2$, then
    \begin{align}
         \left\langle L, \phi_\sfS \phi_\sfT \right \rangle  
        & \le \left(\frac{2\NumVer}{n}\right)^{|V(H)|+|V(I)|- |V(S_1)\cap V(T_1)|-|V(S_2)\cap V(T_2)|} \nonumber \\
        &~~~~ \left( \indc{S_1 = T_1} + h(S_1,T_1) \right)   \left( \indc{S_2 = T_2} + h(S_2,T_2) \right) \label{eq:L_phi_S_T_not_disjoint}  \, , 
    \end{align}
    where for any $S,T\subset \binom{[n]}{2}$,
      \begin{align}
               h(S,T) \triangleq \sum_{E \subset  S \cap T} q^{-\frac{1}{2}|E|} \left(\frac{2\NumVer}{n}\right)^{ \frac{1}{2}\left(|V(E)|+\indc{E = \emptyset}\right)}\, . 
				\label{eq:h_S_T}
           \end{align}
    
    Moreover, 
    \begin{align}
          h(S,T)  \le   
          \frac{ 2^{K} (2\NumVer)^{\frac{\NumVer}{2}}}{\sqrt{ \max\{ \Phi_S, \Phi_T\} }}  
          \, , \label{eq:upper_bound_h}
    \end{align}
    where $\Phi_S$ and $\Phi_T$ are defined as per \prettyref{eq:PhiH}. 
    \end{enumerate}
				\end{lemma}
    
\begin{remark}
    Note that $\langle L, \phi_\sfS \phi_\sfT \rangle= \expects{\phi_\sfS\phi_\sfT}{\calP}$.
    Thus, \prettyref{lmm:L_phi_S_T} bounds from above the correlation between $\phi_\sfS$ and $\phi_\sfT$ under $\calP$ according to different overlapping patterns of $\sfS$ and $\sfT$.
\begin{itemize}
\item In case $(i)$, $\sfS$ and $\sfT$ do not overlap at all.  Note that $a_H a_I=\iprod{L}{\phi_\sfS}\iprod{L}{\phi_\sfT}$. 
Since $\gamma=1+o(1)$ as long as $M^2=o(n)$,
the upper bound \prettyref{eq:L_phi_S_T_disjoint},
loosely speaking, suggests that $\phi_\sfS$ and $\phi_\sfT$  behave as if they were  uncorrelated in this case.

\item In case $(ii)$, it turns out that $\expects{\phi_\sfS\phi_\sfT}{\calP}=0$, which implies that $\Cov(\phi_\sfS, \phi_\sfT)= -\expects{\phi_\sfS}{\calP}\expects{\phi_\sfT}{\calP}<0$ and hence  $\phi_\sfS$ and $\phi_\sfT$ are negatively correlated.

\item In case $(iii)$,  both $S_1, T_1$ and $S_2, T_2$
share common vertices,
we expect $\phi_\sfS$ and $\phi_\sfT$ are positively correlated. Indeed, in view of \prettyref{eq:L_phi_s} and \prettyref{eq:subH}, 
$\iprod{L}{\phi_\sfS}\iprod{L}{\phi_\sfT}$ is on the order of $\rho^{2K} n^{-|V(H)|-|V(I)|}$.
Thus,  
the upper bound \prettyref{eq:L_phi_S_T_not_disjoint}, loosely speaking, suggests that $\expects{\phi_\sfS\phi_\sfT}{\calP}$
is larger than 
$\iprod{L}{\phi_\sfS}\iprod{L}{\phi_\sfT}$
by a multiplicative factor of 
\begin{align*}
    \xi
    & \triangleq \rho^{-2K} (2M)^{|V(H)| + |V(I)|} n^{|V(S_1)\cap V(T_1)|+|V(S_2)\cap V(T_2)|}\\
    &~~~~ \left( \indc{S_1 = T_1} + h(S_1,T_1) \right)   \left( \indc{S_2 = T_2} + h(S_2,T_2) \right).
\end{align*}
This factor $\xi$ is polynomial in $n$
with the exponent given by the number of overlapping vertices. At one extreme,  $\xi$ achieves its maximal value when $\sfS$ and $\sfT$ completely overlap; fortunately, such pairs of 
$\sfS$ and $\sfT$ are few. At the other extreme, when $S_1 \cap T_1=\emptyset$ and $S_2 \cap T_2=\emptyset$, 
crucially we have $h(S_1,T_1) h(S_2, T_2)=\frac{2M}{n}$,
which is the key to guarantee that $\xi$ is sufficiently small in this case. In general,
thanks to \prettyref{eq:upper_bound_h}, 
we can upper bound $h(S,T)$ by bounding $\Phi_{S}$ from below.
\end{itemize}
\end{remark}

To prove~\prettyref{lmm:L_phi_S_T}, the most challenging part is case $(iii)$, which requires us to carefully  consider 
all the possible overlaps between $S_1, T_1$ in $A$ and $S_2, T_2$ in $B$, and bound their contributions by averaging over the latent permutation $\pi$ and exploiting bounds on overlap sizes. See Appendix~\ref{sec:proof_L_phi_S_T} for details.


With~\prettyref{lmm:L_phi_S_T}, we  outline the proof of~\prettyref{prop:ER_var_calP}. Note that 
        \begin{align*}
         \Var_{\calP}\left[f_{\calH}\right]
            &  =\Expect_{\calP}\left[f_{\calH}^2\right]- \left(\Expect_{\calP}\left[f_{\calH}\right]\right)^2  \nonumber \\
            & =  \sum_{[H]\in \calH}a_H \sum_{[I]\in \calH}a_I \sum_{S_1\cong H} \sum_{S_2 \cong H} \sum_{T_1\cong I}\sum_{T_2\cong I}  \left(\left\langle L, \phi_\sfS \phi_\sfT \right \rangle - \left\langle L, \phi_\sfS\right \rangle   \left\langle L, \phi_\sfT\right \rangle \right) \, . 
        \end{align*}
    In view of
    \prettyref{eq:L_phi_s}, we have 
    $ \left\langle L, \phi_\sfS\right \rangle   \left\langle L, \phi_\sfT\right \rangle  \ge 0$
     for any $S_1 \cong S_2 \cong H$ and $T_1 \cong T_2 \cong I$.
    Therefore, applying \prettyref{eq:L_phi_S_T_not_disjoint-1} in~\prettyref{lmm:L_phi_S_T}, we get that the contribution of the single-overlap case to $\Var_{\calP}[f_\calH]$ is non-positive. Hence we can bound  $\Var_{\calP}\left[f_\calH\right]$ from the above by $\termI + \termII$, 
    where
     \begin{align}
       \termI 
        & = \sum_{[H]\in \calH}a_H \sum_{[I]\in \calH}a_I \sum_{S_1\cong H} \sum_{S_2\cong H} \sum_{T_1\cong I}\sum_{T_2\cong I}  \left(\left\langle L, \phi_\sfS \phi_\sfT \right \rangle - \left\langle L, \phi_\sfS\right \rangle   \left\langle L, \phi_\sfT\right \rangle \right) \nonumber 
        \\
        &~~~~  \indc{V(S_1) \cap V(T_1)= \emptyset } \indc{V(S_2) \cap V(T_2)= \emptyset } \label{eq:ER_varP_termI}\, , \\
        \termII 
        & = \sum_{[H]\in \calH} a_H \sum_{[I]\in \calH} a_I\sum_{S_1\cong H} \sum_{S_2\cong H} \sum_{T_1\cong I}\sum_{T_2\cong I}  \left\langle L, \phi_\sfS \phi_\sfT \right \rangle \indc{\left\langle L, \phi_\sfS \phi_\sfT \right \rangle \ge 0 }  \nonumber \\
        &~~~~ \indc{ V(S_1) \cap V(T_1)\neq \emptyset \text{ and } V(S_2) \cap V(T_2) \neq \emptyset}  \label{eq:ER_varP_termII} \, . 
    \end{align}
  Note that  $\termI$ and $\termII$ correspond to the no-overlap case and double-overlap case, respectively. 
We separately bound them by plugging in~\prettyref{eq:L_phi_S_T_disjoint}
and~\prettyref{eq:L_phi_S_T_not_disjoint}
in \prettyref{lmm:L_phi_S_T},  and completing all the summations. See~\prettyref{app:ER_var_calP} for details. 

 \subsection{Proof of \prettyref{thm:low_degree}}
 With \prettyref{prop:mean_p_q_var_q} and \prettyref{prop:ER_var_calP}, we are ready to prove \prettyref{thm:low_degree}.


     \begin{proof}[Proof of \prettyref{thm:low_degree}]
  It suffices to focus on the case $q\le 1/2$.  
If  $q>1/2$, 
consider the complement graphs of the observed graphs, which are correlated \ER graphs with parameter $(n,1-q,\rho)$.
Denote their adjacency matrices by $A^c$ and $B^c$, where $A^c_{ij}=1-A_{ij}$ and 
$B^c_{ij}=1-B_{ij}$.
The condition~\prettyref{eq:K_condition} remains unchanged. Moreover,
the test statistic in \prettyref{eq:f} satisfies $f_{\calT}(A,B) =f_{\calT}(A^c,B^c)$.
This is because the centered version satisfies $\bar{A^c} = - \bar{A}$, so that for each $[H]\in\calT$,
$W_H(\bar{A^c}) = (-1)^{|E(H)|} 
W_H(\bar{A})$ and hence 
$W_H(\bar{A^c}) W_H(\bar{B^c}) 
=W_H(\bar{A}) W_H(\bar{B})$.
Hence, without loss of generality, we assume that $q \le \frac{1}{2}$.

    
    Under  condition~\prettyref{eq:K_condition}, 
    combining \prettyref{prop:mean_p_q_var_q} and \prettyref{prop:ER_var_calP}, we have
     \begin{align*}
        \Expect_{\calQ}[f_{\calT}] = 0, \quad 
        \Var_{\calQ}\left[f_{ \calT}\right] = \Expect_{\calP}[f_{ \calT}] = o\left(\left(\Expect_{\calP}\left[f_{ \calT}\right]\right)^2\right),
        \quad \text{and} \quad
        \Var_{\calP}[f_{\calT}]= o\left(\left(\Expect_{\calP}\left[f_{ \calT}\right]\right)^2 \right). 
        \end{align*}
    Thus, for any constant $0<C<1$, 
    we obtain 
    \begin{align*}
          & \calQ\left(f_{\calT}\ge C \Expect_{\calP}[f_{\calT}] \right) \le \frac{\Var_{\calQ}\left[f_{\calT}\right]}{C^2 \left(\Expect_{\calP}[f_{\calT}]\right)^2} = o(1), \\
          & \calP\left(f_{\calT}\le C \Expect_{\calP}[f_{\calT}] \right) \le \frac{\Var_{\calP}\left[f_{\calT}\right] }{
          (1-C)^2 \left(\Expect_{\calP}\left[f_{\calT}\right]\right)^2} = o(1).
    \end{align*}
   It follows that the testing error of $f_{\calT}(A,B)$ vanishes as $n\diverge$.   
   \end{proof}

\section{Approximate test statistic by color coding} \label{sec:computation_color_coding}

In this section, we provide an efficient algorithm to approximately compute
the test statistic $f_\calT(A,B)$ given in \prettyref{eq:f_calH}, using the idea of color coding. 
Color coding was first introduced by \cite{alon1995color} 
as a randomized algorithm 
to efficiently find simple paths, cycles, or other small subgraphs (query graphs) in a given unweighted graph (host graph), and was further developed in \cite{arvind2002approximation} as a fast randomized algorithm to approximately count  subgraphs isomorphic to a query graph with a bounded treewidth (e.g., trees with bounded number of edges) in a host graph. In particular, given a query graph with $K$ vertices and a host graph with
$n$ vertices, 
the color coding method first assigns colors from $[K]$ to vertices of the host graph uniformly at random and then counts the so-called \emph{colorful subgraphs} (vertices having distinct colors) that are isomorphic to the query graph. Importantly, the process of counting the colorful subgraphs 
can be efficiently done using dynamic programming with a total time complexity that is polynomial in $n$ and exponential in $K$. Our use of color coding is inspired by the recent work \cite[Section 2.5]{hopkins2017bayesian}, which applies color coding to   
compute low-degree polynomial statistics for community detection under the stochastic block model. 
Since it is crucial to work with centered adjacency matrices, we need to extend the existing color coding algorithms 
\cite{alon1995color,arvind2002approximation} that are designed for counting unweighted graph to  weighted graphs.



By applying the color coding method, we first approximately count the signed subgraphs that are isomorphic to a query tree with $K$ edges. Specifically, given $M$ as a weighted adjacency matrix of a graph on $[n]$, we generate a random coloring $\mu: [n] \to [K+1]$ that assigns a color to each vertex of $M$  from the color set $[K+1]$ independently and uniformly at random. Given any $V\subset[n]$, let $\chi_{\mu} (V)$ indicate that $\mu(V)$ is colorful, i.e., $\mu(x)\neq \mu(y)$ for any distinct $x,y \in V$. In particular, if $|V|=K+1$, then $\chi_\mu (V)=1$ with probability 
\begin{equation}
r \triangleq \frac{(K+1)!}{(K+1)^{K+1}}.
\label{eq:r}
\end{equation} 
For any graph $H$ with $K+1$ vertices, we define
  \begin{align}
        X_H (M, \mu )
        & \triangleq  \sum_{S\cong H} \chi_{\mu}(V(S)) \prod_{(i,j)\in E(S)} M_{ij} \, .\label{eq:X_H_M}
    \end{align}
   Then $\expect{X_H (M,\mu) \mid  M }=r W_H(M)$, where $W_H(M)$ is defined in \prettyref{eq:W_H}. Hence, $X_H (M, \mu) /r$ is an unbiased estimator of $W_H(M)$. 
   To further obtain an accurate approximation of $W_H(M)$, we average over multiple copies of  
   $X_{H} (M ,\mu )$ by generating $t$ independent random colorings, where 
   \[
   t \triangleq \ceil{1/r}{}.
   \]
	Next, we plug in the averaged subgraph count to approximately compute $f_\calT(A,B)$. Specifically, we generate $2t$ random colorings $\{\mu_i\}_{i=1}^t$ and $\{\nu_j\}_{j=1}^t$ which are independent copies of $\mu$ that map $[n]$ to $[K+1]$. 
    Then, we define 
   \begin{align}
        Y_{\calT} (A,B)
        & \triangleq \sum_{[H]\in \calT} \aut(H)
         \left(\frac{1}{t} \sum_{i=1}^t X_H (\bar{A},\mu_i ) \right) \left(\frac{1}{t} \sum_{j=1}^t X_H (\bar{B},\nu_j) \right) 
        \label{eq:Y_calT_def}
   \end{align} 
   and
   \begin{equation}
      \tilde f_\calT(A,B) \triangleq \frac{\beta}{r^2} Y_\calT(A,B) \, .          \label{eq:ftilde}
   \end{equation} 
  The following result shows that $\tilde f_\calT(A,B)$ well approximates $f_{\calT}(A,B)$ in a relative sense. 
   \begin{proposition}\label{prop:er_Y_calT_testing_error}
   Suppose that \prettyref{eq:K_condition} holds.
   Then  as $n \to \infty$, under both $\calP$ and $\calQ$,
   \begin{align}
       \frac{\tilde f_\calT - f_{\calT}}{ \expects{f_{\calT}}{\calP}} \xrightarrow{L_2} 0.
       \label{eq:tildef-converge}
   \end{align}
   \end{proposition}
   Since the convergence in $L_2$ implies the convergence in probability, as an immediate corollary of~\prettyref{thm:low_degree}
 and~\prettyref{prop:er_Y_calT_testing_error}, we conclude that 
 the approximate test statistic 
 $\tilde f_\calT$ succeeds under the same condition as the original test
  statistic $f_\calT$, proving 
\prettyref{thm:compution}.  
 
    
 Finally, we show that the approximate test statistic $Y_\calT (A,B)$ can be computed efficiently using \prettyref{alg:algorithm1}. 
   \begin{algorithm}
\caption{Computation of test statistic via color coding}\label{alg:algorithm1}
\begin{algorithmic}[1]
\State{\bfseries Input:} Centered adjacency matrices $\bar A$ and $\bar B$, correlation coefficient $\rho$, and an integer $K$.
\State Apply the constant-time free tree generation algorithm in \cite{dinneen2015constant,wright1986constant} to list all non-isomorphic unrooted unlabeled trees with $K$ edges and return $\calT$. 
\State For each $[H] \in \calT$, compute $ \aut(H)$ by algorithm in \cite{colbourn1981linear}.
\State Generate \iid random colorings $\{\mu_i\}_{i=1}^t$ and $\{\nu_i\}_{i=1}^t$ mapping $[n]$ to $[K+1]$. 
\For{each $i = 1 ,\cdots, t$}
\State{\bfseries} 
For each $H \in \calT$, compute $X_H(\overline{A},\mu_i)$ and $X_H(\overline{B},\nu_i)$ via \prettyref{alg:X_M_H} given in~\prettyref{sec:alg_tree_counting}. 
\EndFor
\State Compute $Y_{\calT}(A,B)$ according to  \prettyref{eq:Y_calT_def}. 
\State{\bfseries Output:}  $Y_{\calT}(A,B)$.
\end{algorithmic}
\end{algorithm}
   
   \begin{proposition}\label{prop:Y_calT_computation}
   \prettyref{alg:algorithm1} computes $Y_\calT(A,B)$ in time 
   $O\left(n^2(3e/\alpha)^{K}  \right)$, where $\alpha$ is Otter's constant in \prettyref{eq:otter}.
   Furthermore, under condition \prettyref{eq:K_condition},  the time complexity reduces to $n^{2+o(1)}$.
   \end{proposition}
   
   Combining~\prettyref{prop:er_Y_calT_testing_error} and~\prettyref{prop:Y_calT_computation} yields~\prettyref{thm:compution}.

   \subsection{Proof of \prettyref{prop:Y_calT_computation}} \label{sec:alg_tree_counting}

        The constant-time free tree generation algorithm provided in \cite{dinneen2015constant,wright1986constant} returns a list of all non-isomorphic unrooted trees in time linear in the total number of trees. 
        Colbourn and Booth \cite{colbourn1981linear} provided an algorithm to compute the automorphism group order of a given tree in time linear in the size of the tree. Hence, the total time complexity to output $\calT$ and compute $a_H$ for each $H \in \calT$ is $O(K|\calT|)$.
      
        Next we introduce a polynomial-time algorithm to compute $X_H(M,\mu)$ given any tree $H$ with $K$ edges, a weighted graph $M$ on $[n]$, and a coloring $\mu: [n] \to [K+1]$.
				See \prettyref{fig:example} and \prettyref{tab:orbits} for an example of $H$ and the various definitions in \prettyref{alg:X_M_H}.
       Note that this algorithm is a generalization 
       of the subgraph counting  method in~\cite{alon2008biomolecular} to weighted graphs and can be further extended to any graphs $H$ with bounded treewidth. A fast implementation of the subgraph counting  method in~\cite{alon2008biomolecular} is provided by~\cite{slota2013fast}.

        \begin{algorithm}[ht] 
            \caption{Computation of $X_H(M,\mu)$}\label{alg:X_M_H}
            \begin{algorithmic}[1]
            \State{\bfseries Input:} A weighted host graph $M$ on $[n]$, a coloring $\mu:[n] \to [K+1]$, and a query tree $H$ with $K$ edges and automorphism group order $\aut(H)$. 
            \State Label the nodes in $H$ and choose an arbitrary node of $H$ as its root. Label the edges by $\{e_1,e_2,\cdots,e_K\}$ in $H$ in the reverse order visited by the depth-first search (DFS), so that the last edge visited is labeled as $e_1$. 
            \State For $i=1, \cdots, K$, let 
						$e_i = (p_i,c_i)$ where $p_i$ is the parent node of $c_i$ in the rooted tree $H$;
            let $F_i$ denote the forest consisting of edges $\{e_1,\ldots,e_i\}$ and $T_i$ the maximal tree containing edge $e_i$ in $F_i$ with root node $p_i$; let $a_i$ (resp.\ $b_i$) denote  the largest index $j<i$ such that $e_j$ is incident to $p_i$ (resp.\ $c_i$); if no such an $e_j$ exists, set $a_i=0$ (resp.\ $b_i=0$) by default. 
						
						\State For $i=1, \cdots, K$, by removing the edge $e_i$, the tree $T_i$ is partitioned into two disjoint trees $T_{a_i}$ rooted  at $p_i$ and $T_{b_i}$ rooted at $c_i$. 
            \State For $i=1, \cdots, K$, 
            for every $x\in [n]$ and every subset $C \subset [K+1]$ of colors with $|C|=|V(T_i)|$, compute recursively
            \begin{align}
             Y(x,T_i, C, \mu ) \triangleq \sum_{y\in [n] \backslash \{x\}} \sum_{ (C_1,C_2)\in \sfC(C)}
             Y(x,T_{a_i}, C_1, \mu ) \times  Y(y,T_{b_i}, C_2, \mu ) \times M_{xy} \, , \label{eq:X_recursive}
            \end{align}
            where 
              \begin{align}
                 \sfC(C) \triangleq \{(C_1,C_2):C_1 \neq \emptyset, \, C_2 \neq \emptyset, \, C_1 \cap C_2 = \emptyset, \, C_1 \cup C_2 =C \} \, , \label{eq:sfC}
             \end{align} 
             and for any rooted tree $T_0$ with a single vertex,
            \begin{align}
                    Y(x,T_{0}, C, \mu ) \triangleq  \indc{ \{\mu(x)\} = C}  \, .\label{eq:x_T_0} 
            \end{align}

            
            \State{\bfseries Output:}
            \[
            \frac{1}{\aut(H)} \sum_{x\in[n]}  Y(x,T_K,  [K+1], \mu ) \, .
            \]
            \end{algorithmic}
            \end{algorithm}

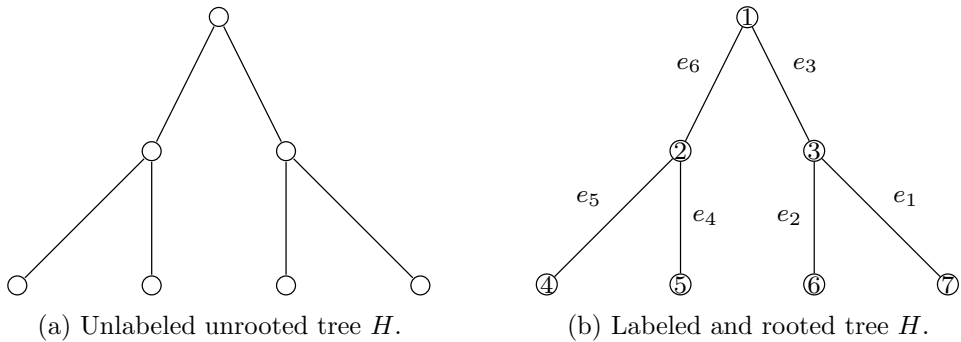
\begin{figure}[H]
\centering
\begin{tabular}
{C{.4\textwidth}C{.4\textwidth}}
\subcaptionbox{Unlabeled unrooted tree $H$.}{
\resizebox{0.35\textwidth}{!}{%
\centering
\begin{tikzpicture}[scale=1.5,font=\scriptsize]
    \draw (2,1) node[fill,circle,scale=2] (1) [nodedot] { } ;
	\draw (1.5,0) node[fill,circle,scale=2] (2) [nodedot] {} ;
	\draw (2.5,0) node[fill,circle,scale=2] (3) [nodedot] {} ;
	\draw (0.5,-1) node[fill,circle,scale=2] (4) [nodedot] {} ;
	\draw (1.5,-1) node[fill,circle,scale=2] (5) [nodedot] {} ;
    \draw (2.5,-1) node[fill,circle,scale=2] (6) [nodedot] {} ;
	\draw (3.5,-1) node[fill,circle,scale=2] (7) [nodedot] {} ;
	\draw (1)--(3) (1)--(2) (3)--(6) (3)--(7) (2)--(5) (2)--(4);
\end{tikzpicture}
}}
&
\subcaptionbox{Labeled and rooted tree $H$.}{
    \resizebox{0.35\textwidth}{!}{%
    \centering
    \begin{tikzpicture}[scale=1.5,font=\scriptsize]
    \draw (2,1) node (1) [nodedot] {1} ;
	\draw (1.5,0) node (2) [nodedot] {2} ;
	\draw (2.5,0) node (3) [nodedot] {3} ;
	\draw (0.5,-1) node (4) [nodedot] {4} ;
	\draw (1.5,-1) node (5) [nodedot] {5} ;
    \draw (2.5,-1) node (6) [nodedot] {6} ;
	\draw (3.5,-1) node (7) [nodedot] {7} ;
	\draw (1) edge[auto=right] node{$e_6$} (2);
	\draw (1) edge[auto=left] node{$e_3$} (3);
	\draw (3) edge[auto=right] node{$e_2$} (6);
	\draw (3) edge[auto=left] node{$e_1$} (7);
	\draw (2) edge[auto=right] node{$e_5$} (4);
	\draw (2) edge[auto=left] node{$e_4$} (5);
\end{tikzpicture}
}}
\end{tabular}
\caption{
In the left panel $H$ is an unlabeled unrooted tree with $6$ edges; in the right panel, we label the nodes by $\{1,2,\cdots, 7\}$ and choose node $1$ as its root, and then label edges by $\{e_1,e_2,\cdots, e_6\}$ in $H$ in the reverse order visited by the DFS.
}
\label{fig:example}
\end{figure}
\begin{table}[!ht]
\centering
\begin{tabular}{ >{\centering\arraybackslash}m{0.2in}  | >{\centering\arraybackslash}m{1.5in}  |>{\centering\arraybackslash}m{1.5in}  |   >{\centering\arraybackslash}m{0.3 in} |>{\centering\arraybackslash}m{0.3 in}}
$i$ & $F_i$ & $T_i$ & $a_i$ & $b_i$ \\
\hline
    $1$ & 
    \tikz[scale=1,font=\scriptsize]{  
	\draw (2.5,0) node (3) [nodedot] {3} ;
	\draw (3.5,-1) node (7) [nodedot] {7} ;
	\draw (3) edge[auto=left] node{$e_1$} (7);
	} 
	& \tikz[scale=1,font=\scriptsize]{  
	\draw (2.5,0) node (3) [nodedot] {3} ;
	\draw (3.5,-1) node (7) [nodedot] {7} ;
	\draw (3) edge[auto=left] node{$e_1$} (7);
	} 
	& 0 & 0 \\
	\hline
	$2$ & \tikz[scale=1,font=\scriptsize]{  
	\draw (2.5,0) node (3) [nodedot] {3} ;
    \draw (2.5,-1) node (6) [nodedot] {6} ;
	\draw (3.5,-1) node (7) [nodedot] {7} ;
	\draw (3) edge[auto=right] node{$e_2$} (6);
	\draw (3) edge[auto=left] node{$e_1$} (7);
	} &  \tikz[scale=1,font=\scriptsize]{  
	\draw (2.5,0) node (3) [nodedot] {3} ;
    \draw (2.5,-1) node (6) [nodedot] {6} ;
	\draw (3.5,-1) node (7) [nodedot] {7} ;
	\draw (3) edge[auto=right] node{$e_2$} (6);
	\draw (3) edge[auto=left] node{$e_1$} (7);
	} 
	& 1 &  0 \\
	\hline
	$3$  & \tikz[scale=1,font=\scriptsize]{  
     \draw (2,1) node (1) [nodedot] {1} ;
	\draw (2.5,0) node (3) [nodedot] {3} ;
    \draw (2.5,-1) node (6) [nodedot] {6} ;
	\draw (3.5,-1) node (7) [nodedot] {7} ;
	\draw (1) edge[auto=left] node{$e_3$} (3);
	\draw (3) edge[auto=right] node{$e_2$} (6);
	\draw (3) edge[auto=left] node{$e_1$} (7);
	} 
	& \tikz[scale=1,font=\scriptsize]{  
     \draw (2,1) node (1) [nodedot] {1} ;
	\draw (2.5,0) node (3) [nodedot] {3} ;
    \draw (2.5,-1) node (6) [nodedot] {6} ;
	\draw (3.5,-1) node (7) [nodedot] {7} ;
	\draw (1) edge[auto=left] node{$e_3$} (3);
	\draw (3) edge[auto=right] node{$e_2$} (6);
	\draw (3) edge[auto=left] node{$e_1$} (7);
	} 
	& 0 & 2 \\
	\hline
	$4$ 
	 & \tikz[scale=1,font=\scriptsize]{  
     \draw (2,1) node (1) [nodedot] {1} ;
	\draw (1.5,0) node (2) [nodedot] {2} ;
	\draw (2.5,0) node (3) [nodedot] {3} ;
	\draw (1.5,-1) node (5) [nodedot] {5} ;
    \draw (2.5,-1) node (6) [nodedot] {6} ;
	\draw (3.5,-1) node (7) [nodedot] {7} ;
	\draw (1) edge[auto=left] node{$e_3$} (3);
	\draw (3) edge[auto=right] node{$e_2$} (6);
	\draw (3) edge[auto=left] node{$e_1$} (7);
	\draw (2) edge[auto=left] node{$e_4$} (5);
	} 
	& 
	\tikz[scale=1,font=\scriptsize]{  
	\draw (1.5,0) node (2) [nodedot] {2} ;
	\draw (1.5,-1) node (5) [nodedot] {5} ;
	\draw (2) edge[auto=left] node{$e_4$} (5);
	} 
	& 0 & 0 \\
	\hline
	$5$ & \tikz[scale=1,font=\scriptsize]{  
     \draw (2,1) node (1) [nodedot] {1} ;
	\draw (1.5,0) node (2) [nodedot] {2} ;
	\draw (2.5,0) node (3) [nodedot] {3} ;
	\draw (0.5,-1) node (4) [nodedot] {4} ;
	\draw (1.5,-1) node (5) [nodedot] {5} ;
    \draw (2.5,-1) node (6) [nodedot] {6} ;
	\draw (3.5,-1) node (7) [nodedot] {7} ;
	\draw (1) edge[auto=left] node{$e_3$} (3);
	\draw (3) edge[auto=right] node{$e_2$} (6);
	\draw (3) edge[auto=left] node{$e_1$} (7);
	\draw (2) edge[auto=right] node{$e_5$} (4);
	\draw (2) edge[auto=left] node{$e_4$} (5);
	} 
	& \tikz[scale=1,font=\scriptsize]{  
	\draw (1.5,0) node (2) [nodedot] {2} ;
	\draw (0.5,-1) node (4) [nodedot] {4} ;
	\draw (1.5,-1) node (5) [nodedot] {5} ;
	\draw (2) edge[auto=right] node{$e_5$} (4);
	\draw (2) edge[auto=left] node{$e_4$} (5);
	} 
	& 4 & 0 \\
	\hline
    $6$ & \tikz[scale=1,font=\scriptsize]{  
     \draw (2,1) node (1) [nodedot] {1} ;
	\draw (1.5,0) node (2) [nodedot] {2} ;
	\draw (2.5,0) node (3) [nodedot] {3} ;
	\draw (0.5,-1) node (4) [nodedot] {4} ;
	\draw (1.5,-1) node (5) [nodedot] {5} ;
    \draw (2.5,-1) node (6) [nodedot] {6} ;
	\draw (3.5,-1) node (7) [nodedot] {7} ;
	\draw (1) edge[auto=right] node{$e_6$} (2);
	\draw (1) edge[auto=left] node{$e_3$} (3);
	\draw (3) edge[auto=right] node{$e_2$} (6);
	\draw (3) edge[auto=left] node{$e_1$} (7);
	\draw (2) edge[auto=right] node{$e_5$} (4);
	\draw (2) edge[auto=left] node{$e_4$} (5);
	} 
	& \tikz[scale=1,font=\scriptsize]{  
     \draw (2,1) node (1) [nodedot] {1} ;
	\draw (1.5,0) node (2) [nodedot] {2} ;
	\draw (2.5,0) node (3) [nodedot] {3} ;
	\draw (0.5,-1) node (4) [nodedot] {4} ;
	\draw (1.5,-1) node (5) [nodedot] {5} ;
    \draw (2.5,-1) node (6) [nodedot] {6} ;
	\draw (3.5,-1) node (7) [nodedot] {7} ;
	\draw (1) edge[auto=right] node{$e_6$} (2);
	\draw (1) edge[auto=left] node{$e_3$} (3);
	\draw (3) edge[auto=right] node{$e_2$} (6);
	\draw (3) edge[auto=left] node{$e_1$} (7);
	\draw (2) edge[auto=right] node{$e_5$} (4);
	\draw (2) edge[auto=left] node{$e_4$} (5);
	} 
	& 3 & 5 \\
	\hline
\end{tabular}
\caption{Example of the definition of $\{T_i\}_{i=1}^6$ and $\{(a_i,b_i)\}_{i=1}^6$ in \prettyref{alg:X_M_H} applied to the labeled rooted tree $H$ in \prettyref{fig:example}.}
\label{tab:orbits}
\end{table}

        The following lemma shows that the output of~\prettyref{alg:X_M_H} coincides with $X_H(M,\mu)$ and bounds the time complexity. 
        \begin{lemma}\label{lmm:X_H_M}
               For any coloring $\mu:[n]\to [K+1]$ and any tree $H$ with $K$ edges,  \prettyref{alg:X_M_H} computes $X_H(M,\mu)$
                in time  $O(K3^{K}n^2)$. 
        \end{lemma}
        
         Using \prettyref{lmm:X_H_M}, whose proof is postponed till~\prettyref{app:proof_lemma_XHM}, we complete the proof of~\prettyref{prop:Y_calT_computation}. 
        For each iteration $i$ in \prettyref{alg:algorithm1}, $X_H(\overline{A},\mu_i)$ and $X_H(\overline{B},\nu_i)$ for all $H \in \calT$ can be computed by  \prettyref{alg:X_M_H} in $O(|\calT| K 3^{K} n^2)$ time. 
         Since $ t = \ceil{1/r} = O(e^{K})$, the total time complexity of \prettyref{alg:algorithm1} to output $Y_\calH(A,B) $ is 
        \[
            O\left(|\calT| K (3e)^{K} n^2 + K|\calT|\right) =   O\left(|\calT| K (3e)^{K} n^2 \right) = O\left( \left(\frac{3e}{\alpha}\right)^{K} n^2 \right) \, ,
        \]
        where the last equality holds because the set $\calT$ of unlabeled trees with $K$ edges satisfies \cite{otter1948number}
     \begin{equation}
         |\calT|=  \left(\frac{1}{\alpha}\right)^{K+1} \left(K+1\right)^{-\frac{5}{2}} (C+o_K(1)), \quad K\to\infty
         \label{eq:otterfull}
     \end{equation}
      where $\alpha \approx 0.33833$ and $C \approx 0.53495$ are absolute constants.
        Finally, under condition \prettyref{eq:K_condition}, we have 
        $K=O(\log n/ \log \log n)$ and hence $(3e/\alpha)^{K}=e^{O(\log n/\log \log n)}=n^{O(1/\log \log n)}=n^{o(1)}$.

\section{Numerical results}\label{sec:numerical}

In this section, we provide numerical results on synthetic data  to corroborate our theoretical findings. 
To this end, we independently generate 100 pairs of graphs that are independent $G(n,q)$, and another 100 pairs from the correlated \ER model $G(n,q,\rho)$.

Fixing $n=1000$, $q=0.1$, and $\rho = 0.99$, we consider trees with $K=7$ edges and $t=1000$ random colorings, and plot the histograms of our approximated test statistics  \prettyref{eq:Y_calT_def} in
\prettyref{fig:histogram_K_7}.
We see that the two histograms under the independent and correlated models are well separated, and the type-I error and type-II error are found to be $5\%$ and $9\%$, respectively,  by selecting the detection threshold as the theoretical value
$\frac{1}{2} \Expect_\calP\left[Y_\calT\left(A,B\right)\right] = \frac{r^2}{2\beta}\rho^{2K} |\calT|$ suggested by \prettyref{thm:low_degree}.
\begin{figure}[ht]
    \centering
    \includegraphics[width=0.5\textwidth]{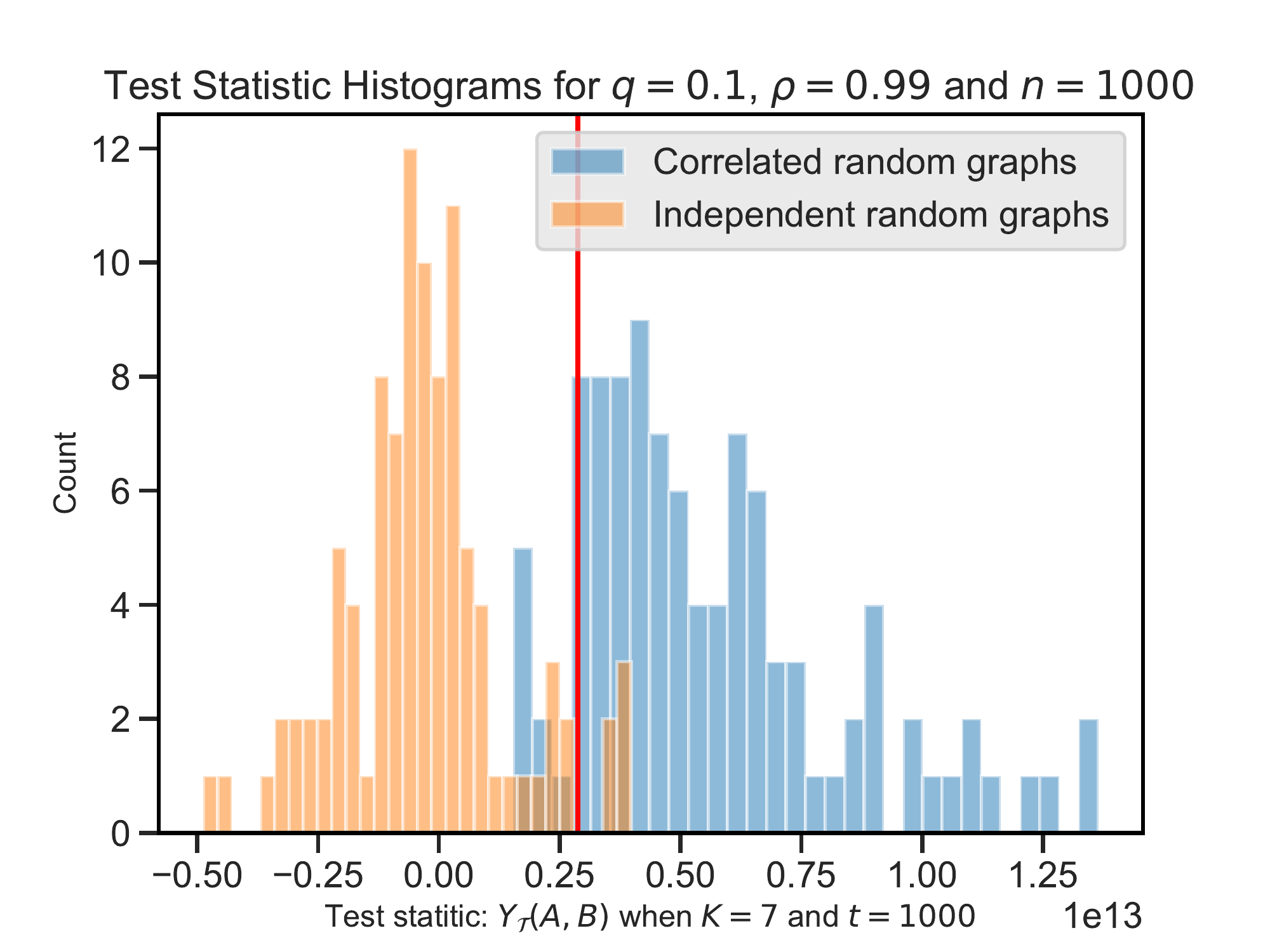}
    \caption{
    The histograms of the approximate tree  counting statistic $Y_{\calT}(A,B)$ \prettyref{eq:Y_calT_def} for $K=7$ and $t=1000$ random colorings over $100$ pairs of graphs, where the orange one corresponds to the independent model, the blue one corresponds to the correlated model, and the red line indicates the detection threshold.
    }
    \label{fig:histogram_K_7}
\end{figure}
To compare the performance of our test statistic under different settings, We also plot the Receiver Operating Characteristic (ROC) curves by
 varying detection threshold and 
 plotting the true positive rate (one minus Type-II error) against the false positive rate (Type-I error).
For comparison, we  also plot the ROC curve for the random classifier, which is simply the diagonal. Finally, we compute the area under the curve (AUC), which can be interpreted as the probability that the test statistic has a larger value for a pair of graphs drawn from the correlated model than that drawn from the independent model independently.

In \prettyref{fig:rho_change}, 
for each plot, we fix $n=1000$, $K=6$, $t=1000$, and $q \in \{0.001,0.01,0.1,0.5\}$, and vary $\rho \in \{0.8, 0.85,0.9,0.95,0.99\}$. 
We observe that as $\rho$ increases, the ROC curve is moving toward the upper left corner and the AUC increases, 
demonstrating that our test statistic has improving performance. 
Moreover, as evident in the mean and variance calculation in \prettyref{prop:mean_p_q_var_q}, we need 
$\rho \ge |\calT|^{-\frac{1}{2K}} \approx 0.82$ in order for the signal-to-noise ratio to exceed one, as there are in total $|\calT|=11$ non-isomorphic trees with $K=6$ edges. Nevertheless, \prettyref{fig:rho_change} shows that our test statistic still achieves non-trivial power even when $\rho =0.8$ is close to this threshold. 
Recall from \prettyref{thm:low_degree} that
 the smallest $\rho$ our test statistic can hope to achieve is $\lim_{K\to\infty} |\calT|^{-\frac{1}{2K}} = \sqrt{\alpha} \approx 0.581$ where $\alpha\approx 0.338$ is Otter's constant. Getting close to this threshold is computationally prohibitive as this convergence is rather slow: For  $K=35$, $|\calT|^{-\frac{1}{2K}}$ is still over $0.65$, at which case the number of unlabeled trees exceeds six trillion.
 (See  \cite{oeis-A000055} for a list of values for $|\calT|$.)

\begin{figure}[ht]
    \centering
    \begin{minipage}{0.5\textwidth}
        \centering
        \includegraphics[width=1\textwidth]{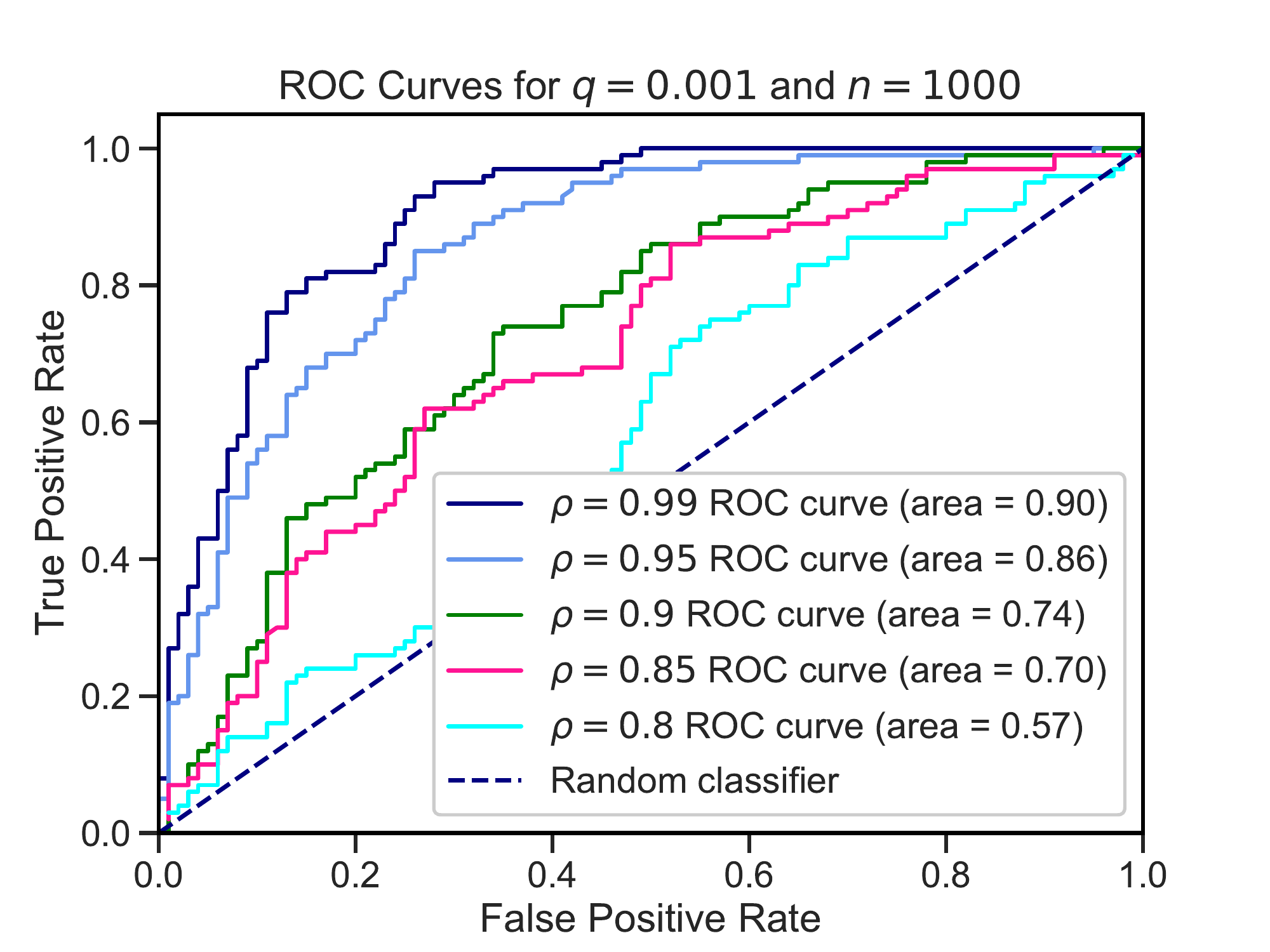} 
    \end{minipage}\hfill
    \begin{minipage}{0.5\textwidth}
        \centering
        \includegraphics[width=1\textwidth]{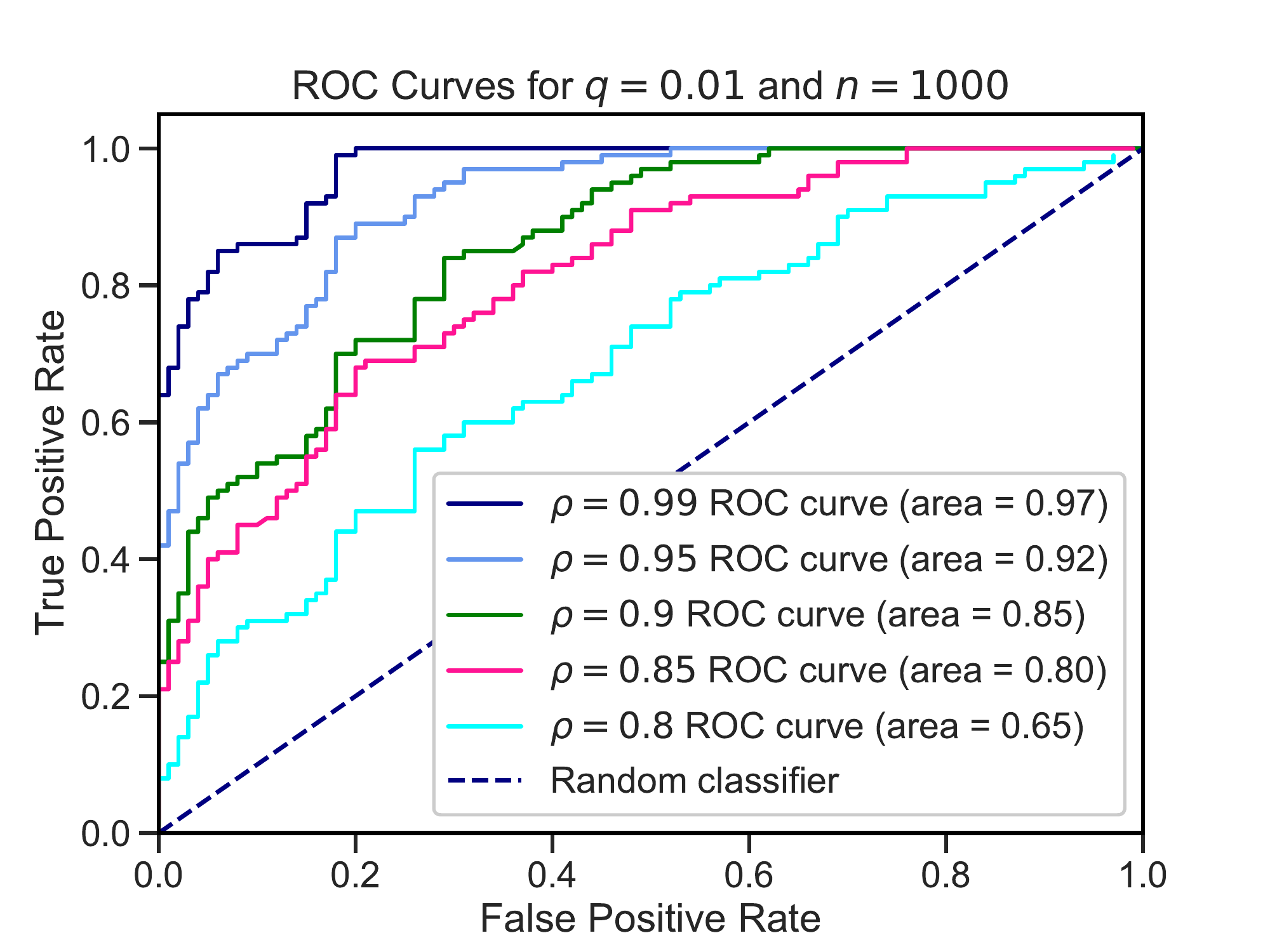} 
    \end{minipage}
        \begin{minipage}{0.5\textwidth}
        \centering
        \includegraphics[width=1\textwidth]{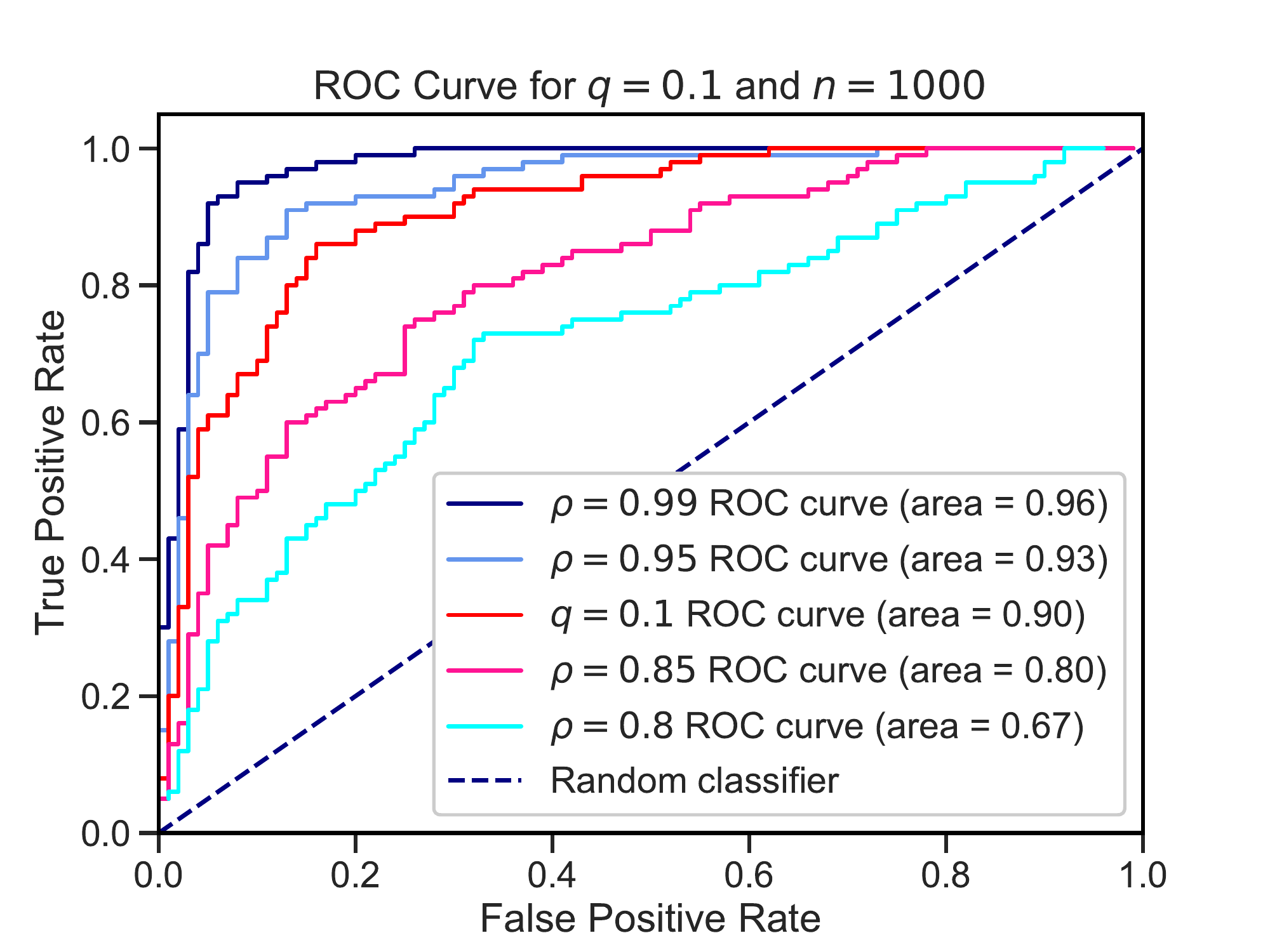} 
    \end{minipage}\hfill
        \begin{minipage}{0.5\textwidth}
        \centering
        \includegraphics[width=1\textwidth]{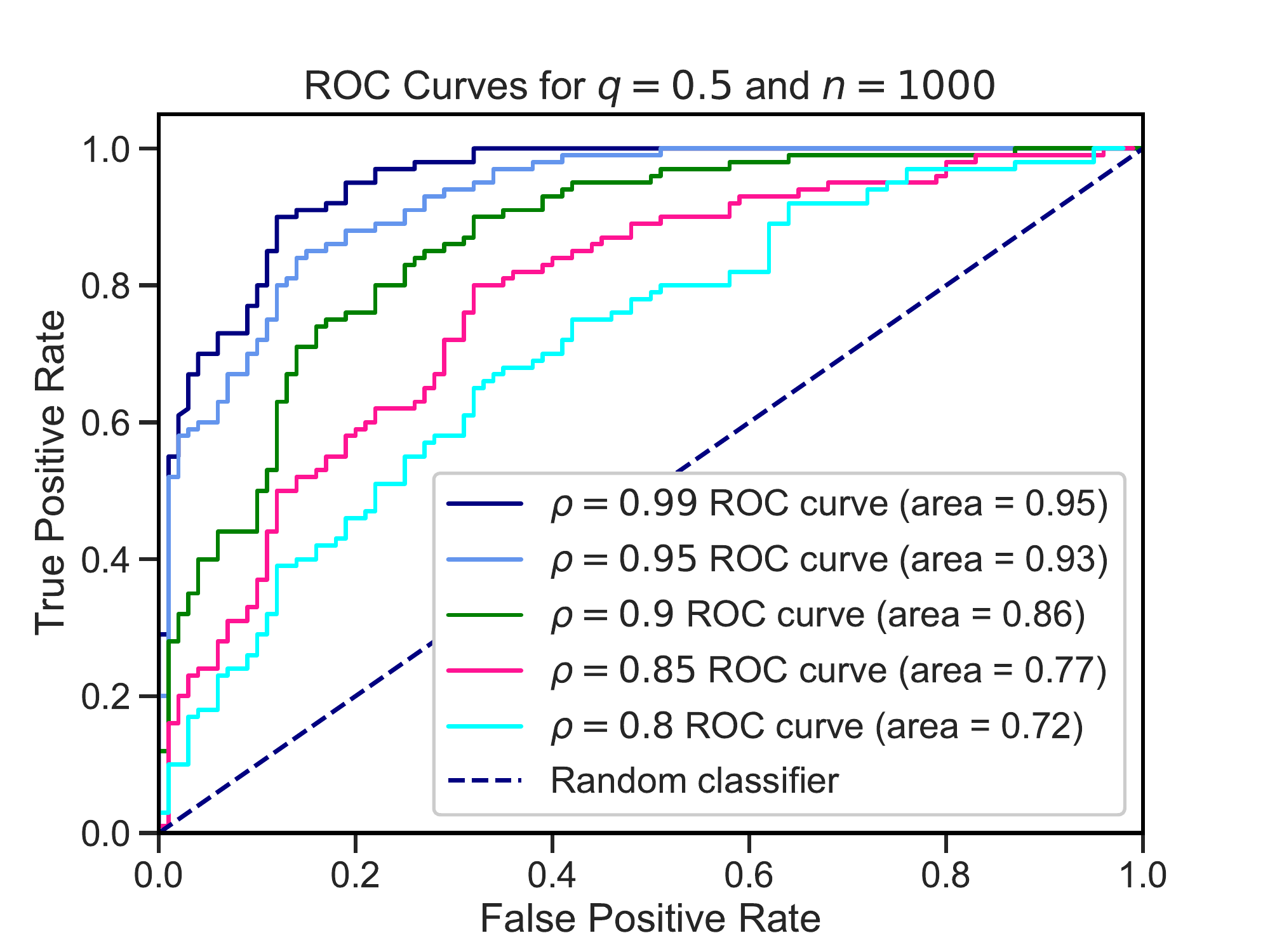} 
    \end{minipage}
    \caption{
    Comparison of the proposed test statistic
    for fixed edge probability $q$ and varying correlation parameter $\rho$.
} 
    \label{fig:rho_change}
\end{figure}
In \prettyref{fig:q_change},
we vary the edge density $q \in \{0.001,0.01,0.1,0.5\}$. 
We observe that our test statistic performs well for a wide range of graph sparsity, except that when $q=0.001$, the performance  slightly degrades. 
This is consistent with the theoretical results, showing that our test statistic works as long as the graphs are not overly sparse.
\begin{figure}[ht]
    \centering
    \begin{subfigure}{0.5\textwidth}
        \centering
        \includegraphics[width=1\textwidth]{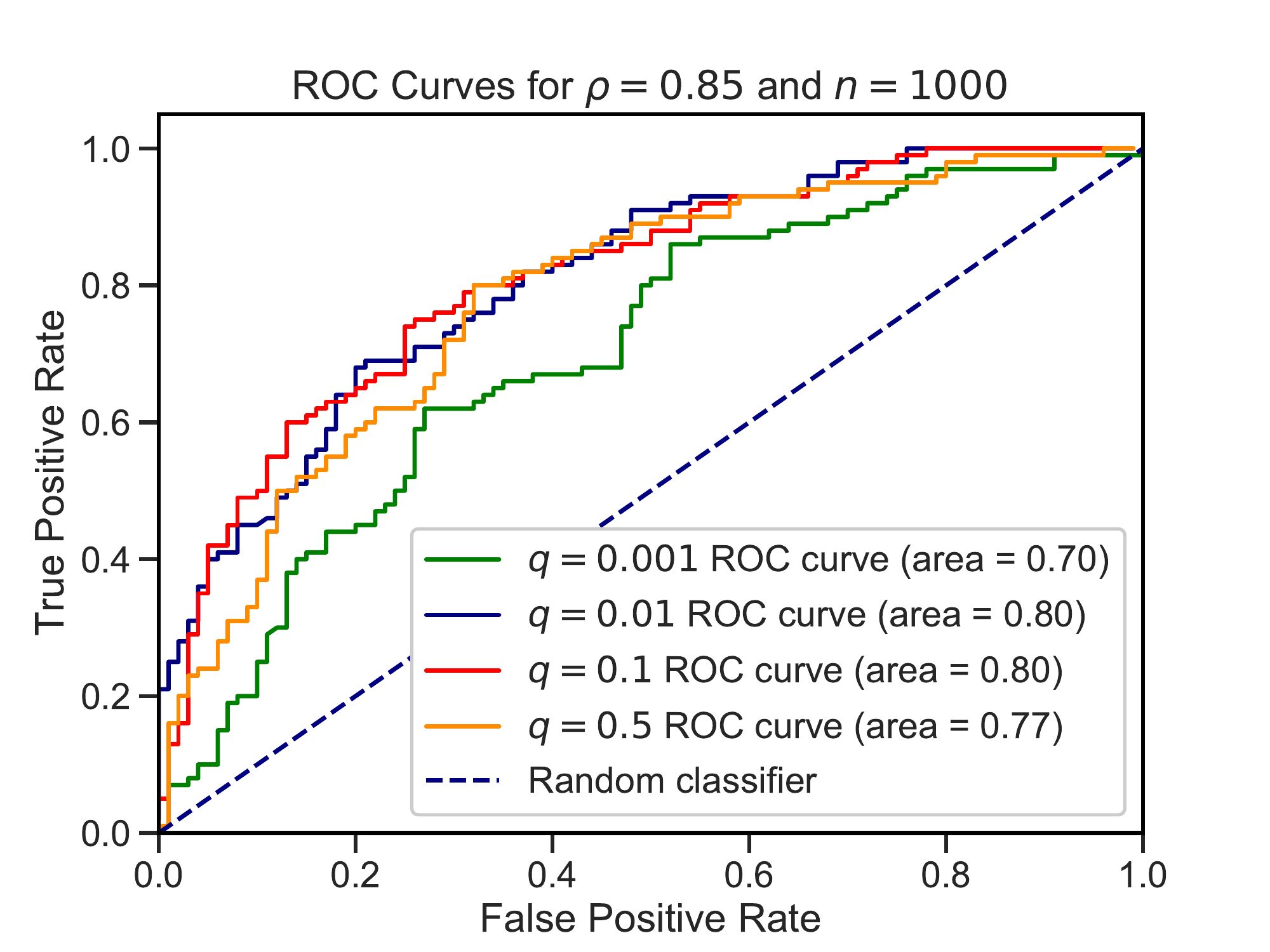} 
    \end{subfigure}\hfill
    \begin{subfigure}{0.5\textwidth}
        \centering
        \includegraphics[width=1\textwidth]{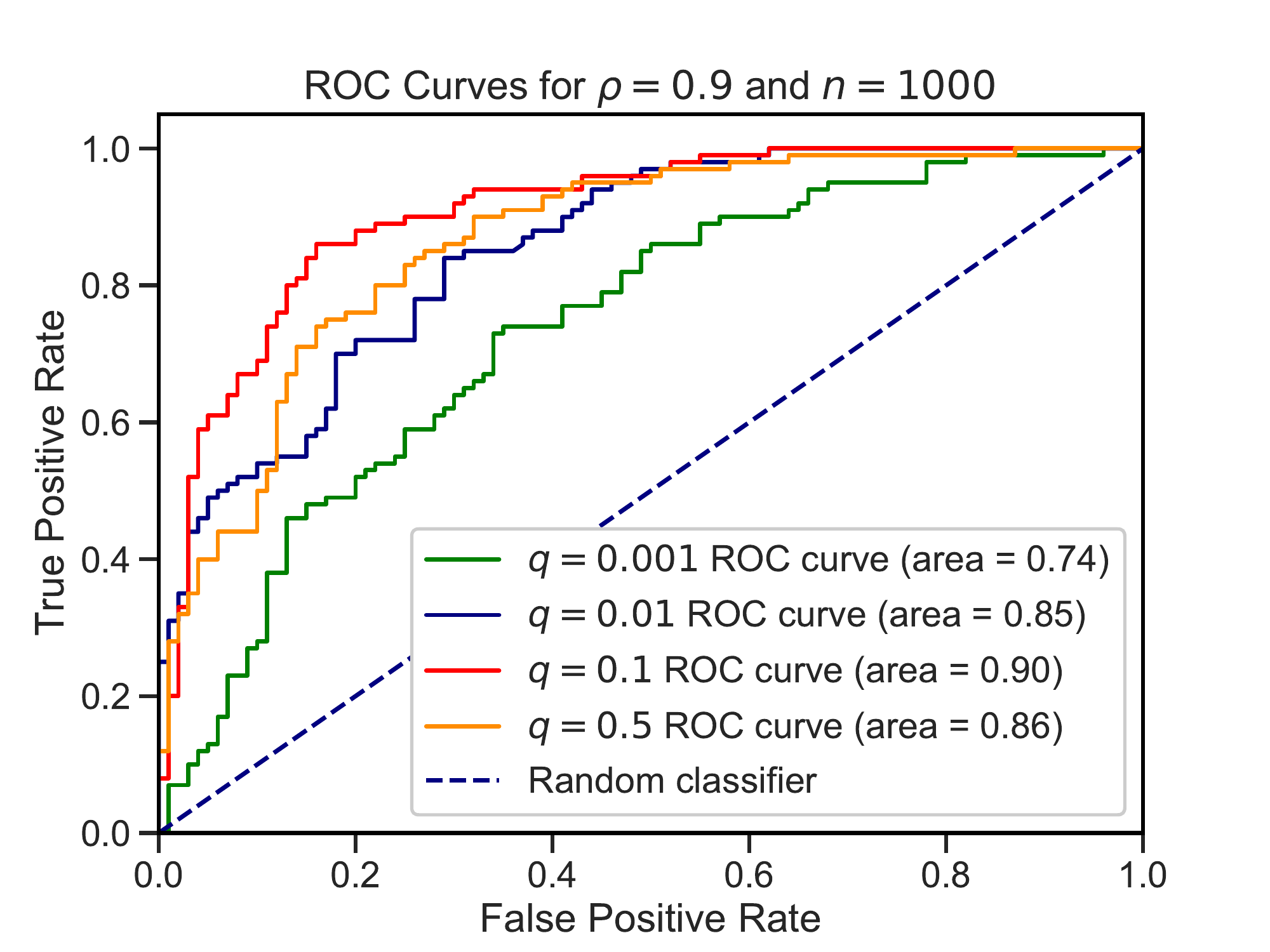} 
    \end{subfigure}
        \begin{subfigure}{0.5\textwidth}
        \centering
        \includegraphics[width=1\textwidth]{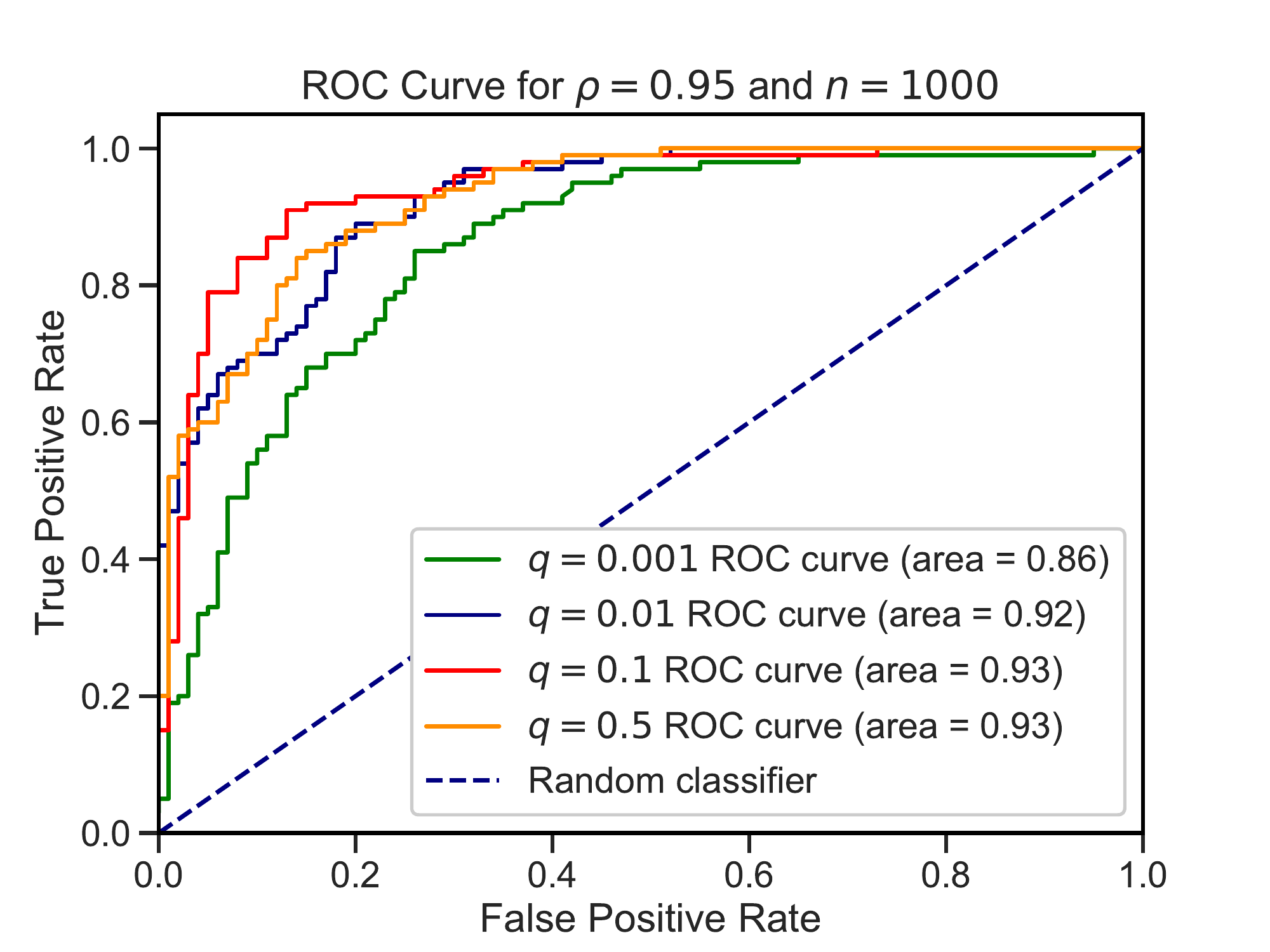} 
    \end{subfigure}\hfill
    \begin{subfigure}{0.5\textwidth}
        \centering
        \includegraphics[width=1\textwidth]{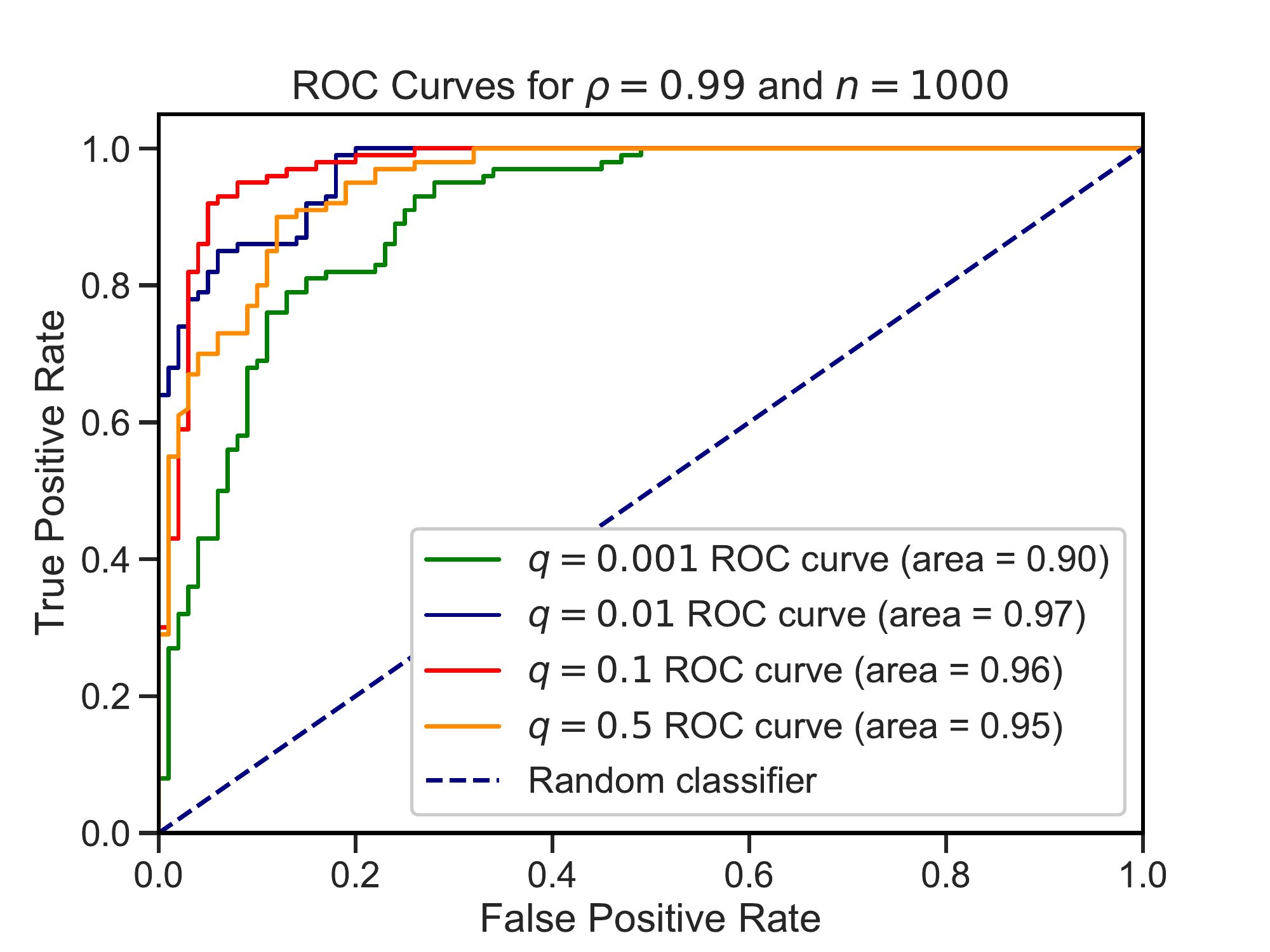} 
    \end{subfigure}
    \caption{
    Comparison of the proposed test statistic
    for fixed correlation parameter $\rho$ and varying edge probability $q$.
    } 
    \label{fig:q_change}
\end{figure}
In \prettyref{fig:K_change_roc}, we fix $n=1000$, $q=0.1$, $\rho=0.95$, and $t=1000$, and vary the tree size $K \in \{2,3,4,5,6,7\}$.
The performance of our test statistic is seen to improve significantly as  $K$ increases. 
In \prettyref{fig:K_change_time}, we plot the median running time of \prettyref{alg:algorithm1} on a pair of random graphs for each $K \in \{2,3,4,5,6,7\}$ when $t=1000$. We observe that the running time increases gradually up to $K=5$ and then rapidly afterwards. This shows a trade-off between 
the statistical performance and the computational complexity as $K$ varies.  
\begin{figure}[ht]
    \centering
    \begin{subfigure}{0.5\textwidth}
        \centering
        \includegraphics[width=1\textwidth]{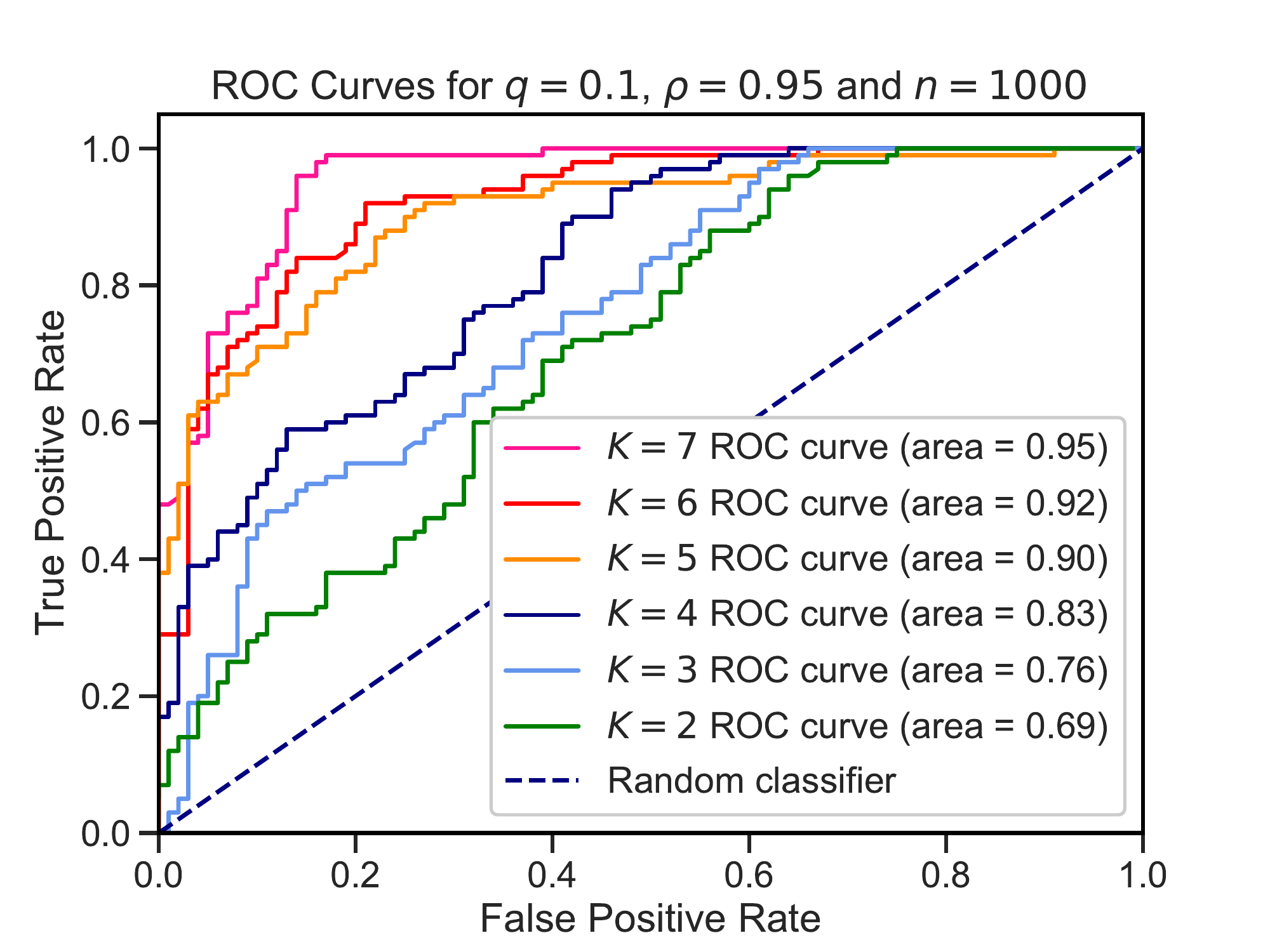} 
        \caption{ROC curves and AUC value. 
        \newline
        }
        \label{fig:K_change_roc}
    \end{subfigure}\hfill 
    \begin{subfigure}{0.5\textwidth}
        \centering
        \includegraphics[width=1\textwidth]{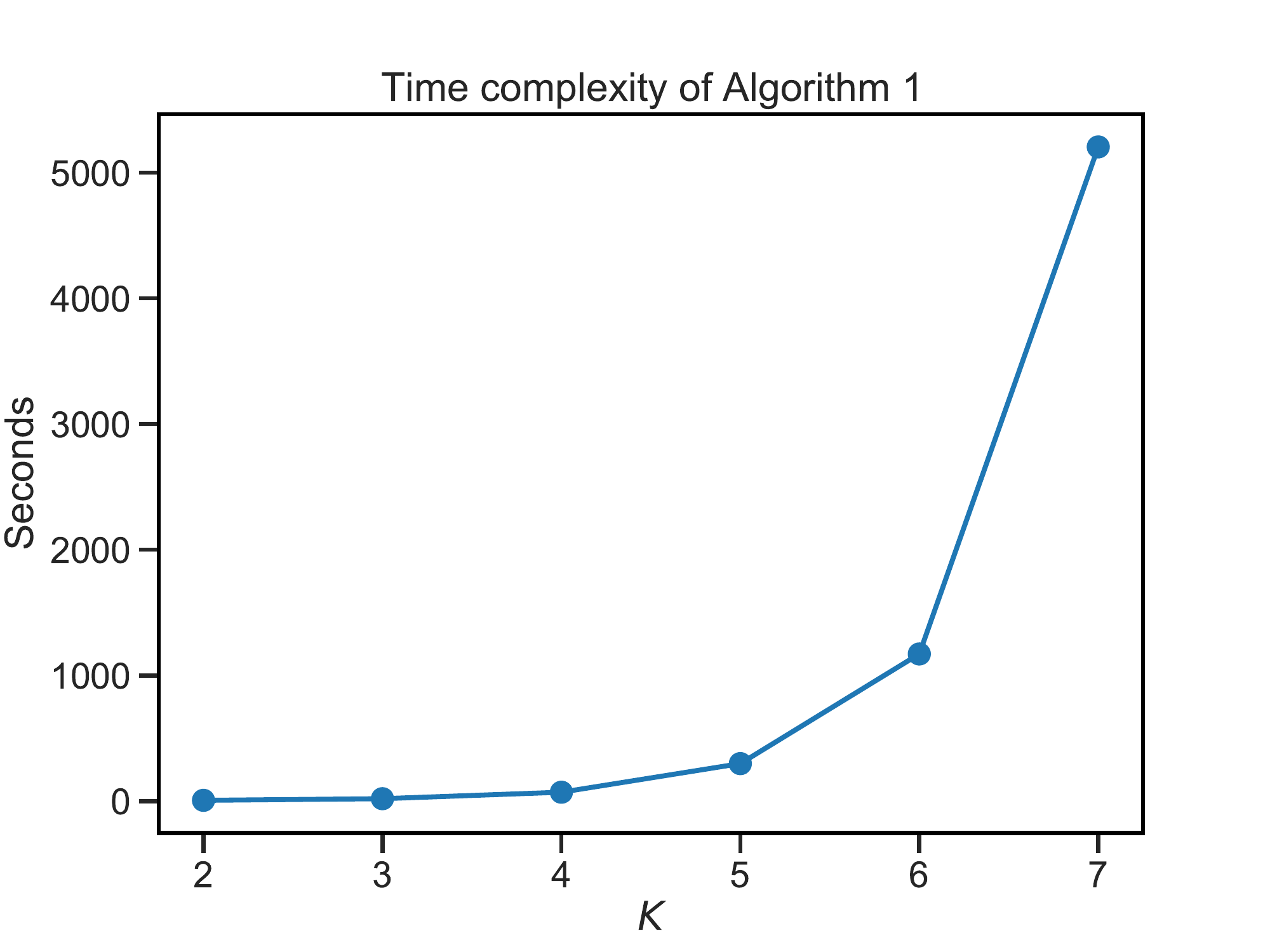} 
        \caption{Median running time of \prettyref{alg:algorithm1} on a single pair of random graphs.}
        \label{fig:K_change_time}
    \end{subfigure}
    \caption{
    Comparison of statistical performance and time complexity of the proposed test statistic with $n=1000$, $q=0.1$, $\rho=0.95$, $t=1000$, and varying tree size $K \in \{2,3,4,5,6,7\}$.
    } 
    \label{fig:K_change}
\end{figure}


Finally, we compare our test statistic with a heuristic variant that is more scalable to large \emph{sparse} graphs. In particular, instead of first centering the adjacency matrices and counting the signed trees
as per~\prettyref{eq:f}, we count the trees in the unweighted graphs in the usual sense
and then subtract their means. This gives rise to the following statistic: 
\begin{align}
    g_\calT(A,B)
    & \beta \triangleq \sum_{[H] \in \calT} \aut(H) \left(W_H(A) - \gamma_H \right) \left(W_H(B) - \gamma_H \right),
    \label{eq:gAB}
\end{align}
where $\beta$ is in \prettyref{eq:beta} and 
$\gamma_H  \triangleq  \expect{W_H(A)} = \expect{W_H(B)} = \frac{\binom{n}{K+1}(K+1)!}{\aut(H)} q^K$.
Note that the form of 
$g_\calT(A,B)$ resembles the test statistic in \cite{barak2019nearly}
as mentioned in~\prettyref{sec:subgraph_count}, but with two crucial differences: 
(a) We restrict our attention to trees as opposed to strictly balanced graphs considered in \cite{barak2019nearly};
(b) The weight in \prettyref{eq:gAB} for each $H$ is proportional to its number of automorphisms.
Now, analogous to $f_{\calT}(A,B)$, we
approximately compute
$g_{\calT}(A,B)$ via color coding. Specifically, we generate $2t$ random coloring $\{\mu_i\}_{i=1}^t$ and $\{\nu_i\}_{i=1}^t$ that map $[n]$ to $[K+1]$, and define
\begin{align*}
    Z_\calT(A,B)
    & \triangleq \sum_{[H] \in \calT} \aut(H) \left(\frac{1}{t}\sum_{i=1}^t X_H (A, \mu_i) - \gamma_H \right)\left(\frac{1}{t}\sum_{j=1}^t X_H (B, \nu_j) - \gamma_H \right) \, ,
\end{align*}
which provides an unbiased estimator for $g_\calT(A,B)$ as
\[
\expect{Z_\calT(A,B)|A,B} = \frac{r^2}{\beta} g_\calT(A,B) \, . 
\]

It turns out that the time complexity for computing $Z_\calT(A,B)$ significantly improves since we work with unweighted sparse graphs as opposed to weighted dense graphs due to centering.  
To see this, 
to compute each $X_H(A,\mu)$,
in the dynamic programming procedure (see \prettyref{eq:X_recursive} in \prettyref{alg:X_M_H}), for every $x\in[n]$, 
we only need to enumerate the neighbors of $x$ rather than all nodes. Thus, the time complexity for computing $X_H(A,\mu)$ becomes $O(m K 3^{K})$,
where $m$ is the number of edges in $A$. This significantly speeds up the computation  when $m \ll n^2$. 

In \prettyref{fig:center_vs_uncenter}, we consider a larger but sparser instance with $n=10000$, $q =0.0001$, $\rho = 0.99$, and $t=1000$, and plot the ROC curve and compute the AUC for $Z_\calT(A,B)$ when $K$ ranges from $2$ to $8$ 
respectively. 
We observe that the performance of $Z_\calT(A,B)$ gets slightly better as $K$ increases from $2$ to $5$; however, even if $K$ is as large as $8$, it still falls short of $Y_\calT(A,B)$ for $K=6$. 
This goes to show the importance of counting \emph{signed} trees in the proposed test statistic. Computationally,
$Z_{\calT}(A,B)$ is much faster to compute than $Y_{\calT}(A,B)$ (600x speedup when $K=6$). 
\begin{figure}[ht]
    \centering
    \includegraphics[width=0.7\textwidth]{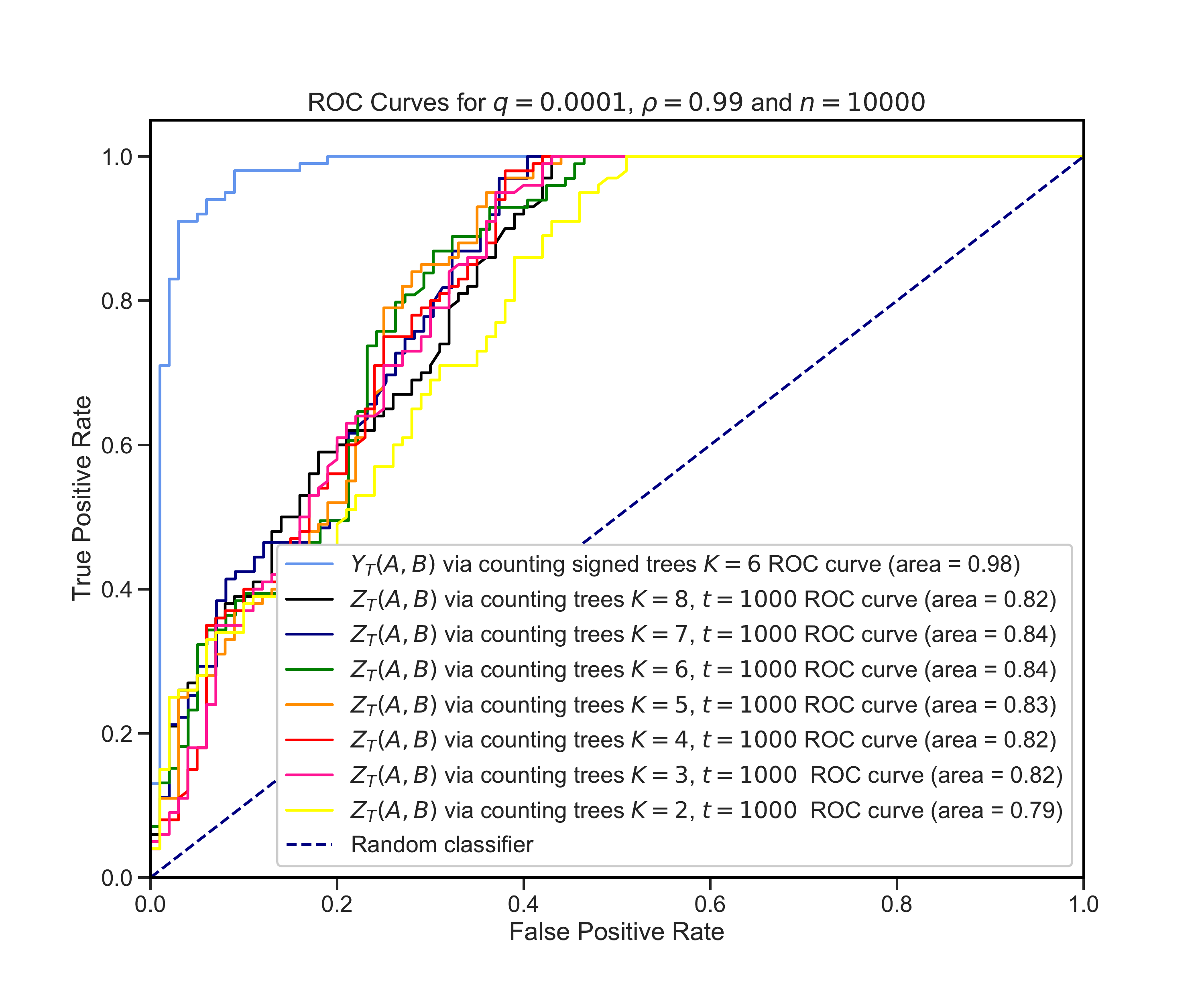} 
    \caption{Comparison of performance of $Y_\calT(A,B)$ when $K=6$, and $Z_\calT(A,B)$ when $K$ ranges from $2$ to $8$, with $n=10000$, $q=0.0001$, $\rho=0.99$, and $t=1000$.
    }
    \label{fig:center_vs_uncenter}
\end{figure}

\section*{Acknowledgment}
     Part of this work was done while the authors were visiting the Simons Institute for the Theory of Computing, participating in the program ``Computational Complexity of Statistical Inference.'' The authors are grateful to Tselil Schramm for suggesting the use of color coding for efficiently counting trees. J.~Xu and S. H.~Yu also would like to thank Louis Hu for suggesting the form of the approximate test statistic by plugging in the averaged subgraph count in~\prettyref{eq:Y_calT_def}.

\appendices

\section{Preliminary facts about graphs}
\label{app:prelim}
		\begin{lemma}
		\label{lmm:O}	
		Let $S,T$ be edge-induced subgraphs of $K_n$.
        \begin{enumerate}[label=(\roman*)]
            \item \label{O:1} 
             Suppose $S \cong T \cong H$ for some graph $H$.
             Let $\pi$ denote the random permutation uniformly distributed over $\calS_n$.
             We have 
            \begin{align*}
                \Prob\left(\pi(S)
                =T\right)
                =\frac{1}{\sub_n(H)}
                = \frac{\aut(H)(n-|V(H)| )!}{n!} \le \left(\frac{|V(H)|}{n}\right)^{|V(H)|}.
            \end{align*} 
            \item \label{O:2}  $V(S \Delta T) \cup \left(V(S) \cap V(T) \right) = V(S)\cup V(T)$, and $ |V(S \Delta T)| + |V(S)\cap V(T)| = |V(S)\cup V(T)|+ |V(S \Delta T)\cap V(S)\cap V(T)|$.
            \item \label{O:3}  
						$ V(S \cup T) = V(S) \cup V(T)$ and $ V(S \cap T) \subset V(S) \cap V(T)$.
            \item \label{O:4}
            Suppose $S$ and $T$ are connected
            and   $|V(S)\cap V(T)|\ge 1$. Then for any $E \subset S \cap T$, 
                \begin{align*}
                     \left|V(S \Delta T) \cup E)\right|
                     &\ge  |V(S)| + |V(T)| - 2 |V(S)\cap V(T)| +|V(E)| + \indc{E=\emptyset, S \neq T}.
                \end{align*}
            \item 
        \label{O:5} For any graphs $H$ and $I$, for any $0\le k \le \min\{|V(H)|,|V(I)|\}$,
        \begin{align*}
            \sum_{S\cong H} \sum_{T \cong I} \indc{ |V(S) \cap V(T)| = k} \le \binom{|V(H)|}{k} \binom{|V(I)|}{k} k! n^{|V(H)|+|V(I)|-k} .
        \end{align*}
        \end{enumerate} 					
		\end{lemma}

        \begin{proof}
       \begin{enumerate}[label=(\roman*)]
            
            \item 						
						Note that $\pi(S)$ is uniformly distributed over the set $\{S' \subset K_n: S' \cong H\}$, whose cardinality is $\sub_n(H)=\sub(H,K_n)$.
						Then $\prob{\pi(S)=T} = \frac{1}{\sub_n(H)}$.
        Applying \prettyref{eq:subH} yields
            \begin{align*}
                \Prob\left(\pi(S)
                =T\right)
                =\frac{\aut(H)(n-|V(H)| )!}{n!} \le \frac{|V(H)|! (n-|V(H)| )!}{n!} \le \left(\frac{|V(H)|}{n}\right)^{|V(H)|},
            \end{align*}
            where the last inequality holds since $\binom{n}{k} \ge \left(\frac{n}{k}\right)^k$.
            \item 
						We remind the reader that in the sequel $S\cup T,S\cap T$ and $S\Delta T$ are edge-induced subgraphs of $K_n$.						
						Since $V(S \Delta T) \subset V(S)\cup V(T)$ and $V(S) \cap V(T) \subset V(S)\cup V(T)$, we have  
            $
            V(S \Delta T) \cup \left(V(S) \cap V(T) \right) \subset V(S)\cup V(T).
            $
            Next, we show 
            $
            V(S)\cup V(T) \subset V(S \Delta T) \cup \left(V(S) \cap V(T) \right).
            $
            It suffices to show that $V(S)\Delta V(T) \subset  V(S \Delta T)$. Indeed, for any $i \in V(S)\cap V(T)^c$, 
						there exists some $j\in V(S)$ such that $(i,j)\in S \cap T^c$.
						Similarly, for any $i \in V(T)\cap V(S)^c$, there exists some $j\in V(T)$ such that $(i,j)\in T \cap S^c$.
						%
						%
            Then, $V(S \Delta T) \cup \left(V(S) \cap V(T) \right) = V(S)\cup V(T)$, which implies that 
            $ 
            |V(S \Delta T)| + |V(S)\cap V(T)| = |V(S)\cup V(T)|+ |V(S \Delta T)\cap \left(V(S)\cap V(T)\right)|.
            $ 
            \item  
            By definition, 
            we have $ V(S \cup T) = V(S) \cup V(T)$, and 
						$ V(S \cap T) \subset V(S) \cap V(T)$ follows because for any $i \in V(S \cap T)$, there exists some $j$ such that $(i,j)\in S \cap T$ and thus $i\in V(S) \cap V(T)$. 
            \item 
						For $E \subset S \cap T$, we have 
                \begin{align*}
                    \left|V\left( (S \Delta T) \cup E\right)\right|
                    = |V(S \Delta T) \cup V(E)| = |V(S \Delta T ) |+ |V(E) |-| V(S \Delta T )\cap V(E)|
                \end{align*}
                where the first equality applies part \ref{O:3}.
								Furthermore, 
								it follows from part \ref{O:2} that $|V(S \Delta T)| =  |V(S)|+|V(T)| - 2 |V(S)\cap V(T)| +|V(S \Delta T)\cap \left(V(S)\cap V(T)\right)|$.
								Thus
								\begin{align*}
                     \left|V\left( (S \Delta T) \cup E\right)\right|
                     & = |V(S)|+ |V(T)| - 2 |V(S)\cap V(T)| + |V(E)| \\
                     &~~~~ + |V(S \Delta T)\cap \left(V(S)\cap V(T)\right)|  -| V(S \Delta T) \cap V(E)| \numberthis \label{eq:VSTE}\\
										& \geq |V(S)|+ |V(T)| - 2 |V(S)\cap V(T)| + |V(E)|,										
                \end{align*}
								where the last step follows from $V(E)\subset V(S) \cap V(T)$
								since $E \subset S \cap T$.
                
  It remains to consider the special case of $E\neq \emptyset$ and $S\neq T$. 
	Continuing \prettyref{eq:VSTE}, it suffices to verify that $|V(S \Delta T)\cap \left(V(S)\cap V(T)\right)|\geq 1$.
	Consider two cases.
	\begin{itemize}
		\item Suppose that $S \cap T \neq \emptyset$. By assumption, $|V(S) \cap V(T)|\ge 1$ and $S, T$ are connected. Thus $S\cup T$ is also connected. Then $ V(S \Delta T)\cap V(S\cap T) \neq \emptyset$, and  
	    $$|V(S \Delta T)\cap \left( V(S)\cap V(T) \right)| \ge |V(S \Delta T)\cap V(S\cap T)|\ge 1 \, , $$
	    where the first inequality holds by part \ref{O:3}.
  
			\item Suppose that $S \cap T = \emptyset$, then $S \cup T = S \Delta T$ and $V(S \Delta T) = V(S)\cup V(T)$ by part \ref{O:3}.
			Thus, 
			$$|V(S \Delta T)\cap \left( V(S)\cap V(T) \right)| =| V(S)\cap V(T) | \ge  1 \, , $$ 
        where the last inequality holds by our standing assumption that $| V(S)\cap V(T) | \ge 1 $. 
  \end{itemize}              
	This concludes the proof of part \ref{O:4}.
	

        \item 
        We have 
                \begin{align*}
                    & \sum_{S \cong H} \sum_{T\cong I} \indc{ |V(S) \cap V(T)| = k}  \\
                    & \le \binom{n}{|V(H)|}\binom{|V(H)|}{k} \binom{n-|V(H)|}{|V(I)|-k} |V(H)|! |V(I)|! \\
                    & = \frac{n!}{(n-|V(H)|)!} \frac{(n-|V(H)|)!}{(n-|V(H)|-|V(I)|+k)!}\binom{|V(H)|}{k}\binom{|V(I)|}{k} k! 
                    \nonumber \\
                    & \le \binom{|V(H)|}{k}\binom{|V(I)|}{k} k! n^{|V(H)|+|V(I)|-k} \, . 
                \end{align*}
        \end{enumerate}
        \end{proof}

\section{Postponed proofs in~\prettyref{sec:stats}}\label{app:stats}

  \subsection{Proof of \prettyref{prop:ER_var_calP}}
		\label{app:ER_var_calP}

  By the discussion in \prettyref{sec:pf-sk-var}, it suffices to bound the two terms \prettyref{eq:ER_varP_termI} and \prettyref{eq:ER_varP_termII}.

 
   \textbf{Bounding term $\termI$.} 
Recall from \prettyref{eq:L_phi_s} that $\left\langle L, \phi_\sfS\right \rangle =a_H$ and by definition $a_H \sub_n(H) = \rho^{|E(H)|} = \rho^{K}$ for all $H \in \calH$. 
Then
\begin{equation}
\Expect_\calP[f_\calH] = |\calH| \rho^{2K}.
\label{eq:EPfH}
\end{equation}
Therefore, applying \prettyref{eq:L_phi_S_T_disjoint} in \prettyref{lmm:L_phi_S_T} yields that 
        \begin{align*}
            \termI
            &  \le \sum_{[H]\in \calH}a_H \sum_{[I]\in \calH}a_I \sum_{S_1\cong H} \sum_{S_2 \cong H} \sum_{T_1\cong I}\sum_{T_2\cong I}a_H a_I \left(\gamma\left(1+\indc{H\cong I}\right) -1\right) \nonumber \\
            & = \sum_{[H]\in \calH} a_H^2 \sum_{[I]\in \calH} a_I^2  \sub_n(H)^2 \sub(I)^2 \left(\gamma\left(1+\indc{H\cong I}\right) -1\right) \nonumber \\
            & = \left(\gamma-1\right) \pth{\sum_{[H]\in \calH}  a_H^2 \sub_n(H)^2 }^2 +  \gamma \sum_{[H]\in \calH } a_H^4  \sub_n(H)^4  \nonumber \\
            & = |\calH|^2 \rho^{4 K}\left( \gamma-1\right) +  |\calH| \rho^{4 K} \gamma \, .
        \end{align*}
        Recall $
        \gamma = \exp\left(\frac{\NumVer^2}{n-2\NumVer+1}\right)$. 
				In view of the definition \prettyref{eq:PhiH}, by choosing $J$ to be a single vertex we have $\Phi_H \le n$ and hence $\Phi_{\calH} \le n$. 
				Thus, by the assumption \prettyref{eq:prop2}, we have $M=o(\log n)$. Thus $\gamma=1+o(1)$. Moreover, by \prettyref{eq:prop1},				
        $|\calH|=\omega(\rho^{-2K})=\omega(1)$.
        It follows that 
        \begin{align}
           \frac{\termI}{\left(\Expect_\calP \left[f_\calH\right]\right)^2} \le \left(\gamma-1\right) + |\calH|^{-1} \gamma = o(1). \label{eq:varp_tree_termI}
        \end{align}
    
  \textbf{Bounding term $\termII$.}    
   For any $S,T\subset \binom{[n]}{2}$, define
            \begin{align}
            \ell(S,T) \triangleq \left(\frac{2\NumVer}{n}\right)^{-|V(S)\cap V(T)|} h(S,T)\, . \label{eq:lst}
            \end{align}
            In the sequel, we will need the following auxiliary result. 
  \begin{lemma} 
        \label{lmm:bk}
        For any $H,I$ in $\calH$ with $K$ edges and at most $\NumVer$ vertices,
        \begin{align*}
             &\sum_{S\cong H} \sum_{T\cong I} \ell(S,T) \indc{|V(S) \cap V(T) | \ge 1} \le  n^{|V(H)|+|V(I)| } b_K ,
        \end{align*}
        where 
       \[
        b_K \triangleq 
           \frac{ 2^{K}(2 \NumVer)^{\frac{\NumVer}{2}} 2^M }{ 
           \sqrt{ \Phi_{\calH} }}  \, ,
        \]
        where $\Phi_\calH = \min_{H \in \calH} \Phi_H$ is defined in \prettyref{prop:ER_var_calP} -- see \prettyref{eq:PhiH}.
        \end{lemma}
        
        \begin{proof}[Proof of \prettyref{lmm:bk}]
    By \prettyref{eq:upper_bound_h}, 
        \begin{align}
            \sum_{S\cong H} \sum_{T\cong I} \ell(S,T)\indc{|V(S) \cap V(T) | \ge 1}
            & \le \frac{ 2^{K}(2 \NumVer)^{\frac{\NumVer}{2}}}{ 
           \sqrt{ \Phi_{\calH} }}  \sum_{S\cong H} \sum_{T\cong I}  \left(\frac{2M}{n}\right)^{-|V(S)\cap V(T)|}  \indc{|V(S) \cap V(T) | \ge 1}. \label{eq:tree_overl_calc}
     \end{align}
    Now, in view of \prettyref{lmm:O}\ref{O:5}, we have that 
            $$
            \sum_{S\cong H} \sum_{T\cong I} \indc{ |V(S) \cap V(T)| = k} \le \binom{|V(H)|}{k} \binom{|V(I)|}{k} k! n^{|V(H)|+|V(I)|-k}  \, ;
            $$
   Therefore, 
   \begin{align*}
           & \sum_{S\cong H} \sum_{T\cong I}  \left(\frac{2M}{n}\right)^{-|V(S)\cap V(T)|}  \indc{|V(S) \cap V(T) | \ge 1}\\
           & \le \sum_{k=1}^{|V(H)|} \sum_{S\cong H} \sum_{T\cong I} \indc{ |V(S) \cap V(T)| = k} \left(\frac{2M}{n}\right)^{- k} \\
           &\le \sum_{k=1}^{|V(H)|} \binom{|V(H)|}{k}  \binom{|V(I)|}{k}  k! n^{|V(H)|+|V(I)|-k}\left(\frac{2M}{n}\right)^{- k} \\
           & \le 2^M n^{|V(H)|+|V(I)|}  \, ,
        \end{align*}
      where the last inequality holds because 
      \begin{align*}
         \sum_{k=1}^{|V(H)|}  \binom{|V(H)|}{k}  \binom{|V(I)|}{k}  k! \left(2M\right)^{- k} & \le \sum_{k=1}^{|V(H)|}  \binom{|V(H)|}{k}  M^k \left(2M\right)^{- k} \\
         & = \sum_{k=1}^{|V(H)|}  \binom{|V(H)|}{k} 2^{-k} \le 2^M \,.
      \end{align*}
      The proof is completed by combining the last display with~\prettyref{eq:tree_overl_calc}. 
    \end{proof}   

\bigskip            
  Now, let us return to the proof of~\prettyref{prop:ER_var_calP}. Fix any $H,I \in \calH$. Fix $\sfS=(S_1,S_2)$ and $\sfT=(T_1,T_2)$ such that $S_1\cong S_2 \cong H$ and $T_1\cong T_2 \cong I$, and $|V(S_1\cap T_1)| \ge 1, |V(S_2\cap T_2)| \ge 1$.
      Applying \prettyref{eq:L_phi_S_T_not_disjoint} in \prettyref{lmm:L_phi_S_T} yields that 
    \begin{align}
            & \left\langle L, \phi_\sfS \phi_\sfT \right \rangle  \indc{\left\langle L, \phi_\sfS \phi_\sfT \right \rangle  \ge 0} \nonumber \\
            & \le 
            \left(\frac{2\NumVer}{n}\right)^{|V(H)|+|V(I)| - |V(S_1)\cap V(T_1)|-|V(S_2)\cap V(T_2)|}\biggl\{\indc{\sfS = \sfT}+ h(S_2,T_2) \indc{S_1=T_1} \nonumber \\
            &~~~~ + h(S_1, T_1) \indc{S_2=T_2} +h(S_1,T_1) h(S_2,T_2) \biggr\} \nonumber \\
            & = 
            \indc{\sfS = \sfT} + \left(\frac{2\NumVer}{n}\right)^{ |V(H)|} \ell(S_1,T_1) \indc{S_2= T_2} + \left(\frac{2\NumVer}{n}\right)^{|V(H)|}\ell(S_2,T_2) \indc{S_1= T_1} + \nonumber \\
            &~~~~  \left(\frac{2\NumVer }{n}\right)^{|V(H)|+|V(I)|} \ell(S_1,T_1) \ell(S_2,T_2) \, . \label{eq:L_phi_S_T_trees}
            \end{align}

   Combining \prettyref{eq:L_phi_S_T_trees} with~\prettyref{lmm:bk} yields that 
    \begin{align}
            & \sum_{S_1\cong H} \sum_{S_2\cong H} \sum_{T_1\cong I}\sum_{T_2\cong I}  \left\langle L, \phi_\sfS \phi_\sfT \right \rangle \indc{\left\langle L, \phi_\sfS \phi_\sfT \right \rangle\ge 0} \cdot \indc{|V(S_1\cap T_1)| \ge 1,\, |V(S_2\cap T_2)| \ge 1} \notag \\
            & \le \sum_{S_1\cong H} \sum_{S_2\cong H} \sum_{T_1\cong I}\sum_{T_2\cong I} \biggl\{ \indc{\sfS = \sfT} +  \left(\frac{2\NumVer}{n}\right)^{ |V(H)|} \ell(S_1,T_1)\indc{S_2= T_2}  + \notag \\
            &~~~~ \left(\frac{2\NumVer}{n}\right)^{|V(H)|}\ell(S_2,T_2) \indc{S_1= T_1}+ \left(\frac{2\NumVer}{n}\right)^{|V(H)|+|V(I)|} \ell(S_1,T_1) \ell(S_2,T_2)\biggr\} \nonumber \\  
            &~~~~ \indc{|V(S_1\cap T_1)| \ge 1,\, |V(S_2\cap T_2)| \ge 1}  \notag \\
            & \le \sub_n(H)^2 \indc{H\cong I} + 2 \sub_n(H)\left(2\NumVer\right)^{\NumVer} n^{|V(H)|} b_K \indc{H\cong I} +  \notag \\
            &~~~~ \left(2\NumVer\right)^{2\NumVer} n^{|V(H)|+|V(I)|} b_K^2 \, . \label{eq:sum-bk}
        \end{align}


       Define
        \begin{align*}
            c_K  & \triangleq  \rho^{-2K} 2^{ 2\NumVer} \NumVer^{4\NumVer} b_K^2 = \frac{ \rho^{-2K}  \NumVer^{5 \NumVer}  2^{ 5\NumVer + 2K+2} }{  \Phi_{\calH} } \, . 
        \end{align*}
		Summing over $H,I\in\calH$, it follows from \prettyref{eq:sum-bk} that 
           \begin{align*}
            \termII 
            & \le  \sum_{[H]\in \calH}a_H \sum_{[I]\in \calH}a_I\biggl\{(\sub_n(H)^2 \indc{H\cong I}  + 2 \sub_n(H)\left(2\NumVer\right)^{\NumVer} n^{|V(H)|} b_K \indc{H\cong I} + \\ 
            &~~~~ \left(2\NumVer\right)^{2\NumVer} n^{|V(H)|+|V(I)|} b_K^2 \biggr\} \\
            &= \sum_{[H]\in \calH}a_H^2 \left(\sub_n(H)^2 + 2 \sub_n(H)\left(2\NumVer\right)^{\NumVer} n^{|V(H)|} b_K \right) +  
						\pth{\sum_{[H]\in \calH} a_H \left(2\NumVer\right)^{\NumVer}n^{|V(H)|} b_K}^2 \\
            &  \overset{(a)}{=}  \sum_{[H]\in \calH} \left( \rho^{2K} + 2 a_H \rho^{K} \left(2\NumVer\right)^{ \NumVer} n^{|V(H)|} b_K \right) + 
						\pth{\sum_{[H]\in \calH} a_H \left(2\NumVer\right)^{\NumVer}n^{|V(H)|} b_K}^2 \\
            &  \overset{(b)}{\le} |\calH|\rho^{2K} \left( 1 + 2 \rho^K \sqrt{c_K} \right) + |\calH|^2\rho^{4 K} c_K , 
            \end{align*}
        where $(a)$ holds as $a_H \sub_n(H) = \rho^{|E(H)|} = \rho^K$; $(b)$ holds by \prettyref{lmm:O}\ref{O:1} as $a_H \le \rho^{K} \left(\frac{M}{n}\right)^{|V(H)|} $ and 
				$|V(H)|\leq M$ for all $H\in\calH$.
        Then, using \prettyref{eq:EPfH}, we have
        \begin{align}
            \frac{\termII}{\left(\Expect_{\calP}[f_{\calH}]\right)^2} 
            & \le \left(|\calH|\rho^{2K}\right)^{-1} \left( 1 + 2 \rho^{K}  \sqrt{c_K} \right) + c_K = o(1) \, , \label{eq:varp_tree_termII} 
        \end{align}
        where the equality holds because $|\calH|\rho^{2K} =\omega(1)$ by \prettyref{eq:prop1} and $c_K = o(1)$ by \prettyref{eq:prop2}. 
        Hence, combining \prettyref{eq:varp_tree_termI} and \prettyref{eq:varp_tree_termII}, 
       we obtain the desired result 
       \[
        \Var_{\calP}[f_{\calH}]  = o\left(\left(\Expect_{\calP}[f_{\calH}]\right)^2\right) \, .
       \]

     Finally, it remains to show that in the special case of trees, namely $\calH=\calT$ and $M=K+1$, the condition~\prettyref{eq:K_condition} ensures that $\calT$ satisfies both \prettyref{eq:prop1} and \prettyref{eq:prop2}. 
     By \prettyref{eq:otterfull}, 
     under the assumption that $\rho^2 > \alpha$ and $K=\omega(1)$, we get \prettyref{eq:prop1}. 
     Moreover, for any tree $H \in \calT$ and any subgraph $J\subset H$, $J$ must be a forest and hence satisfies
		$|V(J)|\geq |E(J)|+1$.
		Thus 
    \begin{align*}
				\Phi_H = \min_{J \subset H, |V(J)| \ge 1} n^{|V(J)|}q^{|E(J)|} 
         =\begin{cases}
         n^{K+1} q^{K} & \text{ if } nq <1 \\
         n  & \text{ if } nq \ge 1
         \end{cases} \, 
     \end{align*}
     and by definition $\Phi_\calT = \Phi_H$.

Under the condition~\prettyref{eq:K_condition}, 
we have $K \le \frac{\log n}{2 \log (\frac{1}{nq})}$ and hence $n^{K+1}q^K \ge \sqrt{n}$. 
Therefore $\Phi_H \ge \sqrt{n}$.
Moreover, \prettyref{eq:K_condition} also implies that $K \le \frac{\log n}{16 \log \log n }$. Since $\rho^2 > \alpha>1/4$, 
the desired \prettyref{eq:prop2} follows from
     \[ 
            \rho^{-2K} (K+1)^{5K+5} 2^{7K+5} \le 2^{9K+5} (K+1)^{5K+5} = o(\sqrt{n}) \, . 
    \]
    \subsection{Proof of \prettyref{lmm:L_phi_S_T}} \label{sec:proof_L_phi_S_T}
    Fix $\sfS=(S_1,S_2)$ and $\sfT=(T_1,T_2)$ such that $S_1 \cong S_2 \cong H$ and $T_1 \cong T_2 \cong I$ where $H$ and $I$ that are connected graphs with $K$ edges and at most $M$ vertices.  
    
				
				
\textbf{Case~(i)}: $V(S_1) \cap V(T_1) = \emptyset$ and $V(S_2) \cap  V(T_2) = \emptyset$.
In this case, we have $S_1 \cap T_1 =\emptyset$ and $S_2 \cap T_2 = \emptyset$. 
        By change of measure and the definition of $\phi_\sfS$ given in~\prettyref{eq:phi_S}, we get that
        \begin{align}
            \left\langle L, \phi_\sfS \phi_\sfT \right \rangle 
            & = 
            \Expect_{\calP}\left[\phi_\sfS(A,B)\phi_\sfT(A,B)\right] \nonumber \\
            & = \Expect_{\pi} \Expect_{\calP|\pi}\left[\sigma^{-2(|E(H)|+|E(I)|)}\prod_{(i,j)\in S_1\cup T_1}\bar{A}_{ij} \prod_{(k,\ell)\in S_2 \cup T_2}\bar{B}_{k,\ell}\right] \nonumber \\
            & = \rho^{|E(H)|+ |E(I)| }\Prob\left(\pi(S_1\cup T_1) = S_2 \cup T_2\right) \nonumber\\
            & = a_H a_I \frac{\Prob\left(\pi(S_1\cup T_1) = S_2 \cup T_2\right) }{\Prob\left(\pi(S_1) = S_2\right)\Prob\left(\pi(T_1) = T_2\right)}. \label{eq:Lemma_1_case_1}
        \end{align}

        Suppose $\pi(S_1\cup T_1) = S_2 \cup T_2$.  
        Since $ V(S_1)\cap V(T_1) = \emptyset$, $ \pi(V(S_1)) \cap  \pi(V(T_1)) = \emptyset$. Since $S_1 \cong S_2 \cong H$ and $T_1 \cong T_2 \cong I$ where $H$ and $I$ that are connected graphs, given $\pi(V(S_1)) \cap  \pi(V(T_1)) = \emptyset$ and $V(S_2) \cap V(T_2) = \emptyset$,
        we must have either
        $\pi(S_1) = S_2$ and $\pi(T_1) = T_2$, or $\pi(S_1) = T_2$ and $\pi(T_2) = S_2 $. 
Hence,
       \begin{align*}
                 \Prob\left(\pi(S_1\cup T_1) = S_2 \cup T_2\right) 
                 & = \Prob\left(\pi(S_1) = S_2,\pi(T_1) = T_2\right)+\Prob\left(\pi(S_1) = T_2,\pi(S_1) = T_2\right) \indc{H \cong I} \\
                 & =  \Prob\left(\pi(S_1) = S_2,\pi(T_1) = T_2\right)\left(1+\indc{H \cong I}\right).
      \end{align*}
    Note that 
      $$ \Prob\left(\pi(S_1) = S_2, \pi(T_1)=T_2\right) = \frac{(n-|V(H)|-|V(I)| )! \aut(H) \aut(I) }{n!} \, .
      $$
      Since $\Prob\left(\pi(S_1) = S_2\right)= \frac{(n-|V(H)|)!\aut(H)}{n!}$ and $\Prob\left(\pi(T_1) = T_2\right) =  \frac{(n-|V(I)|)!\aut(I)}{n!}$, we have
        \begin{align}
          \frac{\Prob\left(\pi(S_1) = S_2, \pi(T_1)=T_2\right)}{\Prob\left(\pi(S_1) = S_2\right) \Prob\left(\pi(T_1) = T_2\right)}  
                & = \frac{n!(n-|V(H)|-|V(I)| )! }{(n-|V(I)|)!(n-|V(H)|)!}\nonumber \\
                & \overset{(a)}{\le} \exp\left( \frac{|V(H)||V(I)|}{n-|V(H)|-|V(I)|+1}\right)\nonumber\\
                & \overset{(b)}{\le}\gamma \, , \label{eq:prob_ratio}
            \end{align}
            where $(a)$ holds because for any $\ell,m \in \naturals$ such that  $\ell+m \le n$, 
            \begin{align*}
              \frac{n!(n-\ell-m)!}{(n-\ell)!(n-m)!} 
                 = \prod_{k=n-\ell+1}^{n} \frac{k}{k-m} \le \left(\frac{n-\ell+1}{n-\ell-m+1}\right)^{\ell}
                \le  \exp\left(\frac{\ell m}{n-\ell-m+1}\right),
            \end{align*}
            where the last inequality holds because for any $x\in \reals$, $1+x \le \exp(x)$; $(b)$ holds because  $|V(H)|,|V(I)| \le M$
            and $\gamma \triangleq \exp\left(\frac{M^2}{n-2M+1}\right)$. Combining~\prettyref{eq:Lemma_1_case_1} and~\prettyref{eq:prob_ratio}, we arrive at the desired bound:
             \begin{align*}
                0\le \left\langle L, \phi_\sfS \phi_\sfT \right \rangle 
                & \le a_H a_I \gamma \left(1+\indc{H\cong I}\right) \, .
            \end{align*}
            
        \textbf{Case~(ii)}: $V(S_1) \cap V(T_1) = \emptyset$ and $V(S_2) \cap  V(T_2) \neq \emptyset$, or $V(S_1) \cap V(T_1) \ne \emptyset$ and $V(S_2) \cap  V(T_2) = \emptyset$.
        While this case itself is easy, we set up some notation for both Case~2 and Case~3 below. 
        By  change of measure, 
        \begin{align}
            & \left\langle L, \phi_\sfS \phi_\sfT\right \rangle =  \Expect_{\calP}\left[\phi_\sfS(A,B)\phi_\sfT(A,B)\right] = \Expect_{\pi} \left[g(\pi)\right], \label{eq:term2_L_st}
        \end{align}
        where 
        \begin{align*}
            g(\pi) 
            & \triangleq \Expect_{\calP|\pi}\left\{\sigma^{-2|E(H) |-2 |E(I)|}\prod_{(i,j)\in S_1 \cap T_1} \bar{A}_{ij}^2 \prod_{(i,j)\in S_1 \Delta T_1}\bar{A}_{ij} \prod_{(k,\ell)\in S_2 \cap T_2 }\bar{B}_{k,\ell}^2 \prod_{(k,\ell)\in S_2 \Delta T_2 }\bar{B}_{k,\ell}\right\}.
        \end{align*}
        
        Given a permutation $\pi$, we define 
        $$K_{11} \triangleq S_1 \Delta T_1 \cap \pi^{-1} \left(  S_2 \Delta T_2 \right),\quad K_{12} \triangleq\left(S_1 \Delta T_1\right)\cap \pi^{-1} \left(S_2 \cap T_2\right) \, , $$ 
        $$K_{21} \triangleq \left(S_1 \cap T_1\right) \cap \pi^{-1}\left(S_2 \Delta T_2\right),\quad K_{22} \triangleq \left(S_1 \cap T_1\right) \cap \pi^{-1} \left(S_2 \cap T_2 \right) \, ,$$
        and 
        $$
        K_{20} \triangleq \left(S_1 \cap T_1\right) \backslash \left(K_{21}\cup K_{22}\right),
        \quad 
        K_{02} \triangleq \left(S_2 \cap T_2\right) \backslash \left(K_{12}\cup K_{22}\right) .
        $$
        In the sequel, we will split the product in the definition of $g(\pi)$ according to the mutually exclusive sets $K_{\ell m}$ defined above, where $0\le \ell, m \le 2$ and $2 \le \ell + m \le 4$. 
        Note that, if there exists $(i,j)\in S_1 \Delta T_1$
        such that $\pi(i,j) \not \in S_2 \cup T_2$, or if there exists $(k,\ell)\in S_2 \Delta T_2 $ such that $(k,\ell) \not \in \pi(S_1 \cup T_1)$, we have $ g(\pi)=0$.  Thus, by independence of the pairs $(A_{ij},B_{\pi(i)\pi(j)})$, we get
        \begin{align}
         g(\pi)
         &=\prod_{\substack{2\le \ell+m \le 4\\ 0\le \ell,m \le 2}}  \sigma^{-(\ell+m)|K_{\ell m}|}\prod_{(i,j)\in K_{\ell m}}
         \Expect_{\calP|\pi}\left[\bar{A}_{ij}^\ell \bar{B}_{\pi(i)\pi(j)}^m \right]      \nonumber \\ 
             &~~~~ \times \indc{S_1 \Delta T_1 =K_{11} \cup K_{12}  }\indc{ \pi^{-1}\left(S_2 \Delta T_2 \right) =K_{11} \cup K_{21} }.
        \label{eq:def_g_pi}
             \end{align}
      Suppose that $|V(S_1)\cap V(T_1)|= 0$ and $|V(S_2)\cap V(T_2)| \ge 1$. 
       Then for any $\pi \in \calS_n$,  $\pi( S_1 \Delta T_1) \not \subset  S_2 \cup T_2 $,
       because the former has $|V(S_1)|+|V(T_1)|$ vertices and the latter has fewer.
       As a result, $g(\pi)=0$ in view of~\prettyref{eq:def_g_pi}. 
       As a result, $\left\langle L, \phi_\sfS \phi_\sfT\right \rangle =0$. 
      The case of $|V(S_2)\cap V(T_2)| = 0$ and $|V(S_1)\cap V(T_1)| \ge 1$ is similar. 
        
   \textbf{Case~(iii)}: $|V(S_1)\cap V(T_1)| \ge 1$ and $|V(S_2)\cap V(T_2)|\ge 1$.
    Recall that in this case we assume $q \le 1/2$.
    We continue to use the notation $g(\pi)$ and $K_{\ell m}$ introduced in Case (ii). 
    Crucially, we need the following lemma, bounding the cross-moments of $\bar{A}_{ij}$ and $\bar{B}_{\pi(i)\pi(j)}$.
    \begin{lemma} \label{lmm:cond-exp}
        Assume $q \le \frac{1}{2}$.
        For any $\pi \in \calS_n$,  $(i,j)\in\binom{[n]}{2}$, and $0 \le \ell, m \le 2$ such that $2 \le \ell + m \le 4$, we have 
        $$
        \left| \Expect_{\calP|\pi}\left[\sigma^{-\ell-m} \bar{A}_{ij}^\ell\bar{B}_{\pi(i)\pi(j)}^m \right] \right|  \le \indc{\ell+m =2} + \sqrt{\frac{1}{q}} \indc{\ell+m=3}  +\frac{1}{q} \indc{\ell+m =4}. 
        $$
       \end{lemma}
       
       \begin{proof}[Proof of \prettyref{lmm:cond-exp}]
       We split the proof into three cases according to the value of $\ell + m$. 
         \begin{itemize}
        \item  If $\ell+m =2$, 
        \begin{align*}
            \Expect_{\calP|\pi}\left[\sigma^{-2}\bar{A}_{ij}^2\right]
            & =\Expect_{\calP|\pi}\left[\sigma^{-2}\bar{B}_{\pi(i)\pi(j)}^2\right]=1,
        \end{align*}
        and 
        \begin{align*}
            \left| \Expect_{\calP|\pi}\left[\sigma^{-2}\bar{A}_{ij}\bar{B}_{\pi(i)\pi(j)}\right] \right| = | \rho| ;
        \end{align*}
        \item  
        If $\ell+m=3$, 
        \begin{align*}
            \Expect_{\calP|\pi}\left[\sigma^{-3}\bar{A}_{ij}^2\bar{B}_{\pi(i)\pi(j)}\right] 
            & = \Expect_{\calP|\pi}\left[\sigma^{-3}\bar{A}_{ij}\bar{B}_{\pi(i)\pi(j)}^2\right]   = \frac{\rho(1-2q)}{\sqrt{q(1-q)}} \, .
        \end{align*}
         Given $q \le \frac{1}{2}$
        and $1-2q\le 1-q$, we have that
        \[
          \left| \Expect_{\calP|\pi}\left[\sigma^{-3}\bar{A}_{ij}^2\bar{B}_{\pi(i)\pi(j)}\right] \right|
          = \left| \Expect_{\calP|\pi}\left[\sigma^{-3}\bar{A}_{ij}\bar{B}_{\pi(i)\pi(j)}^2\right]  \right| \le \sqrt{\frac{1}{q}} \, .
        \]
        
        \item 
        If $\ell+m=4$,  
        we have that
          \begin{align*}
            \Expect_{\calP|\pi}\left[\sigma^{-4}\Bar{A}_{ij}^2 \Bar{B}_{\pi(i)\pi(j)}^2\right] 
            & = \frac{q^2(1-q)^2+\rho q(1-q) (1-2q)^2}{q^2(1-q)^2} \\
            &=  \frac{q(1-q) + \rho(1-2q)^2}{q(1-q)} \, .
        \end{align*}
        Since 
        \[
         -1+q  \le q(1-q) - (1-2q)^2    \le q(1-q)  + \rho  (1-2q)^2  \le q(1-q) + (1-2q)^2 \le 1-q \,  
        \]
        given $q \le \frac{1}{2}$ and $- 1 \le \rho \le 1$,
        we have that
        \begin{align*}
            \left| \Expect_{\calP|\pi}\left[\sigma^{-4}\Bar{A}_{ij}^2 \Bar{B}_{\pi(i)\pi(j)}^2\right] \right|
            & \le \frac{1}{q} \, .
        \end{align*}

        \end{itemize}
        Combining all the cases proves the lemma. 
        \end{proof}
    Applying~\prettyref{lmm:cond-exp} to~\prettyref{eq:def_g_pi}, we get that 
      \[
      g(\pi) \le q^{-\frac{|K_{12}|+|K_{21}|+2|K_{22}|}{2}}\indc{S_1 \Delta T_1 =K_{11} \cup K_{12}  }\indc{ \pi^{-1}\left(S_2 \Delta T_2 \right) =K_{11} \cup K_{21} } \, . 
       \]
      Since 
        $$
            \indc{S_1 \Delta T_1 =K_{11} \cup K_{12} }\indc{ \pi^{-1}\left(S_2 \Delta T_2 \right) =K_{11} \cup K_{21}} \le  \indc{\pi\left(\left(S_1\Delta T_1\right) \cup K_{21} \cup K_{22}\right) = \left(S_2 \Delta T_2 \right) \cup  \pi\left( K_{12} \cup K_{22}\right)} \, .
        $$
        combining it with \prettyref{eq:term2_L_st}, we have
        \begin{align*}       
            \left\langle L, \phi_\sfS \phi_\sfT\right \rangle
                & \le \Expect_{\pi} \left[ q^{-\frac{|K_{12}|+|K_{21}|+2|K_{22}|}{2}} \indc{\pi\left(\left(S_1\Delta T_1\right) \cup K_{21} \cup K_{22}\right) = \left(S_2 \Delta T_2 \right) \cup  \pi\left( K_{12} \cup K_{22}\right)}\right]. 
        \end{align*} 
          Since $ K_{21} \cup K_{22} \subset S_1 \cap T_1$, $\pi\left( K_{12} \cup K_{22}\right) \subset S_2 \cap T_2$ for any $\pi\in \calS_n$, 
        \begin{align}
         \left\langle L, \phi_\sfS \phi_\sfT\right \rangle 
              & \le \Expect_{\pi}\qth{ \sum_{E_1 \subset  S_1 \cap T_1} \sum_{E_2 \subset S_2 \cap T_2}  q^{-\frac{|E_1| + |E_2|}{2}}
              \indc{\pi \left( \left(S_1\Delta T_1 \right)\cup E_1\right) = \left(S_2 \Delta T_2 \right)\cup E_2 }} \nonumber \\
              & \le 
              \sum_{E_1 \subset  S_1 \cap T_1} \sum_{E_2 \subset S_2 \cap T_2} q^{-\frac{|E_1| + |E_2|}{2}} \Prob\left(\pi \left( \left(S_1\Delta T_1 \right)\cup E_1\right) = \left(S_2 \Delta T_2 \right)\cup E_2 \right). \label{eq:L_phi_S_T_not_disjoint_1}
        \end{align}
        Since we assume that $|V(S_1)\cap V(T_1)| \ge 1$, and $|V(S_2)\cap V(T_2)| \ge 1$, by \prettyref{lmm:O}\ref{O:4}, for any  $E_1 \subset S_1 \cap T_1$ and $E_2 \subset S_2 \cap T_2$, 
        \begin{align}
            |V(\left(S_1\Delta T_1\right) \cup E_1)| 
            & \ge  |V(I)|+|V(H)| - 2 |V(S_1)\cap V(T_1)|     \nonumber  \\
            &~~~~ + |V(E_1)| + \indc{E_1 = \emptyset, S_1 \neq T_1}      \label{eq:S_1_Delta_T_1_E_1_0}    \\
            |V(\left(S_2\Delta T_2\right) \cup E_2)|
            & \ge  |V(I)|+|V(H)| - 2 |V(S_2)\cap V(T_2)| \nonumber \\ 
            &~~~~ + |V(E_2)| + \indc{E_2 = \emptyset, S_2 \neq T_2}. \label{eq:S_1_Delta_T_1_E_1} 
        \end{align}
         Then, by \prettyref{lmm:O}\ref{O:1},
                  \begin{align*}
            & \Prob\left(\pi(\left(S_1\Delta T_1\right) \cup E_1) = \left(S_2\Delta T_2\right) \cup E_2\right)\\    
            & \le \left(\frac{\left|V(\left(S_1 \Delta T_1\right) \cup E_1) \right|}{n}\right)^{\left|V(\left(S_1 \Delta T_1\right) \cup E_1) \right|} \indc{\left|V(\left(S_1 \Delta T_1\right) \cup E_1) \right| =\left| V( \left(S_2\Delta T_2\right) \cup E_2 )  \right|}\\
            & \overset{(a)}{\le} \left(\frac{|V(S_1\cup T_1)|}{n}\right)^{  |V(H)|+|V(I)| - |V(S_1)\cap V(T_1)|-|V(S_2)\cap V(T_2)| + \frac{1}{2}\left(|V(E_1)|+|V(E_2)|+\indc{E_1 = \emptyset,S_1\neq T_1}+\indc{E_2 = \emptyset, S_2 \neq T_2}\right)}\\
            & \overset{(b)}{\le} \left(\frac{2M }{n}\right)^{  |V(H)|+|V(I)| - |V(S_1)\cap V(T_1)|-|V(S_2)\cap V(T_2)| + \frac{1}{2}\left(|V(E_1)|+|V(E_2)|+\indc{E_1 = \emptyset,S_1\neq T_1}+\indc{E_2 = \emptyset, S_2 \neq T_2}\right)} ,
            \end{align*}
            where $(a)$ applies both \prettyref{eq:S_1_Delta_T_1_E_1_0} and \prettyref{eq:S_1_Delta_T_1_E_1} and $\left|V(\left(S_1 \Delta T_1\right) \cup E_1) \right| =\left| V( \left(S_2\Delta T_2\right) \cup E_2 )  \right| \le  \left|V(S_1\cup T_1) \right|$, as $E_1 \subset S_1 \cap T_1$ and $\left(S_1 \Delta T_1\right) \cup E_1 \subset S_1 \cup T_1$; $(b)$ holds 
            because 
            $|V(S_1)|,|V(T_1)|\le M$. 
            
            

            Thus, combining the last displayed inequality with \prettyref{eq:L_phi_S_T_not_disjoint_1} yields that 
            \begin{align*}
                & \left \langle L, \phi_\sfS \phi_\sfT \right \rangle \\
                & \le  \sum_{E_1 \subset  S_1 \cap T_1} \sum_{E_2 \subset S_2 \cap T_2} q^{-\frac{1}{2}\left(|E_1| + |E_2|\right)} \\
                &~~~~  \left(\frac{2M}{n}\right)^{  |V(H)|+|V(I)| - |V(S_1)\cap V(T_1)|-|V(S_2)\cap V(T_2)| + \frac{1}{2}\left(|V(E_1)|+|V(E_2)|+\indc{E_1 = \emptyset,S_1\neq T_1}+\indc{E_2 = \emptyset, S_2 \neq T_2}\right)}\\
                & \le  \left(\frac{2M}{n}\right)^{  |V(H)|+|V(I)|- |V(S_1)\cap V(T_1)|-|V(S_2)\cap V(T_2)|} \left(\indc{S_1=T_1} +h(S_1,T_1)\right) \left( \indc{S_2 =T_2} + h(S_2,T_2)\right) \, ,
            \end{align*} 
           where for any $S,T\subset \binom{[n]}{2}$,
           $h(S,T)$ are defined according to~\prettyref{eq:h_S_T}, namely
           $$
                h(S,T) \triangleq \sum_{E \subset  S \cap T} q^{-\frac{1}{2}|E|} \left(\frac{2M}{n}\right)^{ \frac{1}{2}\left(|V(E)|+\indc{E = \emptyset}\right)},  
           $$ 
           and the last inequality holds because for any $S,T$, 
         $$
          \sum_{E \subset  S \cap T} q^{-\frac{1}{2}|E|} \left(\frac{2M}{n}\right)^{ \frac{1}{2}\left(|V(E)|+ \indc{E= \emptyset, S\neq T} \right)} \le \begin{cases}
                  1 + h(S, T)
                  & \text{ if } S = T \\
                  h(S, T)
                 & \text{ if } S \neq T 
          \end{cases} \, .
          $$
          
          Finally, it remains to verify~\prettyref{eq:upper_bound_h}. 
          In particular, 
           \begin{align*}
            h(S,T)
          & \le \frac{ 2^{K} (2\NumVer)^{\frac{\NumVer}{2}}}{ 
          \sqrt{\min_{E \subset S\cap T}  n^{|V(E)|+ \indc{E=\emptyset}} q^{|E| }} }  \le  \frac{ 2^{K} (2\NumVer)^{\frac{\NumVer}{2}}}{ 
          \sqrt{\min_{J \subset S: |V(J)| \ge 1}  n^{|V(J)|} q^{|E(J)| }} }  = \frac{ 2^{K} (2\NumVer)^{\frac{\NumVer}{2}}}{\sqrt{\Phi_S}}
          \, ,
        \end{align*}
        where the first inequality holds because there are at most $2^K$ different choices of edge set $E \subset S\cap T$;
        the second inequality follows because
        \[
\min_{E \subset S\cap T}  n^{|V(E)|+ \indc{E=\emptyset}} q^{|E| }
= 
\min\sth{n,
\min_{E \subset S\cap T, E\neq \emptyset}  n^{|V(E)|} q^{|E| }}
\geq 
\min_{J \subset S, |V(J)|\geq 1}  n^{|V(J)|} q^{|E(J)| }.
\]
with the last minimum over all subgraphs $J$ with at least one vertex. Note that here we identify $S$ with its edge-induced subgraph of $K_n$. By symmetry we also have $
            h(S,T) \le\frac{ 2^{K} (2\NumVer)^{\frac{\NumVer}{2}}}{\sqrt{\Phi_T}}$ 
            and \prettyref{eq:upper_bound_h} follows.

 \section{Postponed proofs in~\prettyref{sec:computation_color_coding}}  \label{app:algorithm}
 

  \subsection{Proof of \prettyref{prop:er_Y_calT_testing_error}}

   To show \prettyref{eq:tildef-converge}, it is equivalent to show
  that under both $\calP$ and $\calQ$,
     \begin{align*}
       \frac{\beta Y_\calT - r^2 f_{\calT}}{r^2 \mu} \xrightarrow{L_2} 0, 
       \quad \text{ where }\mu \triangleq \expects{f_{\calT}}{\calP} 
       \, .
   \end{align*}
   By definition, 
        \begin{align}
        \beta Y_{\calT} (A,B)
        & = \frac{1}{t^2} \sum_{i=1}^t \sum_{j=1}^t   X_\calT^{ij}(A,B)  \, , \label{eq:Y_calT}
   \end{align} 
    where for any $1\le i, j \le t$, 
    \begin{align}
        X_\calT^{ij}(A,B)  
        & \triangleq \sum_{[H]\in \calT} \beta \aut(H)
        X_H (\bar{A}, \mu_i ) X_H (\bar{B}, \nu_j) \nonumber \\
        & = \sum_{[H]\in \calT} a_H \sigma^{-2|H|}
        X_H (\bar{A}, \mu_i ) X_H (\bar{B}, \nu_j) \, , \label{eq:X_calT_ij} 
    \end{align} 
    where the second equality holds by $ \beta \aut(H) =a_H \sigma^{-2|H|} $ for each $H \in \calT$, in view of   \prettyref{eq:beta}, \prettyref{eq:a_H}, and \prettyref{lmm:O}\ref{O:1} in \prettyref{app:prelim}. 
    
    Note that each $X_\calT^{ij}(A,B)/r^2$ is an unbiased estimator of $f_\calT (A,B)$, as
     \begin{align}
       \expect{X_{\calT}^{ij} (A,B)  \mid A,B }
       & =  r^2 \sum_{[H]\in \calT} a_H  \sigma^{-2|H|} W_H(\bar{A}) W_H(\bar{B}) = r^2 f_\calT (A,B) \, . \label{eq:E_X_calH_A_B}
   \end{align} 
    Moreover, 
    $\{X_\calT^{ij}(A,B)\}_{1\le i, j \le t}$ are identically distributed. And conditional on $A$ and $B$, for any $1\le i\le j'\le t$,  $X_\calT^{ij}(A,B)$ and $ X_\calT^{i'j'}(A,B)$ are independent if and only if $ i' \neq i$ and $j' \neq j$. 
It follows that 
$
\expect{\beta Y_\calT(A,B) \mid A, B} = r^2 f_{\calT} (A, B)
$.
Thus $\expect{\beta Y_{\calT}-r^2 f_{\calT}}=0$ under both $\calP$ and $\calQ$.
  
  Next, we bound the variance of $\beta Y_\calT-r^2 f_{\calT}$. 
  In particular, by the law of total variance, under both $\calP$ and $\calQ$,  we get that
   \begin{align}
       \Var\left(\beta Y_\calT - r^2 f_{\calT} \right) 
       & = \Var \left(
       \expect{\beta  Y_\calT -r^2 f_{\calT} \mid A, B} \right)
       + \Expect\left[ \Var\left( \beta  Y_\calT -r^2 f_{\calT} \mid A, B \right)\right] \notag \\
       & \overset{(a)}{=} \Expect\left[ \Var\left(\beta  Y_\calT \mid A, B \right)\right] \nonumber \\
       & = 
       \frac{1}{t^4} \sum_{i=1}^t\sum_{j=1}^t \sum_{i'=1}^t\sum_{j'=1}^t  \Expect\left[ \Cov\left( X_\calT^{ij}(A,B),  X_\calT^{i'j'}(A,B)  \mid A, B \right)\right], \label{eq:total_var}
   \end{align}
   where $(a)$ holds because
   $
\expect{\beta  Y_\calT(A,B) \mid A, B} = r^2 f_{\calT} (A, B)
$ and $\Var\left( \beta  Y_\calT -r^2 f_{\calT} \mid A, B \right)=\Var\left(\beta   Y_\calT \mid A, B\right)$.

   Next, we introduce an auxiliary result, bounding 
the conditional covariance.
   
   \begin{lemma}\label{lmm:er_expect_var_A_B}
    Suppose \prettyref{eq:K_condition} holds and $q \le \frac{1}{2}$. 
    Then under both $\calP$ and $\calQ$,
    for any $0\le i,i',j,j'  \le t$,
    \begin{align*}
          \Expect\left[ \Cov \left( X_{\calT}^{ij}  (A,B), X_{\calT}^{i'j'}(A,B)  \mid A,B\right) \right] 
          \le  \xi  \left( r^{2+\indc{i\neq i'}+\indc{j\neq j'}}   -r^4\right)\mu^2 \, ,
    \end{align*}
    for  some $\xi>0$ and $\xi=o(1)$.
    \end{lemma}
   
        With~\prettyref{lmm:er_expect_var_A_B}, we first finish the proof of~\prettyref{prop:er_Y_calT_testing_error}.
        Combining this with \prettyref{eq:total_var}, 
        and \prettyref{lmm:er_expect_var_A_B}, 
        under both $\calP$ and $\calQ$,
         we get that 
        \begin{align*}
             \Var\left(\beta Y_\calT -r^2 f_\calT \right) 
             & = 
             \frac{1}{t^4} \sum_{i=1}^t\sum_{j=1}^t \sum_{i'=1}^t\sum_{j'=1}^t   o\left( \left( r^{2+\indc{i\neq i'}+\indc{j\neq j'}}   -r^4\right)\mu^2\right)\\
             & =
             o\left(\left(\frac{1}{t^2} r^2 + \frac{2}{t} r^3  \right) \mu^2 \right) = o\left(r^4\mu^2 \right)  \, ,
        \end{align*}
        where the last inequality holds due to $t=\lceil 1/r \rceil$.
        Therefore, $ \frac{\beta Y_\calT - r^2 f_\calT}{r^2 \mu}$ converges to $0$ in the $L_2$ norm.

\bigskip
Finally, we are left to prove 
\prettyref{lmm:er_expect_var_A_B}.
          \begin{proof}[Proof of \prettyref{lmm:er_expect_var_A_B}]
    First, we have
          \begin{align*}
            & \Cov\left( X_{\calT}^{ij}  (A,B), X_{\calT}^{i'j'}(A,B)  \mid A,B\right) \\
            &=  \expect{X_{\calT}^{ij}  (A,B) X_{\calT}^{i'j'}(A,B)  \mid A,B } -  \expect{X_{\calT}^{ij}  (A,B) \mid A,B } \expect{X_{\calT}^{i'j'}  (A,B) \mid A,B } \\
            & =   \expect{X_{\calT}^{ij}  (A,B) X_{\calT}^{i'j'}(A,B)  \mid A,B } - r^4 f_\calT (A,B)^2, 
       \end{align*}
       where the last equality holds by \prettyref{eq:E_X_calH_A_B}.
       By \prettyref{eq:X_calT_ij}, 
       \begin{align}
            \expect{X_{\calT}^{ij}  (A,B) X_{\calT}^{i'j'}(A,B) \mid A,B  } 
            & = \sum_{[H]\in \calT} \sum_{[I] \in \calT} a_H a_I \sigma^{-2|H|-2|I|} \sum_{S_1\cong H} \sum_{S_2 \cong H} \sum_{T_1 \cong I} \sum_{T_2\cong I} \nonumber \\
            &~~~~ \expect{\chi_{\mu_i}(V(S_1)) \chi_{\mu_{i'}}(V(T_1))}\expect{ \chi_{\nu_j} (V(S_2))\chi_{\nu_{j'}}(V(T_2))} \nonumber \\
            &~~~~ \prod_{(k,\ell)\in S_1} \overline{A}_{k\ell} \prod_{(k,\ell)\in S_2} \overline{B}_{k\ell} \prod_{(k,\ell)\in T_1} \overline{A}_{k\ell} \prod_{(k,\ell)\in T_2} \overline{B}_{k\ell } \nonumber \\
            & = \sum_{[H],[I]\in \calT} a_H a_I \sum_{S_1\cong H}\sum_{T_1\cong I}  \sum_{S_2\cong H}\sum_{T_2 \cong I}  \nonumber \\
           &~~~~ \expect{\chi_{\mu_i}(V(S_1)) \chi_{\mu_{i'}}(V(T_1))}\expect{ \chi_{\nu_j} (V(S_2))\chi_{\nu_{j'}}(V(T_2))} \phi_\sfS   \phi_\sfT \, , \label{eq:expect_condition_X_A_B} 
       \end{align}
    where in the last equality $\sfS=(S_1,S_2), \sfT=(T_1,T_2)$ and $\phi_{\sfS}$ is defined in \prettyref{eq:phi_S}.

  
     \begin{itemize}
         \item Under $\calQ$, 
         we have
        \begin{align*}
               \Expect_{\calQ} \left[ X_{\calT}^{ij}  (A,B) X_{\calT}^{i'j'}(A,B) \right]
               & \overset{(a)}{=}\sum_{[H],[I]\in \calT} a_H a_I \sum_{S_1 \cong H}\sum_{T_1\cong I} \sum_{S_2 \cong H}\sum_{T_2 \cong I} \nonumber \\
               &~~~~   \expect{\chi_{\mu_i}(V(S_1)) \chi_{\mu_{i'}}(V(T_1))}\expect{ \chi_{\nu_j}(V(S_2)) \chi_{\nu_{j'}}(V(T_2))} \indc{\sfS = \sfT } \\
               & \overset{(b)}{=} r^{2+\indc{i\neq i'}+\indc{j\neq j'}} \sum_{[H]\in \calT} a_H^2 \sub_n(H)^2\\
               & \overset{(c)}{=} r^{2+\indc{i\neq i'}+\indc{j\neq j'}} \Expect_\calQ \left[f_\calT(A,B)^2 \right], 
         \end{align*}
         where $(a)$ holds because $\left\{\phi_{\sfS}\right\}_{S_1,S_2 \subset\binom{[n]}{2}}$ are orthonormal; $(b)$ holds because when 
         $\sfS = \sfT$, given $\chi_{\mu_i}(V(S_1)) , \chi_{\mu_{i'}}(V(T_1)) $ are independent for $i\neq i'$ and $ \chi_{\nu_j}(V(S_2)) , \chi_{\nu_{j'}}(V(T_2))$ are independent for $j\neq j'$, we have that 
         \begin{align*}
             \expect{\chi_{\mu_i}(V(S_1)) \chi_{\mu_{i'}}(V(T_1))}\expect{ \chi_{\nu_j}(V(S_2)) \chi_{\nu_{j'}}(V(T_2))}
             =  r^{2+\indc{i\neq i'}+\indc{j\neq j'}} \, ;
         \end{align*}
         $(c)$ holds by the definition of $a_H$
         given in \prettyref{eq:a_H} and \prettyref{prop:mean_p_q_var_q}.
         Hence, 
         \begin{align*}
           \Expect_\calQ\left[ \Cov\left( X_{\calT}^{ij}  (A,B), X_{\calT}^{i'j'}(A,B)  \mid A,B\right) \right]
           & =  \left( r^{2+\indc{i\neq i'}+\indc{j\neq j'}} - r^4 \right) \Expect_\calQ \left[f_\calT(A,B)^2 \right] \\
           & =\xi_1  \left( r^{2+\indc{i\neq i'}+\indc{j\neq j'}}   -r^4\right)\mu^2
       \end{align*} 
       for some $\xi_1>0$ and $\xi_1=o(1)$,
        where the last equality holds because  $\Expect_\calQ \left[f_\calT(A,B)^2 \right] = \Expect_\calP \left[f_\calT(A,B) \right] =\mu$
        in view of~\prettyref{prop:mean_p_q_var_q}
        and $\mu=\omega(1)$ in view of  \prettyref{prop:ER_var_calP}, given \prettyref{eq:K_condition} holds. 
   \item Under $\calP$, we have 
  \begin{align*}
      &\Expect_{\calP} \left[ \Cov \left( X_{\calT}^{ij}  (A,B) , X_{\calT}^{i'j'}(A,B) | A,B \right) \right] \\
      & \le \sum_{[H],[I]\in \calT} a_H a_{I} \sum_{S_1 \cong H} \sum_{T_1 \cong I} \sum_{S_2 \cong H}\sum_{T_2 \cong I}\\
      &~~~~  \left( \expect{\chi_{\mu_i} (V(S_1)) \chi_{\mu_{i'}}(V(T_1))}  \expect{\chi_{\nu_j}(V(S_2)) \chi_{\nu_{j'}}(V(T_2))} - r^4\right) \left\langle L, \phi_\sfS \phi_\sfT \right \rangle \indc{\left\langle L, \phi_\sfS \phi_\sfT \right \rangle  \ge 0} \, ,
  \end{align*}
  where the last equality holds by \prettyref{eq:f_calH_1} 
  and \prettyref{eq:expect_condition_X_A_B}. 
  Fix some $S_1, T_1, S_2, T_2$  such that $S_1 \cong H$, $T_1 \cong H$, $S_2 \cong I$ and $T_2 \cong I$ for some $[H], [I] \in \calT$. 
  If $V(S_1)\cap V(T_1) =\emptyset$ and $V(S_2)\cap V(T_2) =\emptyset$, 
  we have
  $$
    \expect{\chi_{\mu_i}(V(S_1)) \chi_{\mu_{i'}}(V(T_1))} = \expect{\chi_{\nu_j}(V(S_2)) \chi_{\nu_{j'}}(V(T_2))} = r^2;
  $$
  If $V(S_1)\cap V(T_1) \neq \emptyset$ or $V(S_2)\cap V(T_2) \neq \emptyset$, 
    we have 
    $$
    \expect{\chi_{\mu_{i}}(V(S_1)) \chi_{\mu_{i'}}(V(T_1))} \le r^{1+\indc{i \neq i' }}, \quad \expect{\chi_{\nu_j}(V(S_2)) \chi_{\nu_{j'}}(V(T_2))}\le r^{1+\indc{j \neq j' }}.
  $$
  Therefore, we have
  \begin{align*}
        & \Expect_\calP\left[  \Cov{\left[ X_{\calT}^{ij} (A,B), X_{\calT}^{i'j'} (A,B) \mid A,B\right] }\right] \\
        & \le  \left( r^{2+ \indc{i \neq i' }  + \indc{j \neq j'} } -r^4\right) \sum_{[H]\in \calT}a_H \sum_{[I]\in \calT} a_I\sum_{S_1\cong H} \sum_{S_2\cong H} \sum_{T_1 \cong I}\sum_{T_2\cong I}  
       \\
        &~~~~ 
        \left\langle L, \phi_\sfS \phi_\sfT \right \rangle \indc{\left\langle L, \phi_\sfS \phi_\sfT \right \rangle \ge 0} \indc{ V(S_1) \cap V(T_1)\neq \emptyset \text{ or } V(S_2) \cap V(T_2) \neq \emptyset}\\
        &  \overset{(a)}{=} 
        \left( r^{2+ \indc{i \neq i' }  + \indc{j \neq j'} } -r^4\right) \sum_{[H]\in \calT}a_H \sum_{[I]\in \calT} a_I\sum_{S_1\cong H} \sum_{S_2\cong H} \sum_{T_1 \cong I}\sum_{T_2\cong I}  
        \\
        &~~~~ \left\langle L, \phi_\sfS \phi_\sfT \right \rangle \indc{\left\langle L, \phi_\sfS \phi_\sfT \right \rangle \ge 0} \indc{ V(S_1) \cap V(T_1)\neq \emptyset \text{ and } V(S_2) \cap V(T_2) \neq \emptyset}\\
        & \overset{(b)}{\le} \xi_2 \left( r^{2+\indc{i\neq i'}+\indc{j\neq j'}} -r^4\right) \mu^2  \, 
  \end{align*}
  for some $\xi_2>0$ and $\xi_2=o(1)$,
 where $(a)$ follows from~\prettyref{eq:L_phi_S_T_not_disjoint-1}; $(b)$ holds by \prettyref{eq:ER_varP_termII} and \prettyref{eq:varp_tree_termII}. 
 \end{itemize}
Choosing $\xi=\min\{\xi_1,\xi_2\}$, our desired result follows.  
     \end{proof}

\subsection{Proof of \prettyref{lmm:X_H_M}}\label{app:proof_lemma_XHM}
\begin{proof}[Proof of \prettyref{lmm:X_H_M}]
        We first bound the total time complexity of \prettyref{alg:X_M_H}.
            The run time for the DFS
        is $O(K)$. 
        The total number of subsets $C\subset[K+1]$  with $|C| = k+1$ is $\binom{K+1}{k+1}$. Fixing a color set $C$ with $|C| = k+1$, the total number of pairs of $(C_1,C_2) \in \sfC(C)$ is $2^{k+1}$. Thus according to~\prettyref{eq:X_recursive}, 
        the total time complexity of computing 
        $Y(x,T_i,C,\mu)$
        for all $x \in [n]$ and 
        all color set $C$ with $|C| = |V(T_i)|=k$
        is $O\left( \binom{K+1}{k+1} 2^{k+1} n^2\right)$. 
        Thus, the total time complexity of \prettyref{alg:X_M_H} is bounded by
        $$ 
        O \left(K + \sum_{i=1}^{K}  \binom{K+1}{|V(T_i)|} 2^{|V(T_i)|} n^2 \right) \le
        O \left(K + \sum_{k=1}^{K} K \binom{K+1}{k+1} 2^{k +1} n^2\right) = O(K 3^{K+1} n^2),
        $$ 
        where the first inequality holds because
        the total number of $i \in [K]$ such that $|V(T_i)|=k$ is at most $K$.
                
                
        Next, we prove the correctness of \prettyref{alg:X_M_H}. For any $V\subset [n]$, $\chi_\mu (V, C) $ is the indicator for the event that $\mu(V)$ is colorful and $\{\mu(x)\}_{x\in V} = C$. 
        For any $x\in [n]$,  any tree $T_0$ with a single vertex $u$, and any color set $ C \subset [K+1]$, define
         \begin{align}
             X(x, T_0, C,\mu)   
             \triangleq \sum_{\phi: \{u\} \to [n]} \chi_\mu (\phi(u), C) \times \indc{\phi(u) = x} 
             =  \chi_\mu (x, C) = \indc{\{\mu(x)\} = C}\, . \label{eq:X_M_T_u_C_single}
        \end{align}
        Moreover, for any $1\le i\le K$ and tree $T_i$ with root $p_i$,  define
        \begin{align}
            X (x,T_i, C, \mu )
            & \triangleq \sum_{\phi: V(T_i) \to [n]}  \chi_\mu (\phi(V(T_i)), C ) \times \indc{\phi(p_i) = x}\times \prod_{(i,j)\in E[\phi(T_i)]} M_{ij}  \, . \label{eq:X_M_T_u_C}
        \end{align}
       
         
        Note that by definition, we have $X(x,T_0, C,\mu )  =  Y(x,T_0, C,\mu)$. 
        We proceed to show that   
        $X(x,T_i, C,\mu )  =  Y(x,T_i, C,\mu)$ for all $1 \le i \le K$.
        Recall that by removing edge $e_i$ in $T_i$, we get two rooted trees $T_{a_i}$ and $T_{b_i}$, where $T_{a_i}$ is rooted at $ p_i$ and $T_{b_i}$ is rooted at $ c_i$. 
        For any mapping $\phi: V(T_i)\to [K+1]$, let $\phi_1$ (resp.\ $\phi_2$) denote $\phi$ restricted to $V(T_{a_i}) $ (resp.\ $V(T_{b_i})$).
       Then we have
        \[
        \chi_\mu(\phi(V(T_i)), C )=
          \sum_{(C_1,C_2) \in \sfC(C)}\chi_\mu (\phi_1(V(T_{a_i})),C_1 ) \times \chi_\mu (\phi_2(V(T_{b_i})), C_2) \, .
        \]
        Hence, by \prettyref{eq:X_M_T_u_C}, 
        \begin{align}
            & X(x,T_i,C,\mu)  \nonumber \\
            & = \sum_{(C_1,C_2) \in \sfC } \left( \sum_{\phi_1: V(T_{a_i}) \to [n]} \chi_{\mu}(\phi_1 (V(T_{a_i})), C_1) \times \indc{ \phi_1(p_i) = x} \times
            \prod_{(i,j)\in E[\phi_1(T_{a_i})]} M_{ij} \right) \nonumber \\
            &~~~~ \times \left( \sum_{\phi_2: V(T_{b_i}) \to [n]} \chi_{\mu}(\phi_2 (V(T_{b_i})), C_2) \times \prod_{(i,j)\in E[\phi_2(T_{b_i})]} M_{ij}  \times M_{x \phi_2(c_i)} \right)\nonumber \\
            & =  \sum_{(C_1,C_2) \in \sfC } \left( \sum_{\phi_1: V(T_{a_i}) \to [n]} \chi_{\mu}(\phi_1 (V(T_{a_i})), C_1) \times \indc{ \phi_1(p_i) = x} \times
            \prod_{(i,j)\in E[\phi_1(T_{a_i})]} M_{ij} \right) \nonumber \\
            &~~~~ \times \sum_{y \in [n]\backslash \{x\}} \left( \sum_{\phi_2: V(T_{b_i}) \to [n]} \chi_{\mu}(\phi_2 (V(T_{b_i})), C_2) \times \indc{\phi_2(c_i) = y} \times \prod_{(i,j)\in E[\phi_2(T_{a_i})]} M_{ij} \right) \times M_{xy} \nonumber \\
            & = \sum_{y \in [n]\backslash \{x\}}\sum_{(C_1,C_2) \in \sfC (C)} X(x,T_{a_i}, C_1, \mu) \times X(y,T_{b_i}, C_2,\mu) \times M_{xy} \, ,
            \label{eq:x_T_recursive} 
        \end{align}
        where the last equality holds by the fact that $T_{a_i}$ is rooted at $p_i$ and 
        $T_{b_i}$ is rooted at $c_i$, \prettyref{eq:X_M_T_u_C_single},
        and \prettyref{eq:X_M_T_u_C}.   
        Hence, by \prettyref{eq:X_recursive}, \prettyref{eq:x_T_0}, \prettyref{eq:x_T_recursive}, and \prettyref{eq:X_M_T_u_C_single}, it follows that for any $0\le i\leq K$, 
        \begin{align}
                     X(x,T_i, C,\mu )  =  Y(x,T_i, C,\mu)   \,. \label{eq:X_Y_Equivalent}
        \end{align}
        In particular, we get that 
        $X(x,T_K, [K+1],\mu)  =  Y(x,T_K,[K+1],\mu)$. Thus to prove the correctness of \prettyref{alg:X_M_H}, it remains to check that $\sum_{x\in [n]} X(x,T_K,[K+1],\mu)=\aut(H) X_H(M,\mu)$.
          By \prettyref{eq:X_M_T_u_C}, we have
        \begin{align}
            & \sum_{x\in [n]} X(x,T_K,[K+1],\mu) \nonumber \\
            & = \sum_{x\in [n]}  \sum_{\phi: V(T_K) \to [n]}  \chi_\mu (\phi(V(T_K)), [K+1]) \times \indc{\phi(p_K) = x}\times \prod_{(i,j)\in E[\phi(T_K)]} M_{ij} \nonumber \\
            & =  \sum_{\phi: V(T_K) \to [n]}  \chi_{\mu} (\phi(V(T_K)), [K+1]) \times \prod_{(i,j)\in E[\phi(T_K)]} M_{ij} \nonumber \\
            & \overset{(a)}{=}  \sum_{\phi: V(H) \to [n]}  \chi_{\mu} (\phi(V(H)), [K+1]) \times \prod_{(i,j)\in E[\phi(H)]} M_{ij} \nonumber \\
            & \overset{(b)}{=} \aut(H)  \sum_{S:S\cong H} \chi_{\mu}(V(S)) \prod_{(i,j)\in S} M_{ij} \nonumber \\
            & \overset{(c)}{=} \aut(H) X_H(M,\mu)\,,  \label{eq:claim2}
        \end{align}
        where $(a)$ holds because $T_K $ is tree $H$ rooted at node $p_K$;  $(b)$ holds because for any $S \subset \binom{[n]}{2}$ such that $S\cong H$,
        there are $\aut(H)$ different mapping $\phi: V(H) \to V(S)$ such that $\phi(H)=S$, and $\chi_{\mu}(V(S)) = \chi_{\mu}(V(S), [K+1])$ by definition; $(c)$ holds by \prettyref{eq:X_H_M}. 
        \end{proof}

\bibliography{low_degree_ref}
\bibliographystyle{alpha}

\end{document}